\documentclass[12pt]{amsart}
\usepackage{amssymb}
\usepackage{amscd}
\usepackage{amsfonts}
\usepackage{latexsym}
\usepackage[all]{xy}
\usepackage{verbatim}
\usepackage{longtable}
\title{Classification of finite dimensional irreducible
modules over W-algebras}
\usepackage{amssymb}
\author{Ivan Losev and Victor Ostrik}
\address{I.L: Department
of Mathematics, Northeastern University, Boston MA 02115 USA;
V.O.: Department of Mathematics, University of Oregon, Eugene OR 97403 USA}
\email{i.loseu@neu.edu; vostrik@darkwing.uoregon.edu}
\thanks{I.L. is supported by the NSF grant DMS-0900907; V.O. is supported by the NSF grant
DMS-0602263}
\newcommand\g{{\mathfrak g}}
\newcommand\bfr{\mathfrak{b}}
\newcommand\h{{\mathfrak h}}
\newcommand{\q}{\mathfrak{q}}

\newcommand{\Mod}{\operatorname{Mod}}
\newcommand{\Sk}{\mathcal{S}}

\newcommand\m{\mathfrak m}
\newcommand\Coh{\operatorname{Coh}}
\newcommand\lf{\mathfrak l}

\newcommand\z{\mathfrak z}

\renewcommand\t{\mathfrak t}

\newcommand\codim{\operatorname{codim}}
\newcommand\Spec{\operatorname{Spec}}

\newcommand\F{\operatorname{F}}
\newcommand\W{{\mathbb{A}}}
\newcommand\K{\mathbb K}
\newcommand\U{\mathcal U}

\newcommand\Verm{\Delta}
\renewcommand\Pr{\operatorname{Pr}}
\newcommand\D{\mathcal D}
\newcommand\Ann{\operatorname{Ann}}
\newcommand\LAnn{\operatorname{LAnn}}
\newcommand\RAnn{\operatorname{RAnn}}
\newcommand\Id{\mathfrak{Id}}

\newcommand\Walg{\mathcal W}
\newcommand\ZZ{\mathbb Z}
\newcommand\A{\mathcal A}

\newcommand\M{\mathcal M}

\newcommand\gr{\operatorname{gr}}
\newcommand\OCat{\mathcal O}
\newcommand\I{\mathcal I}
\newcommand\J{\mathcal J}
\renewcommand\sl{\mathfrak{sl}}

\newcommand\Hom{\operatorname{Hom}}
\newcommand{\ad}{\mathop{\rm ad}\nolimits}

\newcommand{\HC}{\operatorname{HC}}
\newcommand\Centr{\mathcal Z}

\newcommand\mult{\operatorname{mult}}
\newcommand{\VA}{\operatorname{V}}

\newcommand\cell{\sigma}
\newcommand\dcell{\mathbf{c}}

\newcommand\Nil{\mathcal{N}}

\newcommand\Orb{\mathbb{O}}
\newcommand\KFun{\mathcal{K}}
\newcommand\FF{\mathcal{F}}
\newcommand\GF{\mathcal{G}}

\newcommand\KF{\operatorname{K}}
\newcommand\rank{\operatorname{rk}}
\newcommand\Bimod{\operatorname{Bimod}}
\newcommand\Irr{\operatorname{Irr}}
\newcommand\Wh{\operatorname{Wh}}

\newcommand\BQ{\mathbb{Q}}
\newcommand\Spr{\operatorname{Spr}}
\newcommand\bA{\bar{A}}
\newcommand\Hecke{\operatorname{H}}
\newcommand\JCat{\mathfrak{J}}
\newcommand\JAlg{\mathsf{J}}
\newcommand\Mfr{\mathfrak{M}}
\newcommand\BR{\mathbb{R}}
\newcommand\bF{\mathbb{F}}
\newcommand\Fam{\mathcal{F}}
\newcommand\CCat{\mathfrak{C}}
\newcommand\DCat{\mathfrak{D}}
\renewcommand\Vec{\operatorname{Vec}}
\newcommand\MCat{\mathfrak{M}}
\newcommand\YCat{\mathfrak{Y}}
\newcommand{\Fun}{\operatorname{Fun}}
\newcommand{\be}{\mathbf{1}}
\newcommand{\Rep}{\operatorname{Rep}}
\newcommand{\fA}{\mathbf{A}}
\newcommand{\bfA}{\bar{\mathbf{A}}}
\newcommand{\Lagr}{\mathcal{L}}
\newcommand{\bH}{\mathbf{H}}
\newcommand{\bpsi}{\bar{\psi}}
\newcommand{\Prim}{\operatorname{Pr}}
\newcommand{\Mat}{\operatorname{Mat}}
\newcommand{\GL}{\operatorname{GL}}
\newcommand{\Aut}{\operatorname{Aut}}
\newcommand{\BG}{\mathcal{BG}}

\newcommand{\Der}{\operatorname{Der}}
\newcommand{\cellb}{\tau}

\newtheorem{Thm}{Theorem}[section]
\newtheorem{Prop}[Thm]{Proposition}
\newtheorem{Cor}[Thm]{Corollary}
\newtheorem{Lem}[Thm]{Lemma}
\theoremstyle{definition}
\newtheorem{Ex}[Thm]{Example}

\newtheorem{Rem}[Thm]{Remark}
\newtheorem{Conj}[Thm]{Conjecture}
\unitlength=1mm   \numberwithin{equation}{section}
\numberwithin{table}{section} \oddsidemargin=0cm
\begin{document}
\begin{abstract}
Finite W-algebras are certain associative algebras arising in Lie theory.
Each W-algebra is constructed from a pair of a semisimple Lie algebra $\g$  (our base field
is algebraically closed and of characteristic $0$)
and its nilpotent element $e$. In this paper we classify finite dimensional
irreducible modules with integral central character over W-algebras.
In more detail, in a previous paper the first author proved that the component
group $A(e)$ of the centralizer of the nilpotent element under consideration
acts on the set of finite dimensional irreducible modules over the W-algebra
and the quotient set is naturally identified with the set of primitive ideals in
$U(\g)$ whose associated variety is the closure of the adjoint orbit of $e$.
In this paper, for a given primitive ideal with integral central character, we compute
the corresponding $A(e)$-orbit. The answer is that the stabilizer of that orbit is
basically a subgroup of $A(e)$ introduced by G. Lusztig. In the proof we use a variety
of different ingredients: the structure theory of primitive ideals and Harish-Chandra
bimodules for semisimple Lie algebras, the representation theory of W-algebras,
the structure theory of cells and Springer representations, and multi-fusion monoidal categories.
\end{abstract}
\maketitle
\tableofcontents

\section{Introduction}
\subsection{Finite W-algebras}
Finite W-algebras are certain associative algebras arising in  Lie representation theory.
Each W-algebra $\Walg$ is constructed from a pair $(\g,\Orb)$, where $\g$ is a semisimple Lie algebra
over an algebraically closed field $\K$ of characteristic 0, and $\Orb$ is a nilpotent orbit
in $\g$. Some information, including a definition, is recalled in Section \ref{SECTION_Walg}.
For more details about (finite) W-algebras the reader is referred to
the reviews  \cite{Wang},\cite{ICM}.

One of the most basic questions in Representation theory is, given an associative algebra
$\A$, classify its irreducible finite dimensional representations. In this paper we solve this problem
for finite $W$-algebras under the restriction that we only consider representations with
{\it integral} central character. It is known that the center of $\Walg$ is canonically identified
with the center of the universal enveloping algebra $\U:=U(\g)$ of $\g$ and so one uses
this identification to define the notion of
an integral central character for $\Walg$-modules.

Before stating our main result in Subsection \ref{SUBSECTION_main_result} we
would like to explain some prior classification results.


\subsection{Known classification results}
The main theorem of the present paper is a refinement of a classification result
from \cite{HC}, so we are going to explain that result first.

Fix a semisimple Lie algebra $\g$, its nilpotent orbit $\Orb$ and construct the
$W$-algebra $\Walg$ from these data. Let $\Irr_{fin}(\Walg)$ denote the set of finite dimensional irreducible $\Walg$-modules.

This set comes equipped with a finite group action. Namely, let $G$ denote the simply
connected semisimple algebraic group with Lie algebra $\g$. Pick an element
$e\in \Orb$ and consider its centralizer $Z_G(e)$ in $G$. In general, this subgroup
is not connected. Consider the component group $A(e):=Z_G(e)/Z_G(e)^\circ$, where
the superscript ``$^\circ$'' denotes the unit connected component. It turns out that
there is a natural action of $A(e)$ on $\Irr_{fin}(\Walg)$, see Section \ref{SECTION_Walg}
for the definition. This action is by outer automorphisms and so the modules in the same
orbit are basically indistinguishable.

In \cite{HC} the first author described the orbit space $\Irr_{fin}(\Walg)/A(e)$
for the $A(e)$-action on $\Irr_{fin}(\Walg)$. Namely, consider the set $\Prim_{\Orb}(\U)$
of all primitive ideals of $\U$ whose associated variety is the closure $\overline{\Orb}$
of $\Orb$, see Section \ref{SECTION_ideals} for a reminder on primitive ideals. Premet
conjectured, Conjecture 1.2.1 in loc.cit., that there is a natural identification between $\Irr_{fin}(\Walg)/A(e)$ and $\Prim_{\Orb}(\U)$
and this conjecture was proved in Subsection 4.2 of loc.cit.

The set $\Prim_{\Orb}(\U)$ is basically computable in all cases, we gather some
results in Subsection \ref{SUBSECTION_classification}. So, since the $A(e)$-conjugate modules
are practically indistinguishable, the result from \cite{HC} can be regarded
as an almost complete classification. To complete the classification one needs
to determine the  $A(e)$-orbit corresponding to each primitive ideal $\J\in \Prim_{\Orb}(\U)$.
This is a problem that we solve in this paper under the restriction that $\J$ has an
integral central character. The general case is an ``endoscopy-like'' problem
and so seems to be difficult, we only have some conjectures there.

Let us also recall other known classification results although they will not be
used in the present paper. In \cite{BK2} Brundan and Kleshchev produced an explicit
combinatorial description of $\Irr_{fin}(\Walg)$ for $\g=\sl_n$. Here the group
$A(e)$ acts trivially on $\Irr_{fin}(\Walg)$ and so $\Irr_{fin}(\Walg)\cong \Prim_{\Orb}(\U)$.
Modulo that identification, the Brundan-Kleshchev classification is equivalent
to the combinatorial description of $\Prim_{\Orb}(\U)$ due to Joseph. Some combinatorial
descriptions were also obtained for certain orbits in the classical Lie algebras, see
\cite{BG1},\cite{BG2}.

For the minimal nilpotent orbit in any $\g$ the classification of $\Irr_{fin}(\Walg)$
was obtained by Premet in \cite{Premet2}.

Finally, let us mention a result obtained in \cite{Miura}. There the first author obtained a
criterium for a module in the category $\mathcal{O}$ for $\Walg$ to be finite dimensional.
In the case when $\Orb$ is a so called {\it principal Levi} orbit this gives a complete
classification of $\Irr_{fin}(\Walg)$. The description heavily depends on the properties
of $\Prim_{\Orb}(\U)$ and so is rather implicit. Also it is a hard and interesting question of
how the classification  from \cite{Miura} agrees with the result of the present paper.

\subsection{Main theorem}\label{SUBSECTION_main_result}
So our goal is
to describe, for any ideal $\J\in \Prim_{\Orb}(\U)$ with integral central character,
the $A(e)$-orbit in $\Irr_{fin}(\Walg)$ lying over $\J$. We remark that
such $\J$ exists only when the orbit $\Orb$ is {\it special} in the sense of Lusztig,
see \cite[13.1.1]{orange} for the definition of a special orbit.

To a special orbit $\Orb$ one can assign a subset $\dcell$ in the Weyl group $W$ of $\g$
called a {\it two-sided cell}. The two-sided cell splits into the union of subsets called
{\it left cells}. For an integral central character $\lambda$ of $\U$ let $\Prim_{\Orb}(\U_\lambda)$
be the set of primitive ideals in $\Prim_{\Orb}(\U)$ with central character $\lambda$.
The set $\Prim_{\Orb}(\U_\lambda)$ naturally embeds into the set of all left cells
in $\dcell$, and this embedding is a bijection when $\lambda$ is regular.

Lusztig defined a certain quotient $\bA$ of $A(e)$, see \cite[p. 343]{orange}. Further to each left cell $\cell$ inside
$\dcell$ he associated a subgroup $H_\cell\subset \bA$ defined up to conjugacy,
see \cite[Proposition 3.8]{Lusztig_subgroups}. For reader's convenience let us provide
some description here, more details  will be given in Section \ref{SECTION_cell}. To $\dcell,\cell$
one assigns the cell $W$-modules $[\cell]\subset [\dcell]$. To $\Orb$ one assigns the Springer $W\times A(e)$-module
$\Spr(\Orb)$. Then $H_\cell$ has the (defining, in fact) property that $\BQ (\bA/H_\cell)=\Hom_W([\cell],\Spr(\Orb))$
as $A(e)$-modules, while $\bA$ is the minimal quotient of $A(e)$
that acts on $\Hom_W([\dcell],\Spr(\Orb))$.

Now we are ready to state our main result.

\begin{Thm}\label{Thm:very_main}
Let $\J\in \Prim_\Orb(\U)$ have integral central character. Let $\cell$ be the corresponding left cell.
The $A(e)$-orbit over $\J$ is $\bA/H_\cell$.
\end{Thm}

\subsection{Discussion}\label{SUBSECTION_discussion}
We would like to outline some ideas leading to the statement
of the theorem as well as some techniques used in the proof. The reader should keep in mind
that some constructions are explained informally and often not as they are used in the
actual proofs below.

Theorem \ref{Thm:very_main} was conjectured by R.~Bezrukavnikov and the second author
(unpublished).
The main motivation for that
conjecture came from \cite{BFO1}. So, first, we are going to explain what was done in \cite{BFO1}.

Fix a special orbit $\Orb$.
In \cite{Lusztig_tensor} Lusztig assigned a certain multi-fusion (=rigid monoidal, semisimple
with finitely many simples) category
$\JCat_{\Orb}$ to $\Orb$ (or, more precisely, to the corresponding
two-sided cell $\dcell$). The category $\JCat_{\Orb}$ categorifies the block $\JAlg_{\Orb}$
in Lusztig's asymptotic Hecke algebra $\JAlg$. Lusztig conjectured, among other things,
that his category should admit a fairly easy description: it should be isomorphic
to the category $\Coh^{\bA}(Y'\times Y')$, where $Y':=\bigsqcup_{\cell}\bA/H_\cell$, of
$\bA$-equivariant  sheaves of finite dimensional vector spaces on $Y'\times Y'$.
This conjecture was verified in \cite{BFO1}.

Lusztig's category $\JCat_{\Orb}$ can be defined using the representation theory of
$\U$. Namely, consider the monoidal category $\HC(\U_\rho)$ of all Harish-Chandra $\U$-bimodules
whose left and right central characters are trivial. Consider its tensor ideals
$\HC_{\overline{\Orb}}(\U_\rho)\supset \HC_{\partial\Orb}(\U_\rho)$ of all bimodules
supported on the closure $\overline{\Orb}$ and on the boundary $\partial\Orb$,
respectively. Form the quotient $\HC_{\Orb}(\U_\rho)$. We do not know whether this category is  semisimple,
%
in general, in any case, we  can consider its subcategory $\HC_{\Orb}(\U_\rho)^{ss}$ of all semisimple
objects. The last subcategory happens to be closed under  tensor products. Moreover,
it is naturally isomorphic to Lusztig's category $\JCat_{\Orb}$, see \cite{BFO2} for details.

On the other hand, in \cite{HC} the first author related the category $\HC_{\Orb}(\U)$
(define similarly to $\HC_{\Orb}(\U_\rho)$ but without restrictions on the central characters)
to a certain category of finite dimensional $\Walg$-bimodules to be defined next.

Consider a maximal reductive subgroup $Q$ of $Z_G(e)$. This group acts on $\Walg$ by automorphisms and the action
produces the $A(e)$-action on $\Irr_{fin}(\Walg)$. Moreover, there is a $Q$-equivariant embedding
of the Lie algebra $\q$ of $Q$ into $\Walg$.
So we can define the category $\HC^Q_{fin}(\Walg)$ of $Q$-equivariant finite dimensional $\Walg$-bimodules.
The $Q$-equivariance includes the condition that the action of $Q^\circ$ should integrate
the adjoint  $\q$-action, where $\q$ is viewed as a Lie subalgebra of $\Walg$. In other words,
the only additional structure on a $Q$-equivariant bimodule comparing to a usual bimodule
is the action of representatives of the elements of $Q/Q^\circ$. In particular, the subcategory
$\HC^Q_{fin}(\Walg_\rho)^{ss}$ of semisimple $Q$-equivariant finite dimensional bimodules
with left and right central characters $\rho$ roughly looks as the category $\Coh^{A(e)}(Y\times Y)$
with $Y=\Irr_{fin}(\Walg_\rho)$. One way how $\HC^Q_{fin}(\Walg_\rho)^{ss}$ may be different
from $\Coh^{A(e)}(Y\times Y)$ is via twists with 2- and 3-cocycles. We are not going to make this precise here,
we only present an example of what we mean by  a 2-cocycle twist.

Consider the algebra $\A:=\Mat_2(\K)\oplus \Mat_2(\K)$ and its group $\Gamma$ of automorphisms constructed
as follows. Take the dihedral subgroup $\operatorname{Dyh}_8$ of order 8 in $\GL_2(\K)$.
Then for $\Gamma$ take its image under the homomorphism
$\GL_2(\K)\rightarrow \operatorname{PGL}_2(\K)\times \operatorname{PGL}_2(\K)\subset \Aut(\A)$ projecting
$\GL_2(\K)$ onto the first factor. So $\Gamma\cong \ZZ/2\ZZ\times \ZZ/2\ZZ$. Consider the category $\Bimod^\Gamma(\A)$ of $\Gamma$-equivariant
(in the usual sense) $\A$-bimodules. 
It is easy to see that $\Bimod^{\Gamma}(\A)$ is equivalent to the category of all sheaves on $\{1,2\}\times \{1,2\}$,
whose fibers over the diagonal points are genuine $\Gamma$-modules, while fibers over the non-diagonal
points are projective modules corresponding to a non-trivial 2-cocycle. This is what we mean by a 2-cocycle twist.
3-cocycle twists we mentioned have to do with triple product isomorphisms that are a part of the
definition of a monoidal category.

Now let us explain a relationship between the categories $\HC_{\Orb}(\U)$ and $\Bimod^Q_{fin}(\Walg)$.
It was shown in \cite{HC}, Theorem 1.3.1, that there is a fully faithful tensor embedding
$\HC_{\Orb}(\U)\rightarrow \Bimod^Q_{fin}(\Walg)$ of abelian categories, whose image is closed under
taking subquotients. The embedding is compatible with central characters so that
$\HC_{\Orb}(\U_\rho)$ embeds into $\Bimod^Q_{fin}(\Walg_\rho)$. Since the image is closed
under subquotients we see that $\JCat_{\Orb}=\HC_{\Orb}(\U_\rho)$ embeds into
the category $\Bimod^Q_{fin}(\Walg_\rho)^{ss}$, which, as we mentioned, is, basically,
$\Coh^{A(e)}(Y\times Y)$ with various twists.

To relate the approach from \cite{HC} with that from \cite{BFO1} we would like to
assert that not only $Y=Y'$ (which is the regular central character case of Theorem \ref{Thm:very_main})
 but also that the embedding $\JCat_{\Orb}\hookrightarrow \Coh^{A(e)}(Y\times Y)$
realizes $\JCat_{\Orb}$ as $\Coh^{\bA}(Y\times Y)$. There is a bunch of various problems with this
claim -- otherwise we would not write this paper. The first problem to address is as follows:
not every multi-fusion subcategory of $\Coh^{A(e)}(Y\times Y)$ has the form $\Coh^{\bfA}(Y\times Y)$
(with various additional twists) for some quotient $\bfA$ of $ A(e)$ acting on $Y$.
It turns out that a criterium for a subcategory to have that form is that
for each $A(e)$-orbit in $Y\times Y$ the subcategory
has an object supported on that orbit.  It so happens that $\JCat_{\Orb}\subset \Coh^{A(e)}(Y\times Y)$
does have that property but this is a pretty non-trivial fact about W-algebras to be proved
in Section   \ref{SECTION_Walg_further}.

But even if we know that $\JCat_{\Orb}=\Coh^{\bfA}(Y\times Y)$ (again with various twists) there
are other problems: why $\bfA$  viewed as a quotient of $A(e)$ should be the same as $\bA$, why
no twists occur and, most importantly, why $Y=Y'$? An important remark here is that even if a finite
group $\Gamma$ is fixed, knowing the category $\Coh^{\Gamma}(X\times X)$ up to an equivalence of monoidal categories
is not sufficient to recover $X$ as a set with a $\Gamma$-action. The simplest example is as follows:
take $\Gamma=\ZZ/2\ZZ, X_1=\{pt\}, X_2=\Gamma$. Then the multi-fusion categories
$\Coh^{\Gamma}(X_1\times X_1), \Coh^{\Gamma}(X_2\times X_2)$ are equivalent. The presence
of twists makes things even worse: for $\Gamma=\ZZ/2\ZZ\times \ZZ/2\ZZ$ the category
$\Bimod^{\Gamma}(\A)$ considered above is equivalent to $\Coh^\Gamma(X\times X)$ with $X=\Gamma\sqcup\{pt\}$.

Fortunately, that ambiguity can be fixed with, basically, just one powerful tool. That tool
is a relationship between the finite dimensional representations of $\Walg_\rho$
and the Springer representation $\Spr(\Orb)$ obtained by Dodd in \cite{Dodd}. Namely consider the rational $K$-group
of the category of finite dimensional $\Walg_\rho$-modules. This $K$-group gets identified
with the  $\BQ$-span $\BQ (Y)$ of $Y$. One can define a natural
structure of a $W\times A(e)$-module on $\BQ (Y)$.    Dodd proved that there is
a $W\times A(e)$-equivariant embedding $\BQ (Y)\hookrightarrow \Spr(\Orb)$.

The existence of an embedding is a very restrictive condition on $Y$. It turns out that the claim we need
to prove is more or less equivalent to showing that $\BQ (Y)$ is as big as possible,
meaning that it coincides with the maximal submodule
$\Spr(\Orb)^{\dcell}$ of $\Spr(\Orb)$ whose irreducible $W$-submodules appear in the cell module $[\dcell]$.

So, very roughly, the proof of  Theorem \ref{Thm:very_main} uses the following algorithm:
\begin{enumerate}
\item[(0)] Start with a left cell $\cell$ with $H_{\cell}=\bA$ and assume no knowledge
of the $W\times A(e)$-module $\BQ (Y)$ (well, we know that the trivial $\bA$-module always occurs
there but this is not of much help).
\item[(1)] Use the information on the structure of $\JCat_{\Orb}$ and known irreducible
constituents of $\BQ (Y)$ to prove that the stabilizers in $Y$ for more left cells (=primitive
ideals) $\cellb$ coincide with $H_\cellb$.
\item[(2)] Get some new irreducible constituents of $\BQ (Y)$.
\item[(3)] If $\BQ (Y)=\Spr(\Orb)^{\dcell}$ is known, then we are, more or less, done.
Otherwise, return to step 1.
\end{enumerate}
In the course of the proof we will also see that basically no twists  occur in $\JCat_{\Orb}=\Coh^{\bA}(Y\times Y)$.

Of course, the scheme above assumes that there is a left cell $\cell$ with $H_{\cell}=\bA$, so the question
is whether this is always the case. The answer is: always, with three exceptions: one cell
for $E_7$ and 2 cells for $E_8$. These are so called exceptional cells that have to be treated separately,
see \cite{O}.

\subsection{Applications}
In fact, together with Theorem \ref{Thm:very_main} we obtain an alternative
proof of the Lusztig conjecture on the structure of $\JCat_{\Orb}$ mentioned in the previous subsection.
This application is pretty well expected and is straightforward from our proof.

There is also a much less expected application: using Theorem \ref{Thm:very_main} and some
of the techniques used in the proof the first author was able to compute the dimensions
of finite dimensional irreducible modules with integral central characters. Further, he proved
that the dimension of such a module equals the Goldie rank of the corresponding primitive ideal in $\U$.
This and related developments will be a subject of a forthcoming paper \cite{Goldie}.

\subsection{Structure of the paper}
Let us describe the organization of the paper. The paper is broken into sections, the beginning of each section describes its content
in more detail.

Sections \ref{SECTION_ideals}-\ref{SECTION_Walg}, \ref{SECTION_cell}
are preliminary and, basically, contain nothing new. In Section \ref{SECTION_ideals} we introduce certain categories related to
Harish-Chandra bimodules and describe some related constructions.
In our proofs we need to
use the language of multi-fusion and module categories, those are recalled
in Section \ref{SECTION_categories}. In Section \ref{SECTION_Walg}
we recall various facts about $W$-algebras.
Finally, in Section \ref{SECTION_cell} we recall various things related to cells and Lusztig's subgroups,
including the explicit computations of the latter.

In Section \ref{SECTION_Walg_further} is technical, there we prove several
results regarding the functor $\HC_{\Orb}(\U)\rightarrow \HC^Q_{fin}(\Walg)$.
The most important one is Theorem \ref{Thm_main}.

Finally, in Section \ref{SECTION:proof_main} we complete the proof of Theorem \ref{Thm:very_main}.


\subsection{Acknowledgements} We are grateful to R. Bezrukavnikov, C. Dodd, G. Lusztig
and D. Vogan for stimulating discussions.

\subsection{Conventions and notation}
In this subsection we describe conventions and the notation used in this paper.
The notation will be recalled from time to time in the main body of the paper.

\subsubsection{Lie algebras and algebraic groups}
Throughout the paper $\g$ is a reductive Lie algebra defined over an algebraically closed field
$\K$ of characteristic $0$. We fix a Borel subalgebra $\bfr\subset \g$ and a Cartan subalgebra $\h\subset \bfr$.
Let $\Delta,W$ be the root system and the Weyl group corresponding to the choice of $\h$, and $\Delta_+$
the system of positive roots corresponding to the choice of $\bfr$. Let $U(\g)$ denote the universal
enveloping algebra of $\g$. We will often write $\U$ for $U(\g)$.
Further, $G$ denotes a connected reductive algebraic group with Lie algebra $\g$.

By $(\cdot,\cdot)$ we denote a symmetric invariant  form on $\g$ whose restriction to $\h(\BQ)$ is
positive definite. We identify $\g$ with $\g^*$ using that form.

\subsubsection{Nilpotent orbits} By $\Orb$ we denote a nilpotent orbit in $\g$. Starting from Section \ref{SECTION_Walg_further}
we assume that the orbit $\Orb$ is special in the sense of Lusztig, unless otherwise is specified. We pick
an element $e\in \Orb$ and include it into an $\sl_2$-triple $(e,h,f)$. Then $S$ denotes the Slodowy
slice, $S:=e+\z_\g(f)$, where $\z_{\g}(\bullet)$ stands for the centralizer in $\g$.
We set $Q:=Z_G(e,h,f)$, this is a maximal reductive subgroup of the centralizer $Z_G(e)$.
Then $A:=A(e)=Q/Q^\circ=Z_G(e)/Z_G(e)^\circ$ denotes the component group.

Starting from $\g$ and $\Orb$ (or, more precisely, from $\g$ and the $\sl_2$-triple $(e,h,f)$)
one constructs the W-algebra to be denoted by $U(\g,e)$ or, more frequently, by  $\Walg$.

\subsubsection{Central characters}\label{SSS_central}
Recall that the center $\Centr(\U)$ of $\U$ is identified with $\K[\h^*]^W$
via the Harish-Chandra isomorphism: to $z\in \Centr(\U)$ one assigns the polynomial of $\nu$
by which $z$ acts on the irreducible module $L(\nu)$ with highest weight $\nu-\rho$.
Here, as usual, $\rho$ stands for half the sum of all positive roots.

One says that a character of $\Centr(\U)$ is {\it integral} (with respect to $G$) if its representative $\lambda\in \h^*$
lies in the character lattice of $G$, and {\it regular} if $\lambda$ is  non-zero on all coroots.
A non-regular character is also called {\it singular}.

The integral central characters are therefore in one-to-one correspondence with
the dominant weights, where $\lambda\in \h^*$ is called dominant if it is integral
and non-negative on all positive coroots.  We will usually denote dominant weights
(=integral central characters) by Greek letters $\lambda,\mu$. The set of dominant
characters of $G$ will be denoted by $P^+$. The set of strictly dominant characters
(those that are positive on all positive coroots) is denoted by $P^{++}$.
Further we say that $\mu\in P^+$ is compatible with $w\in W$ (or vice versa)
if $w\alpha\in -\Delta_+$ for any $\alpha\in \Delta_+$ with $\langle\mu,\alpha^\vee\rangle=0$.
Equivalently, $w$ is the longest element in $wW_\alpha$.
Clearly any integral element of $\h^*$ is represented in the form $w\mu$ with compatible
$w,\mu$ in a unique way.

For $\lambda\in \h^*$, let $\U_\lambda:=\U/\U \Centr(\U)_\lambda$ be the central reduction
of $\U$ by the ideal in $\U$ generated by the maximal ideal $\Centr(\U)_\lambda$ of $\lambda$ in $\Centr(\U)$.
Of course, with our conventions, $\U_\lambda=\U_{w\lambda}$ for all $w\in W$.

\subsubsection{Cells and Lusztig's groups}
In what follows we usually deal with one special orbit $\Orb$ (or the corresponding
two-sided cell $\dcell$) at a time so we do not indicate the dependence  on $\Orb$ or $\dcell$
when this does not lead to a confusion. By $\bA$ we denote the Lusztig quotient of
the component group $A$.

Left cells inside $\dcell$ are usually denoted by Greek letters $\sigma,\tau$. For a left cell
$\sigma$ by $H_\sigma$ we denote the corresponding Lusztig subgroup in $\bA$.

\subsubsection{Primitive ideals and Harish-Chandra bimodules}
Recall that by a primitive ideal one means the annihilator of an irreducible module.
For $\lambda\in \h^*$ set $\J(\lambda):=\Ann_{\U}L(\lambda)$. According to the Duflo
theorem, every primitive ideal is of the form $\J(\lambda)$ for some (non-unique,
in general) $\lambda\in \h^*$.

The set of all primitive ideals will be denoted by $\Prim(\U)$. By $\Prim_{\Orb}(\U)$
we denote the set of all primitive ideals $\J$ such that the associated variety $\VA(\U/\J)$
coincides with $\overline{\Orb}$. By $\Prim(\U_\lambda)$ we denote the subset of $\Prim(\U)$
consisting of all ideals with central character $\lambda$, i.e., all $\J$ with $\J\cap \Centr(\U)=\Centr(\U)_\lambda$.
We use the notation $\Prim(\U_\Lambda)$ to denote the set of all primitive ideals whose central
characters belong to a subset $\Lambda\subset \h^*$.

The notation for Harish-Chandra bimodules will be explained in more detail in Subsection \ref{SUBSECTION_HC}.
Let us mention now that we consider various categories related to Harish-Chandra bimodules
and our notation for them usually looks like $\,^{\Lambda_1}\!\HC^{\Lambda_2}_{Y}(\U)$.
This stands for the category of all HC bimodules with generalized left central character lying in $\Lambda_1\subset\h^*$,
generalized right central character in $\Lambda_2$, and the associated variety contained in $Y\subset \g\cong \g^*$
-- when $Y$ is closed. When $Y$ is locally closed, $\HC_Y(\U)$ stands for the subquotient
$\HC_{\overline{Y}}(\U)/\HC_{\partial Y}(\U)$. The notion of (generalized) central characters still
makes sense for the objects of such subquotients.

By $^{\Lambda_1}\!\JCat^{\Lambda_2}$ we denote the category $^{\Lambda_1}\!\HC^{\Lambda_2}_{\Orb}(\U)^{ss}$,
where the superscript ``$ss$'' means the semisimple part of the category.
Any missing superscript stands for the genuine central character $\rho$
so that, for instance, $\JCat^{\Lambda_2}$ means $^\rho\!\JCat^{\Lambda_2}$.
Also we will use the notation like $\JCat_\sigma, \,_\cellb\!\JCat, \,_\sigma\!\JCat_\cellb$ for certain subcategories
in $\JCat(=\,^\rho\!\JCat^\rho)$ associated with left cells $\sigma,\cellb$ -- see Subsection \ref{SUBSECTION_cell_HC} for the definitions.

\subsubsection{Irreducible $\Walg$-modules}
The set of all irreducible $\Walg$-modules with central characters from a subset $\Lambda\in \h^*$
will be denote by $Y^{\Lambda}$. As above, we write $Y$ for $Y^{\rho}$.
Below we will introduce an extension $\fA$ of $A(e)$ and a quotient $\bfA$ of $\fA$. The latter
acts on $Y^\Lambda$ and eventually will be shown to coincide with $\bA$. The stabilizer (defined
up to conjugacy) of the $\bfA$-orbit lying over $\J$ will be denoted by $\bH^\lambda_\sigma$,
where $\lambda$ is the central character, and $\sigma$ is the left cell corresponding to $\J$. So our main theorem just asserts
that $\bH^\lambda_\sigma=H_\sigma$.

\subsubsection{Miscellaneous notation}
This notation is summarized below.
\setlongtables
\begin{longtable}{p{2.5cm} p{12.5cm}}
$\A^{opp}$& the opposite algebra of $\A$.\\
$\widehat{\otimes}$&the completed tensor product of complete topological vector spaces/ modules.\\
$(a_1,\ldots,a_k)$& the two-sided ideal in an associative algebra generated by  elements $a_1,\ldots,a_k$.\\
 $A^\wedge_\chi$&
the completion of a commutative algebra $A$ with respect to the maximal ideal
of a point $\chi\in \Spec(A)$.\\
$\Ann_\A(\M)$& the annihilator of an $\A$-module $\M$ in an algebra
$\A$.\\
$\#\CCat$& the number of isomorphism classes of simple objects in an abelian category $\CCat$.\\
$[\CCat]$& the rational $K$-group of an abelian category $\CCat$.\\
$\Coh(X)$& the category of sheaves of finite dimensional
vector spaces on a set $X$.\\
$\Coh^{\Gamma}(X)$& the category of $\Gamma$-equivariant sheaves of finite dimensional
vector spaces on a $\Gamma$-set $X$.\\
$\Der(A)$& the Lie algebra of derivations of an algebra $A$.\\
$\Irr(\Gamma)$& the set of irreducible modules of a finite group $\Gamma$.\\
$G_x$& the stabilizer of $x$ in $G$.\\
$\gr \A$& the associated graded vector space of a filtered
vector space $\A$.\\
$\BQ (X)$& the $\BQ$-linear span of a finite set $X$.\\
$R_\hbar(\A)$&$:=\bigoplus_{i\in
\mathbb{Z}}\hbar^i \F_i\A$ :the Rees $\K[\hbar]$-module of a filtered
vector space $\A$.\\
$\Rep^\psi(\Gamma)$& the category of projective representations of a finite group $\Gamma$
corresponding to a 2-cocycle $\psi$.
\\$\VA(\M)$& the associated variety of $\M$.\\
$\Centr(\g)$& the center of $U(\g)$.\\
\end{longtable}

\section{Preliminaries on  Harish-Chandra bimodules}\label{SECTION_ideals}
\subsection{Subcategories and subquotients}\label{SUBSECTION_HC}
Let us recall that a $\U$-bimodule $\M$ is said to be {\it Harish-Chandra} (shortly, HC) for $G$
if
\begin{itemize}
\item $\M$ is finitely generated,
\item $\M$ coincides with the sum of its finite dimensional
submodules for the adjoint action of $\g$  ($\ad(\xi)m=\xi m-m\xi$),
\item and the adjoint $\g$-action on $\M$ integrates to a $G$-action.
\end{itemize}
Of course, for a simply connected semisimple group $G$ the last condition is satisfied automatically.

For $\M\in \HC(\U)$ one can define its associated variety $\VA(\M)\subset \g^*$ as follows.
Equip $\U$ with the standard PBW filtration. A compatible filtration on $\M$ is said
to be {\it good} if it is $\ad(\g)$-stable and the associated graded $\gr\M$ is a finitely
generated $\gr\U=S(\g)$-module (since the filtration is $\ad(\g)$-stable, the left and
the right $S(\g)$-actions on $\M$ coincide). By the associated variety $\VA(\M)$ of $\M$
one means the support of $\gr\M$ in $\g\cong \g^*=\Spec(S(\g))$.

For a closed $G$-stable subvariety $Y\subset \g^*$ let $\HC_{Y}(\U)$  denote the full subcategory in $\HC(\U)$ consisting
of all bimodules $\M$ with $\VA(\M)\subset Y$. Clearly, $\HC_{Y}(\U)$ is a Serre subcategory
of $\HC(\U)$ (i.e., it is closed under taking subquotients and extensions). We also remark that
$\VA(\M)=\VA(\U/\LAnn(\M))=\VA(\U/\RAnn(\M))$. Here and below
$\LAnn,\RAnn$ denote the left and right annihilators, respectively.

The tensor product over $\U$ defines the structure of a monoidal category on $\HC(\U)$ (the unit object
is $\U$ itself). For two modules $\M_1,\M_2\in \HC(\U)$ we have
\begin{equation}\label{eq:0.1}
\LAnn(\M_1\otimes_\U \M_2)\supset \LAnn(\M_1), \RAnn(\M_1\otimes_{\U} \M_2)\supset \RAnn(\M_2).
\end{equation}
It follows that $\VA(\M_1\otimes_\U \M_2)\subset \VA(\M_1)\cap \VA(\M_2)$.

There is an internal $\Hom$ functor in the category $\HC(\U)$. Namely, for
$\M_1,\M_2\in \HC(\U)$ the space $\underline{\Hom}(\M_1,\M_2):=\Hom_{\U^{opp}}(\M_1,\M_2)$
of homomorphisms of right $\U$-modules has a natural structure of a $\U$-bimodule.
It is known that $\underline{\Hom}(\M_1,\M_2)$ is HC, see e.g., \cite{Jantzen}, 6.36.
Moreover, \begin{equation}\label{eq:annihilators2}\LAnn(\underline{\Hom}(\M_1,\M_2))\subset \LAnn(\M_2),
\RAnn(\underline{\Hom}(\M_1,\M_2))\supset \LAnn(\M_1),\end{equation}
in particular, $\VA(\underline{\Hom}(\M_1,\M_2))\supset \VA(\M_1)\cap\VA(\M_2)$.

Now consider the category $\HC_{\overline{\Orb}}(\U)$ and its subcategory $\HC_{\partial\Orb}(\U)$, where
$\partial\Orb:=\overline{\Orb}\setminus \Orb$. We can form the quotient category
$\HC_{\Orb}(\U):=\HC_{\overline{\Orb}}(\U)/\HC_{\partial\Orb}(\U)$, which comes equipped with
natural tensor product and internal Hom functors induced from $\HC_{\overline{\Orb}}(\U)$.

Let us proceed to central characters.
One says  that $\lambda\in \h^*$ is the left central character of $\M\in \HC(\U)$ if the left $\U$-action
on $\M$ factors through  $\U_\lambda$. We say that $\lambda\in \h^*$ is the generalized central character
of $\M$ if the left $\Centr_\lambda$ action  on $\M$ is locally nilpotent. Here  $\Centr_\lambda$
denotes the maximal ideal of $\lambda$ in the center $\Centr$ of $\U$.
 Right (usual and generalized) central characters are defined similarly. Let
$^{\lambda}\!\HC(\U)$ (resp., $\HC^\lambda(\U)$) stand for the (full) subcategories in $\HC(\U)$
consisting of all HC bimodules with generalized left (resp., right) central character $\lambda$. Next, put
$^{\lambda}\!\HC^{\mu}(\U)=\,^{\lambda}\!\HC(\U)\cap \HC^{\mu}(\U)$.  Then for subsets
$\Lambda_1,\Lambda_2\subset \h^*$ we set
$$^{\Lambda_1}\!\HC(\U):=\bigoplus_{\lambda\in \Lambda_1}\,^{\lambda}\!\HC(\U), \quad \HC^{\Lambda_2}(\U):=\bigoplus_{\mu\in \Lambda_2}\HC^{\mu}(\U),$$
$$^{\Lambda_1}\!\HC^{\Lambda_2}(\U)=\,^{\Lambda_1}\!\HC(\U)\cap \HC^{\Lambda_2}(\U)=\bigoplus_{\lambda\in \Lambda_1,\mu\in \Lambda_2}\,^\lambda\HC(\U)^\mu.$$
Further for a closed subvariety $Y\subset \g^*$ we set $^{?}\HC^{\bullet}_Y(\U):=\,^{?}\HC^{\bullet}(\U)\cap \HC_Y(\U)$.
Then define $^{?}\!\HC^{\bullet}_\Orb(\U):=\,^{?}\!\HC^{\bullet}_{\overline{\Orb}}(\U)/^{?}\!\HC^{\bullet}_{\partial{\Orb}}(\U)$.
Of course, $$^{\Lambda_1}\!\HC^{\Lambda_2}_{\Orb}(\U)=\bigoplus_{\lambda\in \Lambda_1, \mu\in \Lambda_2}\,^\lambda\!\HC^\mu_{\Orb}(\U).$$
Also by $\HC^{(\lambda)}(\U)$ we denote the full subcategory of modules with actual right central character $\lambda$.

Recall the notation $^{\Lambda_1}\!\JCat_{\Orb}^{\Lambda_2}:=\,^{\Lambda_1}\!\HC_{\Orb}^{\Lambda_2}(\U)^{ss}$. When $\Lambda_1=\Lambda_2$ the last category is closed in $^{\Lambda}\HC_{\Orb}^{\Lambda}(\U)$ under the tensor product. For instance, this can be deduced
from Corollary 1.3.2 in \cite{HC} and (\ref{eq:0.1}), see \cite{BFO2} for details. Now \cite[Corollary 1.3.2]{HC} and (\ref{eq:annihilators2}) imply that $^{\Lambda}\!\JCat_{\Orb}^{\Lambda}$ is closed under the internal
Hom functor. Next, there is a unit object $\be$ in $^{\Lambda}\JCat_{\Orb}^{\Lambda}$ that is the class
of $\bigoplus_{\J} \U/\J$, where $\J$ is running over $\Prim_{\Orb}(\U_{\Lambda})$. So we can define
the duality functor $\bullet^*:=\underline{\Hom}(\bullet,\be)$ on $^{\Lambda}\!\JCat_{\Orb}^{\Lambda}$.

Any missed superscript in $^{\Lambda_1}\JCat_{\Orb}^{\Lambda_2}$ means ``$\rho$''.
We would like to remark that a  Harish-Chandra bimodule (for $G$) with left and right central
character that differ by an element of the root lattice is automatically an HC bimodule for
the adjoint group $\operatorname{Ad}(\g)$.

Also in the sequel
$\Orb$ will be fixed so we drop the subscript.

\subsection{Bernstein-Gelfand equivalence}
In this and the subsequent subsection $G$ is supposed to be semisimple and simply connected.
We will need the Bernstein-Gelfand equivalence, \cite{BG}, between $^\mu\!\OCat$ and $^\mu\!\HC^{(\rho)}(\U)$.
Let us recall a few basics about the category $\OCat$. By definition, it consists of all
finitely generated $\U$-modules with locally finite action of $\bfr$
and diagonalizable action of $\h$. Let $\Delta(\lambda)$
denote the Verma module with highest weight $\lambda-\rho$. Next, let $^\mu\!\OCat$ be the full
subcategory in $\OCat$ consisting of all modules with generalized central character $\mu$, and $^{P^+}\!\OCat$
be the subcategory of all modules with integral generalized central character.

Consider a functor $\BG:\,^{P^+}\!\OCat\rightarrow \,^{P^+}\!\HC^{(\rho)}(\U)$
that sends $M\in \,^{P^+}\!\OCat$ to the space $L(\Delta(\rho), M)$
of all $\g$-finite linear maps $\Delta(\rho)\rightarrow M$. This functor is known to be an equivalence
of abelian categories. Under this equivalence
the left annihilator of $L(\Delta(\rho),M)$ coincides with the annihilator of $M$. The right
annihilator of $L(\Delta(\rho),L(w\mu))$ is $J(w^{-1}\rho)$ provided $w,\mu$
are compatible in the sense explained in \ref{SSS_central}.

The equivalence $\BG$ is compatible with tensor products and internal Homs as follows:
we have $\M\otimes_{\U}\BG(N)=\BG(\M\otimes_\U N)$ for all $\M\in \,^{P^+}\!\HC^{P^+}(\U),
N\in \,^{P^+}\!\OCat$. Further we have $\underline{\Hom}(\BG(N_1),\BG(N_2))=L(N_1,N_2)$,
see, for example, \cite{Jantzen}, 6.37.

\subsection{Translation functors}\label{SUBSECTION_shifts}
 Fix $\lambda,\mu\in P^+$. Consider the categories $\U$-$\operatorname{Mod}^\lambda,\U$-$\operatorname{Mod}^\mu$
 of all $\U$-modules with generalized central characters $\lambda,\mu$. Then we have an exact functor $T_\lambda^\mu:\U$-$\operatorname{Mod}^\lambda\rightarrow
\U$-$\operatorname{Mod}^\lambda$ called the {\it translation functor}, see, for instance, \cite{Jantzen}, Kapitel 4.
Namely, let $\nu$ be the dominant weight lying in $W(\mu-\lambda)$. Then for $N\in \U$-$\operatorname{Mod}^\lambda$
for $T_\lambda^\mu(N)$ we take the component of character $\mu$ in $L(\nu+\rho)\otimes N$.


The functors
$T_\lambda^\mu$ enjoy the following properties:
\begin{enumerate}
\item $T_\lambda^\mu$ is an equivalence whenever $W_\lambda=W_\mu$.
\item Suppose $\lambda$ is regular. Then $T^\mu_\lambda$ maps
$L(w\lambda)$ to $L(w\mu)$ if $w$ and $\mu$ are compatible
and to 0 otherwise. In particular, if $T^{\lambda}_\mu(L(w_1\lambda))\cong T^{\lambda}_\mu(L(w_2\lambda))\neq 0$,
then $w_1=w_2$.
\item Suppose again that $\lambda$ is regular. Then $T_\lambda^\mu\circ T_\mu^\lambda$
is the sum of $|W_\mu|$ copies of the identity functor.
\end{enumerate}

(1) and (3) can be found in \cite{BG}, 4.1 and 4.2, respectively, while for (2) the reader is
referred to \cite{Jantzen}, 4.12.

On the category $\U$-$\operatorname{Mod}^{\lambda,k}$ consisting of all $\U$-modules
annihilated by $\Centr_\lambda^k$, the functor $T_\lambda^\mu$ is given by the tensor product by the Harish-Chandra
bimodule $\mathcal{T}_\lambda^\mu(k)$ that is the generalized eigenspace for $\mu$ in $L(\nu+\rho)\otimes \U/\U\Centr_\lambda^k$.
Therefore $T_\lambda^\mu=\varprojlim_k \mathcal{T}_\lambda^\mu(k)\otimes_\U\bullet$.
Property (3) means that $\mathcal{T}_\lambda^\mu(k')\otimes_{\U}\mathcal{T}_\mu^\lambda(k)=(\U/\U\Centr_\mu^k)^{\oplus |W_\mu|}$
for $k'\gg 0$ whenever $\lambda$ is regular.

In fact, $T_\lambda^\mu$ defines an inclusion preserving map between the sets of ideals in $\U_\lambda,\U_\mu$,
see \cite{Jantzen}, 5.4-5.8. Namely, we can define the map $T_\lambda^\mu$ by setting $T_\lambda^\mu(\Ann_{\U_\lambda} M):=\Ann_{\U_\mu}(T_\lambda^\mu(M))$. For regular $\lambda$  the map $T_\lambda^\mu$ restricts
to a map $T_\lambda^\mu:\Prim_{\Orb}(\U_\lambda)\rightarrow \Prim_{\Orb}(\U_\mu)\cup\{\U_\mu\}$. Next, there is an embedding $\underline{T}_\mu^\lambda: \Prim_{\Orb}(\U_\mu)\rightarrow \Prim_{\Orb}(\U_\lambda)$. For an irreducible $M\in \OCat^\mu$ this embedding sends $\Ann_\U M$  to the only minimal prime ideal of  $\Ann_\U(T_\mu^\lambda(M))$ that does not map to $\U_\mu$
under $T_\lambda^\mu$. Explicitly, $T_\lambda^\mu(J(w\lambda))=J(w\mu)$ if $w$ and $\mu$ are compatible and  $T_\lambda^\mu(J(w\lambda))=\U_\mu$ else. For compatible $\mu$ and $w$ we have $\underline{T}_\mu^\lambda(J(w\mu))=J(w\lambda)$.

Using the Bernstein-Gelfand equivalence,  we get the functor $T_{\rho}^\lambda: \,^{\rho}\!\HC^{\rho}_{\Orb}(\U)\rightarrow \,^{\lambda}\!\HC^{\rho}_{\Orb}(\U)$ with the following properties:
\begin{itemize}
\item[(i)] The functor sends an irreducible to an irreducible or 0.
\item[(ii)] Two irreducible bimodules with the same left annihilator
are sent or not sent to zero simultaneously.
\item[(iii)] The functor induces a bijection between the set of irreducibles in
$^{\rho}\!\HC^\rho_{\Orb}(\U)$ that it does not annihilate and the set
of all irreducibles in $^{\lambda}\!\HC^{\rho}_{\Orb}(\U)$. This bijection preserves
the right annihilators.
\end{itemize}


\section{Reminder on multi-fusion categories  and their modules}\label{SECTION_categories}
\subsection{Multi-fusion categories: definition and examples}
In this subsection we are going to recall various definitions, constructions
and results related to multi-fusion monoidal categories and their module
categories. The reader who is not familiar with the subject should view
it as a categorification of the representation theory of finite dimensional semisimple
associative algebras over an algebraically closed field. Of course, the categorical
framework is somewhat more involved.

A rigid monoidal $\K$-linear abelian category $\CCat$ is said to be {\it multi-fusion} if
\begin{itemize} \item it has finite dimensional $\Hom$-spaces, \item it is semisimple, \item all objects
have finite length, \item and there are finitely many isomorphism classes of simple objects.
\end{itemize} In particular, we will see that
the category $\JCat_{\Orb}$ (and, more generally, the category $\,^{\Lambda}\!\JCat_{\Orb}^\Lambda$)
is multi-fusion and this is the category we are mostly interested in.

Large class of multi-fusion categories can be produced from so called
{\it centrally extended} $\Gamma$-sets, where $\Gamma$ is a finite group.
Namely, let $X$ be a finite set acted on by a finite group $\Gamma$.
By the centrally extended structure on $X$ one means a $\Gamma$-invariant collection
$\psi=(\psi_x)_{x\in X}$ of classes $\psi_x\in H^2(G_x,\K^\times)$. Given
$X$ and $\psi$ we can form the category $\Coh^{\Gamma,\psi}(X\times X)$ of
$\psi$-twisted $\Gamma$-equivariant sheaves of finite dimensional vector
spaces on $X\times X$. Here ``$\psi$-twisted'' means that the fiber of any
obect in $(x,y)$ is a projective $\Gamma_{(x,y)}$-module, whose Schur
multiplier is $\psi_x|_{\Gamma_{(x,y)}}-\psi_y|_{\Gamma_{(x,y)}}$. The tensor
product on this category is defined by convolution. It is straightforward to
check that the category $\Coh^{\Gamma,\psi}(X\times X)$ is multi-fusion.

Yet another example is
the category $\Vec_{\Gamma}^\omega$, where $\omega\in H^3(\Gamma,\K^\times)$.
The simple objects $1_\gamma$ of this category are parameterized by $\gamma\in \Gamma$.
By definition, we have $1_{\gamma}\otimes 1_{\gamma'}=1_{\gamma\gamma'}$ and the triple
product isomorphism are defined using some 3-cocycle representing $\omega$. The category
$\Vec_\Gamma^\omega$ is a even a fusion category, which, by definition, means that the unit object is
simple. The categories $\Vec_\Gamma^\omega$ admit an axiomatic description: those are only fusion
categories such that all simple objects are invertible, see e.g. \cite[p. 183, Example (vi)]{Omod}.

We will say that a multi-fusion category $\DCat$ is a {\em quotient} of a multi-fusion category $\CCat$
if there is a tensor functor $F: \CCat \to \DCat$ such that any object of $\DCat$ is a direct summand
of an object of the form $F(X)$, $X\in \CCat$. For examples, the quotients of $\Vec_{\Gamma}$ are precisely the categories
$\Vec_{\bar{\Gamma}}^\omega$, where $\bar{\Gamma}$ is a quotient of $\Gamma$ and $\omega\in H^3(\bar{\Gamma},\K^\times)$
is such that its pull-back to $\Gamma$ is trivial.

If $\CCat_1,\CCat_2$ are multi-fusion categories, then their direct sum $\CCat_1\oplus \CCat_2$
acquires a natural structure of a multi-fusion category. So we get the notion of
an indecomposable multi-fusion category.
For example, it is easy to see that the categories $\Coh^{\Gamma,\psi}(Y\times Y),
\Vec_{\Gamma}^\omega$ are indecomposable.
The category $^{\Lambda}\!\JCat_{\Orb}^{\Lambda}$ is also indecomposable, see Proposition
\ref{Prop:J_indecomp}.

To conclude the subsection, let us mention that if $\CCat$ is a multi-fusion category, then $[\CCat]$ is a based ring in a sense of \cite{Lusztig_subgroups}, in particular, $[\CCat]$ is semisimple by \cite[1.2(a)]{Lusztig_subgroups}.
For example, $[\Vec_\Gamma^\omega]=\BQ (\Gamma)$.

\subsection{Module categories: definition and examples}\label{SUBSECTION_module_categories}
Now let $\CCat$ be a multi-fusion category. A $\CCat$-module
is a semisimple $\K$-linear abelian category $\MCat$ equipped with a tensor product functor
$\otimes:\CCat\boxtimes \MCat\rightarrow \MCat$ together with a collection of associativity
isomorphisms $(C_1\otimes C_2)\otimes M\xrightarrow{\sim}C_1\otimes (C_2\otimes M)$
for all objects $C_1,C_2$ of $\CCat$ and $M$ of $\MCat$.
The rational $K$-group $[\MCat]$ of a $\CCat$-module $\MCat$ is naturally a $[\CCat]$-module.

The module categories we are mostly interested in are the  $\JCat_{\Orb}$-module $\YCat:=\Coh(Y)$,
and, more generally, the $\,^{\Lambda}\!\JCat_{\Orb}^\Lambda$-module $\YCat^\Lambda:=\Coh(Y^\Lambda)$.

An example of a $\Coh^{\Gamma,\psi}(X\times X)$-module is provided by the category
$\Coh(X)$. The module structure is again given by convolution.

For $\CCat$-modules $\MCat_1,\MCat_2$ the direct sum $\MCat_1\oplus\MCat_2$
has a natural $\CCat$-module structure. So we get the natural notion of an indecomposable
$\CCat$-module. For example, a $\Coh^{\Gamma,\psi}(X\times X)$-module $\Coh(X)$
is always indecomposable. The main result of Section \ref{SECTION_Walg_further} is equivalent to saying that
the $\,^\Lambda\!\JCat^\Lambda_{\Orb}$-module $\YCat^\Lambda$ is indecomposable for any $\Lambda$.

Let us describe indecomposable $\Vec_{\Gamma}^\omega$-modules, see e.g. \cite[Example 2.1]{Orbi}.
Up to equivalence, they are classified by pairs $(\Gamma_0, \psi)$, where $\Gamma_0$ is a subgroup of $\Gamma$ defined up to conjugacy
and $\psi$ is a $\K^\times$-valued 2-cochain on $\Gamma_0$ with $d\psi$ cohomological to
$\omega$. The simples in the corresponding category are parameterized by the points of $\Gamma/\Gamma_0$,
and, moreover, the rational $K$-group is $\BQ (\Gamma/\Gamma_0)$ as a $[\Vec_{\Gamma}^\omega]=\BQ (\Gamma)$-module.

Finally let us remark that the definition introduced above is of left module categories.
Similarly, one can speak about right module categories. Also for two multi-fusion categories
one can speak about their bimodules.

\subsection{Functors between $\CCat$-modules}
Let $\CCat$ be a multi-fusion category and $\MCat_1,\MCat_2$ be its module categories. We can consider the category
$\Fun_{\CCat}(\MCat_1,\MCat_2)$ of $\CCat$-linear functors $\MCat_1\rightarrow \MCat_2$. Its objects are functors
$\MCat_1\rightarrow\MCat_2$ together with ``$\CCat$-linearity isomorphisms'' (that piece of data is needed because the $\CCat$-linearity
cannot be a condition; in the categorical world this is always
an additional structure, see, e.g., \cite{Omod}, for details).

The category $\Fun_{\CCat}(\MCat,\MCat)$ has a natural tensor product
and is a multi-fusion category (indecomposable if $\CCat$ is indecomposable,
genuinely fusion if $\MCat$ is indecomposable), see \cite[Theorem 2.18]{ENO}.
This category is said to be dual to $\CCat$ with respect to
$\MCat$ and is denoted by $\CCat^*_{\MCat}$. We remark that $\MCat$ is naturally a right $\CCat^*_{\MCat}$-module
and we have the double centralizer property: $\CCat=(\CCat^*_{\MCat})^*_{\MCat}$,
\cite[Theorem 4.2]{Omod}.
We remark that for the double centralizer property it is important that
$\CCat$ is indecomposable.

In general, $\Fun_{\CCat}(\MCat_1,\MCat_2)$ is a  $\CCat_{\MCat_1}^*$-$\CCat_{\MCat_2}^*$-bimodule.

Below we will need various more or less standard facts about the functor
categories.

\begin{Lem}\label{Lem:0.2} Let $X$ be a finite set acted on by a finite
group $\Gamma$. Consider the category $\Coh^{\Gamma,\psi}(X\times X)$
and its rigid monoidal (hence multi-fusion) subcategory $\CCat$ such that
$\MCat:=\Coh(X)$ is indecomposable over $\CCat$.

 (i) The dual category of $\Coh^{\Gamma,\psi}(X\times X)$ with respect to the module
category $\MCat$ is $\Vec_{\Gamma}$.

(ii)  $\CCat^*_{\MCat}=\Vec_{\bar{\Gamma}}^\omega$, where
$\bar{\Gamma}$ is some quotient of $\Gamma$ and $\omega \in H^3(\bar{\Gamma}, \K^\times)$ is such that the functor
$\Vec_\Gamma \to \Fun_{\CCat}(\MCat,\MCat)$ factors through $\Vec_{\Gamma} \to \Vec_{\bar{\Gamma}}^\omega$ (so $\MCat$ is a module
category over $\Vec_{\bar{\Gamma}}^\omega$).

(iii)  $\CCat\cong  (\Vec_{\bar{\Gamma}}^\omega)^*_\MCat$.
\end{Lem}
\begin{proof}
(i) is basically  a corollary of the double centralizer property -- $\Coh^{\Gamma,\psi}(X\times X)$ is just
 the dual to $\Vec_\Gamma$ with respect to
$\Coh(X)$.

(ii): Since the functor $\CCat\hookrightarrow \Coh^{\Gamma}(X\times X)$ is {\em injective} in the sense
of \cite[\S 5.7]{ENO}, we have by \cite[Proposition 5.3]{ENO} that the dual functor $\Vec_{\Gamma}\to \CCat_\MCat^*$ is {\em surjective} in the sense of {\em loc. cit.} It follows that any simple object of $\CCat_\MCat^*$
is invertible, whence $\CCat_\MCat^*=\Vec_{\bar{\Gamma}}^\omega$.


(iii) is the double centralizer property.
\end{proof}

We say that an object $e\in \CCat$ is a direct summand of the unit object if there is another object $f\in \CCat$ and an isomorphism
$e\oplus f\simeq \be$. It is clear that such an object is an idempotent, that is, there is a canonical isomorphism $e\otimes e\simeq e$,
and is self-dual: $e\cong e^*$.
For a simple object $X\in \CCat$ we have either $X\otimes e\simeq X$ or $X\otimes e=0$. Thus $\CCat \otimes e$ is a full subcategory of $\CCat$;
clearly $\CCat \otimes e$ is a left $\CCat$-module subcategory of $\CCat$.
Similarly, for a $\CCat$-module category $\MCat$ we have a full subcategory $e\otimes \MCat \subset \MCat$.
This subcategory is a  right $\CCat^*_{\MCat}$-submodule in $\MCat$.

Let us consider some examples. Let $\Gamma,X,\psi$ be such as above. Let $X=\bigsqcup_{i=1}^n X_i$
be the $\Gamma$-orbit decomposition. Then one has the decomposition
$\be=\bigoplus_{i=1}^n e_i$, where the summands are in one-to-one correspondence with the $\Gamma$-orbits
in $X$. We have $\Coh^{\Gamma,\psi}(X\times X)\otimes e_i=\Coh^{\Gamma,\psi}(X\times X_i)$ and
$e_i\otimes \Coh(X)= \Coh(X_i)$.

We are going to investigate some properties of the categories $\Fun_{\CCat}(\CCat\otimes e, \MCat)$.

First, we need to know various aspects  of how the duality agrees with passing
to the subcategories of the form $e\otimes \CCat\otimes e$.

%
%

\begin{Lem}\label{Lem:0.3pre}
Let $\CCat$ be a  multi-fusion category and $\MCat$ be its
left module. Consider the decomposition $\be=\bigoplus_i e_i$ and set
$\MCat_i:=e_i\otimes \MCat$. Then we have the following natural identifications:

(i)  $\Fun_{\CCat}(\CCat\otimes e_i, \MCat)= \MCat_i$
of right $\CCat^*_{\MCat}$-modules.

(ii)  $(e_i\otimes \CCat\otimes e_i)^*_{\MCat_i}\cong \CCat^*_{\MCat}$ of multi-fusion categories,
provided $\CCat$ is indecomposable.

(iii) $e_j\otimes \CCat\otimes e_i=\Fun_{\CCat^*_{\MCat}}(\MCat_i,\MCat_j)$ of abelian categories.
\end{Lem}
\begin{proof}
The identification in (i) is given by $F\mapsto F(e_i)$.

Using $\CCat=(\CCat^*_\MCat)^*_\MCat$ we get
$$\CCat =\bigoplus_{i,j}\Fun_{\CCat_\MCat^*}(e_i\otimes \MCat,e_j\otimes \MCat).$$
Clearly, $\Fun_{\CCat_\MCat^*}(e_i\otimes \MCat,e_j\otimes \MCat)\subset e_j\otimes \CCat \otimes e_i$,
whence $\Fun_{\CCat_\MCat^*}(e_i\otimes \MCat,e_j\otimes \MCat)=e_j\otimes \CCat \otimes e_i$, which is (iii).

Setting $i=j$ and applying the duality again, we get (ii).
\end{proof}

\begin{Rem}\label{Rem:idemp_nonzero}
(i) implies that $e_j\otimes \CCat\otimes e_i=\Fun_{\CCat}(\CCat\otimes e_i,\CCat\otimes e_j)$.
Also (ii) implies that $\MCat_i$ is nonzero.
\end{Rem}

Second, we need to describe the $K$-groups of certain $\Fun$-categories.

\begin{Lem}\label{Lem:0.3}
Let $\CCat$ be a multi-fusion category, $\MCat$ its left module, and $e$
a direct summand of the unit object in $\CCat$.
We have a natural (in particular, $[\CCat^*_{\MCat}]$-linear)
isomorphism $\Hom_{[\CCat]}([\CCat\otimes e], [\MCat])=[\Fun_{\CCat}(\CCat\otimes e, \MCat)]$.
\end{Lem}
\begin{proof}
Thanks to Lemma \ref{Lem:0.3pre},
$\Fun_\CCat(\CCat \otimes e,\MCat)\simeq e\otimes \MCat$. It is clear that $[e\otimes \MCat]=[e][\MCat]$, where
$[e]\in [\CCat]$ is the class of $e\in \CCat$. An isomorphism of the lemma  follows since
we have the following equalities of  $[\CCat^*_{\MCat}]$-modules
$$\Hom_{[\CCat]}([\CCat \otimes e],[\MCat]))=\Hom_{[\CCat]}([\CCat][e],[\MCat])=
[e][\MCat]=[e\otimes \MCat].$$
\end{proof}

\begin{Rem} Let us notice that for arbitrary module categories $\MCat_1, \MCat_2$ even the equality
$\# \Fun_\CCat(\MCat_1,\MCat_2)=\dim_\BQ \Hom_{[\CCat]}([\MCat_1],[\MCat_2])$ does not hold.
For example, set $\CCat=\Rep(\Gamma)$ for some finite group $\Gamma$. Then the category
$\MCat:=\Vec$ of vector spaces is a left $\CCat$-module -- the action of
$\CCat$ on $\MCat$ factors through the forgetful functor $\Rep(\Gamma)\rightarrow \Vec$.
We have  $\Fun_\CCat(\MCat,\MCat)=\Vec_\Gamma$, while
$\Hom_{[\CCat]}([\MCat],[\MCat])$ is 1-dimensional.
\end{Rem}

\section{Preliminaries on W-algebras}\label{SECTION_Walg}
\subsection{Definition}\label{SUBSECTION_W_def}
First we  recall the definition of the W-algebras given in \cite{Wquant} (with slight refinements obtained
in \cite{HC}). All results mentioned here can be found in \cite{HC}, Subsections 2.1, 2.2.

Recall that $G$ is an arbitrary connected reductive group.

 The Slodowy slice $S\subset \g$ is, by definition, the affine subspace
$e+\z_\g(f)\subset\g$, where $\z_\g(\cdot)$ denotes the centralizer. It will
be convenient for us to consider $S$ as a subvariety in $\g^*$.

Consider the {\it equivariant Slodowy slice} $X:=G\times S\subset G\times \g^*=T^*G$. The group $G$ acts on $T^*G$ and on $X$ by left translations. Moreover, we have  actions of the one-dimensional torus $\K^\times$ and of the group $Q:=Z_G(e,h,f)$ on $T^*G$
defined by
\begin{align*}&t.(g,\alpha)=(g\gamma(t)^{-1}, t^{-2}\gamma(t)\alpha),\\ &q.(g,\alpha)=(gq^{-1}, q\alpha),\\ &t\in\K^\times,q\in Q, g\in G,\alpha\in \g^*.\end{align*}
Here $\gamma:\K^\times\rightarrow G$ is the one-parameter group with $d_1\gamma=h$.
The subvariety $X\subset T^*G$ is $Q\times \K^\times$-stable.

According to \cite{Wquant}, Subsection 3.1, there is a $G\times Q$-invariant
symplectic form $\omega$ on $X$ satisfying the additional condition
$t.\omega=t^2\omega, t\in \K^\times$. This form is obtained by restricting
to $X$ the natural form on $T^*G$.

Using the Fedosov deformation quantization, in \cite{Wquant} the first author constructed
a star-product $*:\K[X]\otimes_{\K} \K[X]\rightarrow \K[X][\hbar], f*g:=\sum_{i=0}^\infty D_i(f,g)\hbar^{2i},$ satisfying the following conditions:
\begin{itemize}
\item[(i)] A natural $\K[\hbar]$-bilinear extension of $*$ to $\K[X][\hbar]\otimes_{\K[\hbar]}\K[X][\hbar]$
is associative. Moreover, $1\in \K[X][\hbar]$ is a unit for $*$.
\item[(ii)] $D_0(f,g)=fg, D_1(f,g)-D_1(g,f)=\{f,g\}$, where $\{\cdot,\cdot\}$ stands for the Poisson
bracket associated with $\omega$.
\item[(iii)] $*$ is $G\times Q$ equivariant, i.e., all $D_i$ are $G\times Q$-equivariant maps.
Also $*$ is homogeneous, that is, $D_i$ has degree $-2i$ for any $i$.
\end{itemize}

By a {\it homogeneous equivariant $W$-algebra} we mean the space $\widetilde{\Walg}_\hbar:=\K[X][\hbar]$
equipped with the star-product constructed above. A {\it homogeneous $W$-algebra}, by definition, is
$\Walg_\hbar:=\widetilde{\Walg}_\hbar^G$. Finally, by definition, the equivariant $W$-algebra $\widetilde{\Walg}$
is $\widetilde{\Walg}_\hbar/(\hbar-1)$, and the $W$-algebra $\Walg$ is
$\widetilde{\Walg}^G=\Walg_\hbar/(\hbar-1)$. Let $\KF_i\Walg$ denote the image of
the $\K^\times$-eigenspace in $\Walg_\hbar$ with eigencharacter $t\mapsto t^i$. Then the spaces
$\KF_i\Walg$ form an increasing exhaustive filtration on $\Walg$.


By \cite{HC}, there is a $G\times Q$-equivariant map $\g\times\q\rightarrow \widetilde{\Walg}_\hbar, \xi\mapsto \widehat{H}^{\Walg}_\xi,$
that is a {\it quantum comoment map} for the action of $G\times Q$ on $\widetilde{\Walg}$, i.e.,
$[\widehat{H}^\Walg_\xi, f]=\hbar^2\xi_{\widetilde{\Walg}}f$ for any $f\in \widetilde{\Walg}_\hbar,\xi\in \g\oplus \q$.
In the r.h.s. $\xi_{\widetilde{\Walg}}$ means the derivation of $\widetilde{\Walg}_\hbar$ coming from the group action. Taking the quotient by $\hbar-1$ we get the quantum comoment map $\g\times \q\rightarrow \widetilde{\Walg}$ also denoted by $\widehat{H}^{\Walg}_\xi$.

The quantum comoment maps $\q\rightarrow \widetilde{\Walg}_\hbar,\widetilde{\Walg}$ are  $G$-invariant and so
their images lie in $\Walg_\hbar,\Walg$, respectively.
Also the algebra homomorphism  $\U\rightarrow \widetilde{\Walg}$ induced by the quantum comoment map for the $G$-action is $G$-equivariant. So restricting it to the $G$-invariants we get
a homomorphism $\Centr\rightarrow \Walg$.

Below we will need the following lemma
that essentially  appeared in \cite{Wquant}.

\begin{Lem}\label{Lem:1.5.1}
There is a $G\times Q$-equivariant
isomorphism $\K[G]\otimes \Walg\rightarrow \widetilde{\Walg}$
of right $\Walg$-modules. Here $Q$ acts by right translations on $\K[G]$
and diagonally on the tensor product.
\end{Lem}
\begin{proof}
We have an obvious  $G\times Q$-equivariant isomorphism $\K[G]\otimes \K[S]\rightarrow \K[X]$. Lift
$\K[G]$ to a $G\times Q$-stable subspace in $\widetilde{\Walg}$. We claim that the multiplication map
$\K[G]\otimes \Walg\rightarrow \widetilde{\Walg}$ is an isomorphism with required properties.
This map is injective, because its associated graded map is.
 The surjectivity follows from the observation
that the filtration on any $G$-isotypic component in $\widetilde{\Walg}$ is bounded
from below. The latter is a consequence of the fact that $\K[S]$ is positively graded, see the proof
of Proposition 2.1.5 in \cite{Wquant}.
\end{proof}

We finish this subsection by recalling Premet's definition, \cite{Premet1}, Section 4, of the W-algebras that was
historically first.

Introduce a grading on $\g$ by eigenvalues of $\ad h$:
$\g:=\bigoplus \g(i), \g(i):=\{\xi\in\g| [h,\xi]=i\xi\}$
so that $\gamma(t)\xi=t^i\xi$ for $\xi\in \g(i)$. Define the element $\chi\in \g^*$ by
$\chi=(e,\cdot)$ and the skew-symmetric form
$\omega_\chi$ on $\g(-1)$ by $\omega_\chi(\xi,\eta)=\langle\chi,[\xi,\eta]\rangle$.
It turns out that this form is symplectic. Fix a Cartan subalgebra $\t\subset \q$. Pick a $\t$-stable lagrangian subspace $l\subset \g(-1)$
and define the subalgebra $\m:=l\oplus\bigoplus_{i\leqslant -2}\g(i)$. Then
$\chi$ is a character of $\m$. Define the shift $\m_{\chi}=\{\xi-\langle\chi,\xi\rangle,\xi\in \m\}\subset \g\oplus\K$.
Essentially, in \cite{Premet1} the W-algebra was defined as the quantum Hamiltonian reduction
$(\U/\U\m_\chi)^{\ad \m}$.

We checked in \cite{Wquant}, see also \cite{HC}, Theorem 2.2.1, and the discussion after it,
that both definitions agree and also that the homomorphism
$\Centr\rightarrow \Walg$ coincides with one considered in \cite{Premet1}, 6.2,
and so is an isomorphism of $\Centr$ with the center of $\Walg$, see \cite{Premet2}, footnote 2.
Below we always identify $\Centr$ with the center of $\Walg$
using this isomorphism.

A useful feature of Premet's construction is that it allows to construct functors
between the categories of $\U$- and $\Walg$-modules. We say that a left $\U$-module $M$ is
a {\it Whittaker} module if $\m_\chi$ acts on $M$ by locally
nilpotent endomorphisms. In this case $M^{\m_\chi}=\{m\in M| \xi
m=\langle\chi,\xi\rangle m, \forall \xi\in\m\}$ is a nonzero $\Walg$-module. As
Skryabin proved in the appendix to \cite{Premet1}, the functor
$M\mapsto M^{\m_\chi}$ is an equivalence between the category of
Whittaker $\U$-modules and the category $\Walg$-$\Mod$ of
$\Walg$-modules. A quasiinverse equivalence
is given by $N\mapsto \Sk(N):=(\U/\U\m_\chi)\otimes_\Walg N$,
where $\U/\U\m_\chi$ is equipped with a natural structure of a
$\U$-$\Walg$-bimodule. In the sequel we will call $\Sk$ the {\it Skryabin functor}.

\subsection{Decomposition theorem}\label{SUBSECTION_decomposition}
All results of this subsection are taken from \cite{HC}, Subsection 2.3.

Set $V:=[\g,f]$. Equip $V$ with a symplectic form
$\omega(\xi,\eta)=\langle\chi,[\xi,\eta]\rangle$, an action of
$\K^\times: t.v=\gamma(t)^{-1}v$, and an  action of $Q$ restricted from $\g$.
Consider the homogeneous Weyl algebra $\W_\hbar:=T(V)[\hbar]/(u\otimes v-v\otimes u-\hbar^2 \omega(u,v))$. We have $Q$- and
$\K^\times$-actions of $\W_\hbar$ induced from $V$ (with $q.\hbar=\hbar, t.\hbar=t\hbar$).
As a vector  space, $\W_\hbar$ coincides with $\K[V^*][\hbar]$, while the product
 on $\W_\hbar$ is the Moyal-Weyl star-product.
The quotient $\W:=\W_\hbar/(\hbar-1)$ is the usual Weyl algebra.

Consider the cotangent bundle $T^*G$ of $G$. There is a $G\times Q$-equivariant homogeneous (with respect to
the Kazhdan action) star-product
$*:\K[T^*G]\otimes \K[T^*G]\rightarrow \K[T^*G][\hbar]$.
Set $\chi:=(e,\cdot)\in \g^*, x:=(1,\chi)\in X\subset T^*G$. The star-products on
$\K[T^*G][\hbar], \K[X][\hbar]$ extend by continuity to the completions
$\K[T^*G]^\wedge_{Gx}[[\hbar]],\K[X]^\wedge_{Gx}[[\hbar]]$. Also the star-product
on $\W_\hbar$ extends by continuity to $\W^\wedge_\hbar:=\K[V^*]^\wedge_0[[\hbar]]$.
See \cite{Wquant}, Subsection 3.3 for details.

We remark that the algebra $\K[T^*G]^\wedge_{Gx}[[\hbar]]$ is the completion of
$\K[T^*G][\hbar]$ with respect to the preimage of the ideal of the orbit $Gx$
in $\K[T^*G]$. A similar claim holds for $\K[X]^\wedge_{Gx}[[\hbar]]$.

The decomposition theorem below asserts that $\K[T^*G]^\wedge_{Gx}[[\hbar]]$ can be decomposed
into the completed tensor product of $\W^\wedge_\hbar$ and $\K[X]^\wedge_{Gx}[[\hbar]]$.
Moreover, this decomposition agrees with quantum comoment maps.

More precisely, we have quantum comoment maps $\g\times \q\rightarrow \K[T^*G][\hbar], \K[X][\hbar]$
and also a quantum comoment map $\q\rightarrow \W_\hbar$ for the $Q$-action on
$\W_\hbar$, see \cite{HC}, Subsections 2.1,2.2 for details. So we get
Lie algebra homomorphisms $\g\times \q\rightarrow \K[T^*G]^\wedge_{Gx}[[\hbar]],
\W^\wedge_\hbar\widehat{\otimes}_{\K[[\hbar]]}\K[X]^\wedge_{Gx}[[\hbar]]$ that are quantum comoment maps
for the $G\times Q$-actions. Also we remark that $\K^\times$ acts on $\K[T^*G]^\wedge_{Gx}[[\hbar]],
\W^\wedge_\hbar\widehat{\otimes}_{\K[[\hbar]]}\K[X]^\wedge_{Gx}[[\hbar]]$.

\begin{Prop}[\cite{HC}, Theorem 2.3.1]\label{Prop:10.1}
There is a $G\times Q\times \K^\times$-equivariant isomorphism $\Phi_\hbar: \K[T^*G]^\wedge_{Gx}[[\hbar]]\rightarrow
\W^\wedge_\hbar\widehat{\otimes}_{\K[[\hbar]]}\K[X]^\wedge_{Gx}[[\hbar]]$ of topological $\K[[\hbar]]$-algebras intertwining the quantum comoment maps from $\g\times\q$.
\end{Prop}

Using the star-product we
can equip $\U_\hbar:=\K[\g^*][\hbar]=\K[T^*G][\hbar]^G$ with a new associative product.
The quotient $\U_\hbar/(\hbar-1)$ is identified with $\U$. To recover $\U_\hbar$ from
$\U$ recall the notion of the Rees algebra. Namely, let $\A$ be an associative algebra
equipped with an increasing exhaustive $\ZZ_{\geqslant 0}$-filtration $\F_i\A$. Then, by definition,
the Rees algebra $R_\hbar(\A)$ is $\bigoplus_{i\geqslant 0}\F_i\A \hbar^i\subset \A[\hbar]$.
Now equip $\U$ with the "doubled" standard filtration: the space $\F_i\U$ is spanned
by all monomials $\xi_1\ldots \xi_j, 2j\leqslant i$. We get $R_\hbar(\U)=\U_\hbar$.

The action of $Q$ on $\U_\hbar$ has a quantum comoment map that is nothing else
but the natural embedding $\q\hookrightarrow \g\subset R_\hbar(\U)= \U_\hbar$. We can form
the completion $\U^\wedge_\hbar:=\K[\g^*]^\wedge_{\chi}[[\hbar]]$ and extend the star-product from
$\U_\hbar$ to $\U^\wedge_\hbar$. Alternatively, the star-product on $\U^\wedge_\hbar=\K[T^*G]^\wedge_{Gx}[[\hbar]]^G$ is obtained by restriction from $\K[T^*G]^\wedge_{Gx}[[\hbar]]$.

Similarly, define the completion $\Walg^\wedge_\hbar:=\K[X]^\wedge_{Gx}[[\hbar]]^G$
and equip it with the product induced from $\Walg_\hbar$. The isomorphism $\Phi_\hbar$ from Proposition \ref{Prop:10.1}
restricts to a $Q\times\K^\times$-equivariant isomorphism $\U^\wedge_\hbar\rightarrow
\W^\wedge_\hbar\widehat{\otimes}_{\K[[\hbar]]}\Walg^\wedge_\hbar$ intertwining the quantum
comoment maps from $\q$.

\subsection{Primitive ideals and Harish-Chandra bimodules vs W-algebras}\label{SUBSECTION_primitive}
In this subsection we will explain results of \cite{HC} on a relationship between
\begin{itemize}
\item The set $\Id_{fin}(\Walg)$ of two-sided ideals of finite codimension in $\Walg$ and the
set $\Id_{\Orb}(\U)$ of two-sided ideals of $\J\subset\U$ with $\VA(\U/\J)=\overline{\Orb}$.
\item The category $\HC_{fin}^Q(\Walg)$ of finite dimensional $Q$-equivariant bimodules
and the  subquotient  $\HC_{\Orb}(\U)$ of  $\HC(\U)$. The latter stands for the category of
Harish-Chandra bimodules for $G$, i.e., those where the adjoint action of $\g$ integrates
to a $G$-action.
\end{itemize}

We remark that although in \cite{HC} only the case when $G$ is semisimple and simply connected
was considered, the general case is obtained from there in a straightforward way.

Namely, we have two maps $\I\mapsto \I^{\dagger}:\Id_{fin}(\Walg)\rightarrow \Id_{\Orb}(\U),
\J\mapsto \J_{\dagger}: \Id_{\Orb}(\U)\rightarrow \Id_{fin}(\Walg)$ having the
following properties.

\begin{Thm}\label{Thm:1.2.1}
\begin{itemize}
\item[(i)] $\I^\dagger\cap \Centr(\g)=\I\cap \Centr(\g)$.
\item[(ii)] $\J_\dagger\cap \Centr(\g)\supset \J\cap \Centr(\g)$.
\item[(iii)] $\codim_{\Walg}\J_\dagger=\mult_{\Orb}\U/\J$, where the right hand side denotes
the multiplicity of $\U/\J$ on $\Orb$.
\item[(iii)] $\J_{\dagger}$ is $Q$-stable for any $\J\in \Id_{\Orb}(\U)$.
\item[(iv)] If $\I$ is a $Q$-stable element of $\Id_{fin}(\Walg)$, then $\I=(\I^{\dagger})_{\dagger}$.
\item[(v)] $(\J_\dagger)^\dagger/\J\in \HC_{\partial\Orb}(\U)$ for any
$\J\in \Id_{\Orb}(\U)$.
\end{itemize}
\end{Thm}

Now we proceed to functors between the categories of  HC $\U$- and $\Walg$- bimodules.

Let us explain what we mean by an HC $\Walg$-bimodule.
Let $\Nil$ be a $\Walg$-bimodule. We say that $\Nil$ is HC if it is finitely generated and there is an increasing filtration $\F_i\Nil$  such that
\begin{itemize}\item[(i)] $[\KF_i \Walg, \F_j \Nil]\subset \F_{i+j-2}\Nil$
\item[(ii)] and $\gr \Nil$ is a finitely generated $\gr\Walg=\K[S]$-module.\end{itemize} Since $[\KF_i\Walg, \KF_j\Walg]\subset \KF_{i+j-2}\Walg$, we see that $\Walg$ itself is HC. Also any finite dimensional
bimodule is HC. In general, the filtration $\F_\bullet$ is bounded from below and any
$\F_i\Nil$ is finite dimensional.  An important series of infinite dimensional
HC $\Walg$-bimodules including $\Walg$ will be constructed in Subsection \ref{SUBSECTION_UL}. The category of HC $\Walg$-bimodules will be denoted by $\HC(\Walg)$.

By a {\it $Q$-equivariant}
HC bimodule we mean a $\Walg$-bimodule $\Nil$ equipped with
\begin{itemize}\item  a $Q$-action compatible with the $Q$-action
on $\Walg$ \item and a $Q$-stable filtration $\F_i\Nil$ as above such that the differential of the $Q$-action
coincides with the adjoint action of $\q\subset \Walg$
\end{itemize}
(the differential is defined because the filtration is $Q$-stable and any subspace $\F_i\Nil$ is finite dimensional).
The category of $Q$-equivariant HC bimodules will be denoted by $\HC^Q(\Walg)$.
Let $\HC^Q_{fin}(\Walg)$  denote the category of finite dimensional (and so automatically Harish-Chandra)
$Q$-equivariant  bimodules.

The tensor product over
$\Walg$ defines  monoidal structures on $\HC(\Walg)$ and $\HC^Q(\Walg)$ (the unit object is
$\Walg$ itself).

In \cite{HC} the first author has constructed a functor $\bullet_\dagger:\HC(\U)\rightarrow \HC^Q(\Walg)$.
We will need the construction in the present paper so let us recall it.

Pick $\M\in \HC(\U)$. Equip it with a {\it good filtration} $\F_i\M$. Recall that a filtration
$\F_i\M$ is said to be good if
\begin{itemize}
\item
it is compatible with the filtration $\F_i\U$ on $\U$, i.e., $\F_i\U\cdot\F_j\M,\F_j\M\cdot\F_i\U\subset \F_{i+j}\M$.
\item $\F_i\M$ is $\ad(\g)$-stable for all $i$.
\item $\gr\M$ is a finitely generated $S(\g)=\gr\U$-module.
\end{itemize}

Consider the Rees $\U_\hbar$-bimodule $\M_\hbar:=\bigoplus \F_i\M \hbar^i\subset \M[\hbar]$.
Then $\M^\wedge_\hbar:=\U^\wedge_\hbar\otimes_{\U_\hbar}\M_\hbar$ has a natural structure of
a $Q$-equivariant Harish-Chandra $\U^\wedge_\hbar$-bimodule in the sense
of \cite{HC}, Subsection 2.5. According to \cite{HC},
Proposition 3.3.1, $\M^\wedge_\hbar\cong \W^\wedge_\hbar\widehat{\otimes}_{\K[[\hbar]]}\Nil'_\hbar$,
where $\Nil'_\hbar$ is the space of all $m\in \M^\wedge_\hbar$ such that $vm=mv$ for all
$v\in V\subset \W^\wedge_\hbar\subset \W^\wedge_\hbar\widehat{\otimes}_{\K[[\hbar]]}\Walg^\wedge_\hbar$.
Then $\Nil'_\hbar$ carries a natural structure of a $Q$-equivariant HC $\Walg^\wedge_\hbar$-bimodule.
Let $\Nil_\hbar$ stand for the space of all $\K^\times$-finite vectors in $\Nil'_\hbar$
(a vector is said to be $\K^\times$-finite if it lies in a finite dimensional $\K^\times$-stable submodule).
Then $\Nil_\hbar$ is a $Q$-equivariant HC $\Walg_\hbar$-bimodule. Finally, set
$\M_\dagger:=\Nil_\hbar/(\hbar-1)\Nil_\hbar$. It was shown in \cite{HC}, Section 3.4,
that $\M_\dagger$ does not depend on the choice of a filtration on $\M$ and that $\M\mapsto\M_\dagger$
is a functor $\HC(\U)\rightarrow \HC^Q(\Walg)$.

The following result was obtained in \cite{HC}, Theorem 1.3.1, Theorem 4.1.1.

\begin{Thm}\label{Thm_dagger}
\begin{enumerate}
\item The functor $\M\mapsto \M_\dagger:\HC(\U)\rightarrow \HC^Q(\Walg)$ is exact.
Moreover, $\U_\dagger=\Walg$ and for an ideal $\J\subset \U$ its image under the functor
$\bullet_\dagger$ coincides with the ideal $\J_\dagger$ mentioned above.
\item $\bullet_\dagger$ maps $\HC_{\overline{\Orb}}(\U)$ to $\HC^Q_{fin}(\Walg)$. There is a functor $\Nil\mapsto \Nil^{\dagger}:\HC^Q_{fin}(\Walg)\rightarrow \HC_{\overline{\Orb}}(\U)$
right adjoint to $\M\mapsto \M_{\dagger}: \HC_{\overline{\Orb}}(\U)\rightarrow \HC^Q_{fin}(\Walg)$.
\item Let $\M\in \HC_{\overline{\Orb}}(\U)$. Then
$\dim\M_{\dagger}=\mult_{\overline{\Orb}}(\M)$, and the kernel and the cokernel of the natural
homomorphism $\M\rightarrow (\M_\dagger)^\dagger$ lie in $\HC_{\partial\Orb}(\U)$.
\item $\M\rightarrow \M_{\dagger}$ is a tensor functor.
\item $\operatorname{LAnn}(\M)_{\dagger}=\operatorname{LAnn}(\M_\dagger), \operatorname{RAnn}(\M)_{\dagger}=\operatorname{RAnn}(\M_\dagger)$ for any $\M\in \HC_{\overline{\Orb}}(\U)$.
\item Let $\M\in \HC(\U)$ and $\Nil\subset \M_{\dagger}$ be a $Q$-stable
subbimodule of finite codimension. Then $\Nil=\M'_\dagger$ for some
$\M'\subset \M$ with $\VA(\M/\M')=\overline{\Orb}$.
\item The functor $\M\mapsto \M_\dagger$ gives rise to an equivalence of
$\HC_{\Orb}(\U)$ and some full  subcategory in
$\HC_{fin}^Q(\Walg)$ closed under taking subquotients.
\end{enumerate}
\end{Thm}

%
%

\subsection{Category $\OCat$ for a W-algebra and the equivalence $\KFun$}\label{SUBSECTION_Ocat}
All results of this subsection can be found in \cite{Ocat}.

Here we will discuss a certain category of $\Walg$-modules. To define this category we fix a
torus $T\subset Q$ and a cocharacter $\theta:\K^\times\rightarrow T$. We assume that
$\theta$ is regular: i.e., for $\g_0:=\z_\g(\theta)$ we have $ \t:=\z(\g_0)$.
We remark that $e,h,f\in\g_0$.

We have an eigenspace decomposition $\Walg=\bigoplus_{\alpha\in \ZZ}\Walg_\alpha$ with respect to $\theta$. Set
$\Walg_{\geqslant 0}:=\sum_{\alpha\geqslant 0}\Walg_\alpha, \Walg_{>0}:=\sum_{\alpha>0}\Walg_\alpha,
\Walg_{\geqslant 0}^+:=\Walg_{\geqslant 0}\cap \Walg\Walg_{>0}, \Walg^0:=\Walg_{\geqslant 0}/\Walg_{\geqslant 0}^+$. Clearly, $\Walg_{\geqslant 0}$ is a subalgebra in $\Walg$, and $\Walg_{\geqslant 0}^+$ is an ideal in $\Walg_{\geqslant 0}$.

By definition, the category $\widetilde{\OCat}^{\t}$ for the pair $(\Walg,\theta)$ consists of all $\Walg$-modules
$N$ satisfying the following conditions:

\begin{itemize}
\item $N$ is finitely generated.
\item $\t$ acts on $N$ by diagonalizable endomorphisms.
\item $\Walg_{>0}$ acts on $N$ by locally nilpotent endomorphisms.
\end{itemize}

This category (in a somewhat different form) was first introduced by Brundan, Goodwin and Kleshchev
in \cite{BGK}.

The category $\widetilde{\OCat}^{\t}(\theta)$ has analogs of Verma modules. Namely,  take a finitely generated $\Walg^0$-module $N$ with diagonalizable action of $\t$. We can consider $N^0$ as a $\Walg_{\geqslant 0}$-module
by letting $\Walg_{\geqslant 0}^+$ act by zero. Then we have the Verma module $\Verm^\theta(N^0):=\Walg\otimes_{\Walg_{\geqslant 0}}N^0$. The module $\Verm^\theta(N^0)$ has a unique irreducible quotient provided $N^0$ is irreducible.
We denote this quotient by $L^\theta(N^0)$. 

The functor $N^0\mapsto \Verm^\theta(N^0)$ from the category $\Walg^0$-$\Mod^\t$ of finitely generated $\Walg^0$-modules with diagonalizable action of $\t$ to $\widetilde{\OCat}^\t$ has a right adjoint: $\FF:M\mapsto M^{\Walg_{>0}}$.

It turns out that the category $\widetilde{\OCat}^\t$ is equivalent to a certain category of
{\it generalized Whittaker} modules to be defined now.

Let $\g_{>0}\subset \g$ denote the sum of all eigenspaces  for $\ad\theta$ with
positive eigenvalues. Then $\g_{\geqslant 0}:=\g_0\oplus \g_{>0}$ is a
parabolic subalgebra of $\g$ and $\g_{>0}$ is its nilpotent radical.
Let $\underline{\m}\subset \g_0$ be the subalgebra defined  analogously to $\m\subset\g$,
Further, set $\widetilde{\m}:=\g_{>0}\oplus \underline{\m}, \widetilde{\m}_\chi:=\{\xi-\langle\chi,\xi\rangle,\xi\in \widetilde{\m}\}$. By a generalized Whittaker module we mean a $\U$-module $\M$ satisfying the following conditions:

\begin{itemize}
\item $\M$ is finitely generated.
\item $\t$ acts on $\M$ by diagonalizable endomorphisms.
\item $\widetilde{\m}_\chi$ acts on $\M$ by locally nilpotent endomorphisms.
\end{itemize}

The category of generalized Whittaker modules will be denoted by $\widetilde{\Wh}^{\t}(e,\theta)$.


The category $\widetilde{\Wh}^{\t}(e,\theta)$ also contains analogs of Verma modules defined as follows.
Let $\underline{\Walg}$ denote the W-algebra constructed for the pair $(\g_0,e)$ and let $\Sk_0$
be the Skryabin equivalence for $\g_0,e$. Pick a finitely generated $\underline{\Walg}$-module
$\underline{N}$ with diagonalizable action of $\t$. Set $\Verm^{e,\theta}(\underline{N}):=\U\otimes_{U(\g_{\geqslant 0})}\Sk_\lf(\underline{N})$. This module has a unique irreducible quotient $L^{e,\theta}(\underline{N})$.
As before, the functor $\Verm^{e,\theta}$ has the right adjoint $\GF:M\mapsto M^{\widetilde{\m}_\chi}$.

The main result of \cite{Ocat} is the following theorem.

\begin{Thm}\label{Thm_Ocat}
There is an equivalence $\KFun:\widetilde{\Wh}^\t(e,\theta)\rightarrow
\widetilde{\OCat}^\theta(\theta)$ of abelian categories and an isomorphism $\Psi:\underline{\Walg}\rightarrow
\Walg^0$ satisfying the following conditions:
\begin{enumerate}
\item $\Ann_{\Walg}(\KFun(M))^{\dagger}=\Ann_{\U}(M)$ for any $M\in \widetilde{\Wh}(e,\theta)$.
\item The functors $\Psi^*\circ\FF\circ \KFun$ and $\GF$ from $\widetilde{\Wh}^\t(e,\theta)$
to $\underline{\Walg}$-$\Mod^\t$ (the category of $\underline{\Walg}$-modules
with diagonalizable $\t$-action) are isomorphic. Here $\Psi^*$ denotes the pull-back functor
between the categories of modules induced by $\Psi$.
\item The functors $\KFun\circ\Verm^{e,\theta},\Verm^\theta\circ (\Psi^{-1})^*$ from $\underline{\Walg}\text{-}\Mod^\t$
to $\widetilde{\OCat}^\t(\theta)$ are isomorphic.
\end{enumerate}
\end{Thm}

We will need a construction of the equivalence $\KFun$. This equivalence is a push-forward
with respect to an isomorphism of an appropriate topological algebras that we are going
to recall now.

We set $\U^\wedge:=\varprojlim_{n\rightarrow \infty }\U/\U\widetilde{\m}_\chi^n$. The category
$\widetilde{\Wh}^{\t}(e,\theta)$ is nothing else but the category of finitely generated topological $\U^\wedge$-modules
with discrete topology   such that the action of $\t\subset\U^\wedge$ is diagonalizable.

The subspace $\widetilde{\m}\cap V$ is lagrangian in $V$. So we can form the completion
$$\W^\wedge:=\varprojlim_{n\rightarrow \infty} \W/\W(\widetilde{\m}\cap V)^n, $$ where
$\W=\W_\hbar/(\hbar-1)$ stands for the Weyl algebra of $V$.

Also we need a completion $\Walg^\wedge:=\varprojlim_{n\rightarrow \infty}\Walg/\Walg\Walg_{>0}^n$
of $\Walg$. We remark that  $\widetilde{\OCat}^\t(\theta)$ coincides with the category of finitely generated topological $\Walg^\wedge$-modules with discrete topology such that the action of $\t$ is diagonalizable. Finally, set $\W(\Walg)^\wedge:=\W^\wedge\widehat{\otimes}\Walg^\wedge$ (in \cite{Ocat} we used a different construction
of $\W(\Walg)^\wedge$ -- using an appropriate completion of $\W(\Walg)$ -- but it is easy to
see that the two constructions are equivalent). There is an equivalence of the categories
of topological $\Walg^\wedge$- and $\W(\Walg)^\wedge$-modules with discrete topologies.
Namely, we send a $\Walg^\wedge$-module $N$ to $\K[\widetilde{\m}\cap V]\otimes N$ (recall that
$\W$ can be thought as the algebra $\mathcal{D}(\widetilde{\m}\cap V)$ of differential operators on  the lagrangian subspace $\widetilde{\m}\cap V$, so $\K[\widetilde{\m}\cap V]$ is the tautological
module over this algebra). A quasiinverse equivalence sends an $\W(\Walg)^\wedge$-module $M$  to $M^{\widetilde{\m}\cap V}$.

So to establish an equivalence of $\widetilde{\Wh}^\t(e,\theta)$ and $\widetilde{\OCat}^\t(\theta)$
it is enough to produce an isomorphism $\Phi:\U^\wedge\rightarrow \W(\Walg)^\wedge$
intertwining the embeddings of $\t$. Such an isomorphism was constructed in \cite{Ocat}. Its properties imply conditions
(i)-(iii) of Theorem \ref{Thm_Ocat}. Let us recall this construction given in \cite{Ocat},
Section 5.

The torus $\K^\times\times T$ acts on $\U^\wedge_\hbar,\Walg^\wedge_\hbar,\W^\wedge_\hbar$,
where the $\K^\times$-action is Kazhdan, and $T$ acts as a subgroup of $Q$. Embed $\K^\times$ into $\K^\times\times T$ so that
the differential of the embedding maps $1$ to $(1,-n\theta)$, where $n$ is sufficiently
large. We remark that $\widetilde{\m}$  contains  the sum of eigenspaces
for $\K^\times$ corresponding to characters $t\mapsto t^i$ with $i\leqslant 0$.

Let $\K^\times$ act on the algebras in consideration via this embedding  (this  action  will be called a
{\it twisted Kazhdan action}). Consider the subalgebras $(\U^\wedge_\hbar)_{\K^\times-l.f.},
\W^\wedge_\hbar(\Walg^\wedge_\hbar)_{\K^\times-l.f.}$ consisting of all $\K^\times$-finite vectors. Let
$\U^\heartsuit, \W(\Walg)^\heartsuit$ denote the quotients of these algebras by $\hbar-1$. Then $\Phi_\hbar$
induces an isomorphism $\U^\heartsuit\rightarrow \W(\Walg)^\heartsuit$. There are
natural embeddings $\U^\heartsuit\hookrightarrow \U^\wedge, \W(\Walg)^\heartsuit\hookrightarrow \W(\Walg)^\wedge$
and $\Phi$ extends uniquely to an isomorphism $\U^\wedge\rightarrow \W(\Walg)^\wedge$ of topological
algebras. We remark that $\Phi$ also induces an isomorphism $\Psi:\underline{\Walg}\rightarrow \Walg^0$.

\section{Further study of functor $\bullet_\dagger$}\label{SECTION_Walg_further}
\subsection{Main result}
In this section we will study some further properties of the functor $\bullet_\dagger: \HC(\U)\rightarrow \HC^Q(\Walg)$.
We fix some connected reductive algebraic group $G$ and consider Harish-Chandra bimodules related to that group.
Our main result is the following theorem.
\begin{Thm}\label{Thm_main}
Let $N_1,N_2$ be irreducible finite dimensional $\Walg$-modules with integral central characters,
whose difference lies in $P^+$ (of course, for a semisimple simply connected group the last
condition is vacuous). Then there exists an irreducible object $\M\in \HC_{\Orb}(\U)$  such that $\Hom(N_1,N_2)$
is a direct summand of $\M_\dagger$.
\end{Thm}
We remark that for any irreducible $\M\in \HC_{\Orb}(\U)$ its image under $\bullet_\dagger$
is a simple object in $\HC^Q_{fin}(\Walg)$ and hence a semisimple finite dimensional
$\Walg$-bimodule. Since any simple object in $\HC_{fin}(\Walg)$ has the form $\Hom(N_1,N_2)$
for some irreducibles $N_1,N_2$, the claim of Theorem \ref{Thm_main} makes sense.

Let us explain the scheme of the proof of Theorem \ref{Thm_main}.
First, Subsection \ref{SUBSECTION_UL}, we will give some very implicit description of
the image of $\HC_{\Orb}(\U)$ under $\bullet_\dagger$. Then we will
examine a relationship between the functors $\bullet_\dagger$ and $\mathcal{K}$,
Subsection \ref{SUBSECTION_dagger_K}. Next, Subsection \ref{SUBSECTION_equivalence},  we will introduce a certain equivalence relation on the set $\Irr_{fin}(\Walg)$ of finite dimensional irreducible $\Walg$-modules.
We will see that Theorem \ref{Thm_main} means that any two irreducibles with integral
central characters are equivalent. Theorem \ref{Thm_main} is then proved in the next
two subsections: in Subsection \ref{SUBSECTION_even} we prove it in the case
when $e$ is even, and in Subsection \ref{SUBSECTION_reduction} we reduce a general
case to that one using the results of Subsection \ref{SUBSECTION_dagger_K}.

Before proving Theorem \ref{Thm_main} we will establish several easier
claims. First, in Subsection \ref{SUBSECTION_inthoms} we will show that
$\bullet_\dagger$ intertwines the internal $\Hom$ functors. Using this
we will present an easy proof that $^{\Lambda}\!\JCat^{\Lambda}_{\Orb}$ is a multi-fusion
category. Also we will show that this multi-fusion category is indecomposable.
This will be done in Subsection \ref{SUBSECTION_J_mult_fus}. In Subsection \ref{SUBSECTION_W_trans}
we will prove the inclusion (up to conjugacy) $\bH_\alpha\subset \bH^\lambda_\alpha$
for a dominant weight $\lambda$ and a left cell $\alpha$ that are compatible with each other.

\subsection{$\bullet_\dagger$ vs internal $\Hom$}\label{SUBSECTION_inthoms}
Recall the internal Hom in $\HC(\U)$, Subsection \ref{SUBSECTION_HC}.
On the other hand, let $\Nil_1,\Nil_2\in \HC^Q(\Walg)$. Let $\underline{\Hom}(\Nil_1,\Nil_2)$
denote the space of right $\Walg$-module homomorphisms. Then $\underline{\Hom}(\Nil_1,\Nil_2)$
has a natural structure of a  $\Walg$-bimodule and also a $Q$-action.

The goal of this subsection is to prove that $\bullet_\dagger$ intertwines the $\underline{\Hom}$-bifunctors.
But first we prove the following lemma.

\begin{Lem}\label{Lem:int_hom_W}
For any $\Nil_1,\Nil_2\in \HC^Q(\Walg)$ the bimodule $\underline{\Hom}(\Nil_1,\Nil_2)$ is also in $HC^Q(\Walg)$.
\end{Lem}
\begin{proof}
The proof is pretty standard.
We need to equip $\Nil:=\underline{\Hom}(\Nil_1,\Nil_2)$ with a $Q$-stable filtration $\F_i\Nil$
having the properties indicated in Subsection \ref{SUBSECTION_primitive}.
Pick $Q$-stable filtrations $\F_i\Nil_1,\F_i\Nil_2$ satisfying the  conditions analogous
to (1),(2) in Subsection \ref{SUBSECTION_primitive}. Then define $\F_i\Nil$ to be the set of all maps $\varphi\in \Nil$
such that $\varphi(\F_j\Nil_1)\subset \F_{i+j}\Nil_2$ for all $j$. It follows
from \cite{HC}, Lemma 2.5.1, that $\Nil_1$
is a finitely generated right $\Walg$-module. This easily implies that the filtration
$\KF_i\Nil$ is exhaustive. The inclusion $[\KF_i\Walg,\F_j\Nil]\subset \F_{i+j-2}\Nil$ is checked in a straightforward way.
To prove that $\gr\Nil$ is finitely generated
we note that $\gr\Nil$ is naturally embedded into $\Hom_{\K[S]}(\gr \Nil_1,\gr \Nil_2)$.
Since both $\gr\Nil_1,\gr\Nil_2$ are finitely generated $\K[S]$-modules, we see that
$\Hom_{\K[S]}(\gr\Nil_1,\gr\Nil_2)$ is finitely generated. Hence (2).
 \end{proof}

\begin{Prop}\label{Prop:int_hom}
The bifunctors $\underline{\Hom}(\bullet,\bullet)_\dagger, \underline{\Hom}(\bullet_\dagger,\bullet_\dagger):
\HC(\U)\times \HC(\U)\rightarrow \HC^Q(\Walg)$ are isomorphic.
\end{Prop}
\begin{proof}
The proof is in several steps, corresponding to the steps in the construction of $\bullet_\dagger$.

{\it Step 1.}
Pick $\M^1,\M^2\in \HC(\U)$ and pick good filtrations  $\F_i\M^1,\F_i\M^2$. Let $\M^1_\hbar,\M^2_\hbar$
stand for the Rees bimodules. Equip $\M:=\underline{\Hom}(\M^1,\M^2)$ with the filtration
$\F_i\M$ analogous to the filtration in the proof of Lemma \ref{Lem:int_hom_W}. Then, similarly to
the proof of Lemma \ref{Lem:int_hom_W}, $\F_i\M$
is a good filtration.

We claim that the Rees bimodule $\M_\hbar$ is naturally identified with the space
$\Hom_{\U_\hbar^{opp}}(\M^1_\hbar,\M^2_\hbar)$ of homomorphisms of right $\U_\hbar$-modules.
To show this we need to verify that $\F_i\M$ is identified with the space of
maps of degree $i$ in $\Hom_{\U_\hbar^{opp}}(\M^1_\hbar,\M^2_\hbar)$, i.e., the space
of right $\U_\hbar$-module homomorphisms $\varphi$ mapping $\F_j\M^1\hbar^j$ to $ \F_{i+j}\M^2\hbar^{i+j}$
for all $j$. Let $\varphi':\M^1\rightarrow \M^2$ be a map coinciding with the composition $\F_j\M^1\cong \F_j\M^1\hbar^j\rightarrow \F_{i+j}\M^2\hbar^{i+j}\cong\F_{i+j}\M^2$ on $\F_j\M^1$. Since $\varphi$ is $\K[\hbar]$-linear, we see that $\varphi'$ is well-defined. It is easy to see that $\varphi\mapsto \varphi'$ defines a bijection between the space of maps of degree $i$ in $\Hom_{\U_\hbar^{opp}}(\M^1_\hbar,\M^2_\hbar)$ and
$\F_i\M$.

{\it Step 2.}
We claim that $\M^\wedge_\hbar$ is naturally identified with $\Hom_{\U_\hbar^{\wedge opp}}(\M^{1\wedge}_\hbar, \M^{2\wedge}_\hbar)$.
Since $\M^1_\hbar,\M^2_\hbar$ are   finitely generated as right $\U_\hbar$-modules
and $\U^\wedge_\hbar$ is a flat right $\U_\hbar$-module, we see that $$\M_\hbar^\wedge=\U_\hbar^{\wedge}\otimes_{\U_\hbar}\Hom_{\U_\hbar^{opp}}(\M^1_\hbar, \M^{2}_\hbar)=\Hom_{\U_\hbar^{opp}}(\M^1_\hbar,\M^{2\wedge}_\hbar).$$
Our claim now follows from the fact that $\M^{1\wedge}_\hbar$ is naturally identified with
$\U^\wedge_\hbar\otimes_{\U_\hbar}\M_\hbar^1$ (compare with \cite{HC}, Proposition 2.4.1, (1)).

{\it Step 3.} Now let $\Nil'^{1}_\hbar,\Nil'^{2}_\hbar$ be the spaces of $\ad V$-invariants
in $\M^{1\wedge}_\hbar,\M^{2\wedge}_\hbar$, respectively. Then $\M^{i\wedge}_\hbar=\W^\wedge_\hbar\widehat{\otimes}_{\K[[\hbar]]}\Nil'^{i}_\hbar$.
It is easy to see that $$\Hom_{\U^{\wedge opp}_\hbar}(\M^{1\wedge}_\hbar,\M^{2\wedge}_\hbar)=
\W^\wedge_\hbar\widehat{\otimes}_{\K[[\hbar]]}\Hom_{\Walg^{\wedge opp}_\hbar}(\Nil'^1_\hbar,\Nil'^2_\hbar).$$

{\it Step 4.} Also it is easy to see that the space of $\K^\times$-finite vectors in
$\Hom_{\Walg^{\wedge opp}_\hbar}(\Nil'^1_\hbar,\Nil'^2_\hbar)$ is naturally identified
with $\Hom_{\Walg_\hbar^{opp}}(\Nil^1_\hbar,\Nil^2_\hbar)$, where $\Nil^i_\hbar$ stands for the
space of $\K^\times$-finite vectors in $\Nil'^i_\hbar$.

{\it Step 5.} Finally, set $\Nil^i:=\Nil^i_\hbar/(\hbar-1)\Nil^i_\hbar$, $\Nil:=\underline{\Hom}(\Nil^1,\Nil^2)$.
Of course, $\Nil^i=\M^i_\dagger$.
Let $\Nil_\hbar$ stand for the Rees bimodule of $\Nil$. Analogously to step 1, $\Nil_\hbar$ is naturally
identified with $\Hom_{\Walg_\hbar^{opp}}(\Nil^1_\hbar,\Nil^2_\hbar)$. In particular,
$$\underline{\Hom}(\M^1_\dagger,\M^2_\dagger)=\Nil\cong \Hom_{\Walg_\hbar^{opp}}(\Nil^1_\hbar,\Nil^2_\hbar)/(\hbar-1)\Hom_{\Walg_\hbar^{opp}}(\Nil^1_\hbar,\Nil^2_\hbar)=
\underline{\Hom}(\M^1,\M^2)_\dagger.$$
\end{proof}

\subsection{$^{\Lambda}\!\JCat^{\Lambda}_{\Orb}$ as a multi-fusion category}\label{SUBSECTION_J_mult_fus}
From now on the orbit $\Orb$ is supposed to be special. We fix a finite subset $\Lambda\subset P^+$.
Recall, the very end of Subsection \ref{SUBSECTION_HC}, that the category $^{\Lambda}\!\JCat^{\Lambda}_{\Orb}$ has a tensor product
bifunctor, a unit object $\be$, and also has a duality functor $\bullet^*$.

\begin{Lem}\label{Lem:J_MF}
With respect to these data $^{\Lambda}\!\JCat^{\Lambda}_{\Orb}$ is a multi-fusion category.
\end{Lem}
\begin{proof}
The proof is based on Theorem \ref{Thm:1.2.1}.
The functor $\bullet_{\dagger}$ embeds $^{\Lambda}\!\JCat^{\Lambda}_{\Orb}$ as a full subcategory
into $\Bimod^Q(\Walg_\Lambda^{ss})$. This subcategory is closed under taking direct summands.
Here $\Walg^{ss}_\Lambda$ is the quotient of $\Walg$ by the intersection of all primitive ideals of
finite codimension with central characters in $\Lambda$. Also $\bullet_\dagger$ maps
$\be$ to the unit object  $\Walg_\Lambda^{ss}$ of $\Bimod^Q(\Walg_\Lambda)$,
intertwines the tensor products, the internal homs -- and hence the duality functors.   Now the claim of the
lemma follows from an easy fact that $\Bimod^Q(\Walg_\Lambda^{ss})$ itself is  multi-fusion.
\end{proof}

The following claim seems to be well-known (at least for $\Lambda=\{\rho\}$, see \cite{orange}, 12.16)
but we are going to provide its proof for reader's convenience.

\begin{Prop}\label{Prop:J_indecomp}
The multi-fusion category $^{\Lambda}\!\JCat^{\Lambda}_{\Orb}$ is indecomposable.
\end{Prop}
\begin{proof}
We need to show that there is no decomposition $\be=e\oplus f$ in $^{\Lambda}\!\JCat^{\Lambda}_{\Orb}$
such that $e\,^{\Lambda}\!\JCat^{\Lambda}_{\Orb} f=0$ (and then automatically, thanks to duality,
$e\,^{\Lambda}\!\JCat^{\Lambda}_{\Orb} f$). The object $\be$ is represented by
$\sum_{\J\in \Prim_{\Orb}(\U_\Lambda)} \U/\J$. We remark that each $\U/\J$ is a simple
direct summand of $\be$ (because there is no inclusion between the elements of $\Prim_{\Orb}(\U)$).
So our assumption implies that there is a partition $P_e\sqcup P_f= \Prim_{\Orb}(\U_\Lambda)$
such that the left and right annihilators of each simple $\M\in \,^{\Lambda}\!\HC^\Lambda(\U)$
with $\VA(\M)=\overline{\Orb}$ either both lie in $P_e$ or both lie in $P_f$.

Assume, first, that $\Lambda$ contains a regular element, say $\varrho$.

First of all, let us notice that $\Prim_{\Orb}(\U_\varrho)$ is contained in either $P_e$ or
$P_f$. There are several possible proofs of this. For instance, consider the equivalence
relations $\sim,\sim_L,\sim_R$ on $W$ indicating whether two elements lie in the same
two-sided, left, or right cell, respectively (this will be recalled in Section \ref{SECTION_cell} below).
Then it is known that $\sim$ is generated by
$\sim_L,\sim_R$. Now take a simple HC bimodule $\BG(L(w\varrho))$ with $w\in \dcell$ such that its left (and
hence right) annihilator lies in $P_e$. The set of all $w$ with this property is closed under
both $\sim_L,\sim_R$ and hence coincides with $\dcell$. So $\Prim_{\Orb}(\U_\varrho)\subset P_e$.

To show that $P_e=\Prim_{\Orb}(\U_\Lambda)$ we recall that the right annihilator
of $\BG(L(w\lambda))$, where $w$ and $\lambda$ are compatible, is $\J(w\lambda)$, while the right annihilator is $\J(w^{-1}\varrho)$
and hence lies in $P_e$. But the ideals of the form $\J(w\lambda)$ exhaust $\Prim_{\Orb}(\U_\Lambda)$.

Now suppose that $\Lambda$ contains no regular elements. Set $\Lambda'=\Lambda\sqcup\{\varrho\}$. The category
$^{\Lambda}\!\JCat^{\Lambda}_{\Orb}$ can be realized as $e'\,^{\Lambda'}\!\JCat^{\Lambda'}_{\Orb}e'\subset
\,^{\Lambda'}\!\JCat^{\Lambda'}_{\Orb}$.
Such a subcategory is itself indecomposable, see Remark \ref{Rem:idemp_nonzero}.
\end{proof}

As we have mentioned in the proof,  the objects $\U/\J, \J\in \Prim_{\Orb}(\U_\Lambda)$
are simple direct summands of $\be$. The object corresponding to a dominant weight $\lambda$
and a left cell $\sigma$ will be denoted by $e^\lambda_\sigma$.

The following corollary is standard and well-known: it can be deduced from \cite[3.8]{GJ} and the standard
 properties of cells, see, e.g., \cite[Corollary 12.16]{orange}. 
%
Also it follows from Proposition \ref{Prop:J_indecomp} and Remark \ref{Rem:idemp_nonzero}.

\begin{Cor}\label{Cor:indecomp}
For any two primitive ideals $\J_1,\J_2\in \Prim_{\Orb}(\U_{P^+})$ there is a simple HC
bimodule $\M$ with $\J_1=\LAnn_{\U}(\M), \J_2=\RAnn_{\U}(\M)$.
\end{Cor}


\subsection{Translations to/from the walls for $\Walg$-modules}\label{SUBSECTION_W_trans}
The goal of this subsection is to prove the following proposition.
\begin{Prop}\label{Prop:W_translation}
Let $\cell$ be a left cell, and $\lambda$ a dominant weight compatible with $\cell$.
Then $\bH_\cell\subset \bH_\cell^\lambda$ with the equality when $\lambda$ is strictly
dominant.
\end{Prop}
\begin{proof}
Recall  the Harish-Chandra bimodules $\mathcal{T}_\bullet^\bullet(\bullet)$ that appeared in
Subsection \ref{SUBSECTION_shifts}.
Pick an irreducible $\Walg$-module $N$ such that the primitive ideal corresponding to $N$, i.e., $\Ann_{\Walg}(N)^\dagger$,
has central character $\lambda$ and corresponds to the left cell $\sigma$. Then for sufficiently large $k$ we have $$(\mathcal{T}_\rho^\lambda(k))_\dagger\otimes_{\Walg}
(\mathcal{T}_\lambda^\rho(1))_\dagger\otimes_{\Walg}N=(\mathcal{T}_\rho^\lambda(k)\otimes_{\U} \mathcal{T}_\lambda^\rho(1))_\dagger\otimes_{\Walg}N=N^{\oplus |W_\lambda|}.$$
Set $\widetilde{N}=(\mathcal{T}_{\lambda}^\rho(1))_\dagger\otimes_{\Walg}N$ and let $N'$ be an irreducible
subquotient of $\widetilde{N}$ with $(\mathcal{T}_\rho^\lambda(k))_\dagger\otimes_{\Walg}N'\neq 0$.
Then $N'\in Y_{\cell}$.
Let $Q_0$ denote the stabilizer
of $N'$ under the action of $Q$. Since the bimodule $(T_{\rho}^\lambda(k))_\dagger$ is $Q$-equivariant,
we see that $(T_{\rho}^\lambda(k))_\dagger\otimes_{\Walg}N'$ is stable under the action of $Q_0$. But
$(T_{\rho}^\lambda(k))_\dagger\otimes_{\Walg}N'$ is just the direct sum of several copies of $N$
and so $N$ is $Q_0$-stable. This proves the claim.
\end{proof}

\subsection{The image of $\bullet_\dagger$}\label{SUBSECTION_UL}
We are starting the proof of Theorem \ref{Thm_main}.
The goal of this subsection is to give a (very implicit) description
of the image of $\HC_{\overline{\Orb}}(\U)$ under $\bullet_\dagger$. The main result
is Proposition \ref{Prop:11.2} below.

Let $L$ be a finite dimensional $G$-module. Then the space $\U_L:=L\otimes\U$ has a natural structure of
a Harish-Chandra bimodule: the left and right products with elements of $\g$ are defined by
$$\xi\cdot l\otimes u=(\xi. l)\otimes u+l\otimes \xi u, (l\otimes u)\cdot\xi=l\otimes u\xi.$$

Let us describe the $Q$-equivariant Harish-Chandra $\Walg$-bimodule
$\U_{L\dagger}$.

Set $\Walg_L:=\Hom_G(L^*,\widetilde{\Walg})$. The $Q$-action
and the bimodule structure on $\Walg_L$
are induced  from those  on $\widetilde{\Walg}$.

\begin{Lem}\label{Lem:10.2}
$\Walg_L$ is a $Q$-equivariant HC $\Walg$-bimodule. Moreover, there is a $Q$-equivariant  isomorphism $\Walg_L\cong L\otimes\Walg$ of  right $\Walg$-modules.
\end{Lem}
\begin{proof}
A required filtration on $\Walg_L$ is induced from the filtration on $\widetilde{\Walg}$:
the associated graded is $\Hom_G(L^*, \K[X])$. This is a finitely generated $\K[S]=\K[X]^G$-module.
The second claim of the lemma follows from Lemma \ref{Lem:1.5.1}.
\end{proof}

\begin{Prop}\label{Prop:ULdagger}
$\U_{L\dagger}\cong \Walg_L$.
\end{Prop}
\begin{proof}
Recall the isomorphism $$\Phi_\hbar:\K[T^*G]^\wedge_{Gx}[[\hbar]]\rightarrow \W^\wedge_\hbar \widehat{\otimes}_{\K[[\hbar]]}\K[X]^\wedge_{Gx}[[\hbar]]$$
from Proposition \ref{Prop:10.1}. It induces the isomorphism
$$\Phi_\hbar:\Hom_G(L^*,\K[T^*G]^\wedge_{Gx}[[\hbar]])\rightarrow \Hom_G(L^*,
\W^\wedge_\hbar \widehat{\otimes}_{\K[[\hbar]]}\K[X]^\wedge_{Gx}[[\hbar]])$$
of $Q$-equivariant $\U^\wedge_\hbar$-bimodules. Note that $\Hom_G(L^*, \K[T^*G][\hbar])$ is precisely
$R_\hbar(\U_L)$, where the filtration on $\U_L$ is given by $\F_i\U_L:=\F_i\U\otimes L$. Also note that  $$\Hom_G(L^*,\K[T^*G]^\wedge_{Gx}[[\hbar]])\cong R_\hbar(\U_L)^\wedge.$$ On the other hand,
\begin{align*}&\Hom_G(L^*,\W^\wedge_\hbar \widehat{\otimes}_{\K[[\hbar]]}\K[X]^\wedge_{Gx}[[\hbar]])=
\Hom_G(L^*,\K[X]^\wedge_{Gx}[[\hbar]])\widehat{\otimes}_{\K[[\hbar]]}\W^\wedge_\hbar.
\end{align*}
Again, $\Hom_G(L^*, \K[X]^{\wedge}_{Gx}[[\hbar]])$ is the completion of $\Hom_G(L^*,\K[X][\hbar])$.
So $\Hom_G(L^*,\K[X][\hbar])$  coincides with the space of $\K^\times$-finite vectors
in $\Hom_G(L^*, \K[X]^{\wedge}_{Gx}[[\hbar]])$ by \cite{HC}, Proposition 3.3.1.
The quotient of $\Hom_G(L^*, \K[X][\hbar])$ modulo $\hbar-1$ is nothing else
but $\Walg_L$. Now the claim of the proposition follows directly from the construction of $\bullet_\dagger$.
\end{proof}

Now we are ready to give some description of $\HC_{\overline{\Orb}}(\U)_\dagger$.

\begin{Prop}\label{Prop:11.2}
Let $\Nil\in \HC_{fin}^Q(\Walg)$. Then the following conditions are equivalent:
\begin{enumerate}
\item There is $\M\in \HC_{\overline{\Orb}}(\U)$ with $\M_{\dagger}\cong \Nil$.
\item $\Nil$ is a quotient of $\Walg_{L}$ for some finite dimensional $G$-module
$L$.
\item $\Nil$ is a $Q$-stable sub-bimodule in $\Walg_{L}/R$, where $L$ is a finite dimensional
$G$-module and $R$ is a $Q$-stable sub-bimodule in $\Walg_{L}$ of finite codimension.
\end{enumerate}
\end{Prop}
\begin{proof}
$(1)\Rightarrow (2)$: It is easy to show that $\M$ is a quotient of $\U_L$ for some
finite dimensional $G$-module $L$. Since $\bullet_\dagger: \HC(\U)\rightarrow \HC^Q(\Walg)$ is an exact functor,
(2) follows from Proposition \ref{Prop:ULdagger}.

$(2)\Rightarrow (3)$: This is tautological.

$(3)\Rightarrow (1)$: This follows from assertion (6) of Theorem \ref{Thm_dagger}.
\end{proof}

\subsection{$\bullet_\dagger$ vs $\KFun$}\label{SUBSECTION_dagger_K}
In this subsection we are going to obtain the following result that seems to be of independent
interest. We preserve the notation of Subsection \ref{SUBSECTION_Ocat}.

\begin{Thm}\label{Thm_tensor}
The bifunctors   $(X,M)\mapsto \KFun(X\otimes M), X_\dagger\otimes \KFun(M)$ from $\HC(\U)\times \widetilde{\Wh}^\t(e,\theta)$ to $\widetilde{\OCat}^\t(\theta)$ are isomorphic.
\end{Thm}

We will give a proof after a series of lemmas.

For $X\in \HC(\U)$ we set $X^\wedge:=\varprojlim_{n\rightarrow \infty} X/X\widetilde{\m}_\chi^n$. Since
$\widetilde{\m}$ consists of nilpotent elements, we see that $X^\wedge$ has a natural structure of a $\U^\wedge$-bimodule (compare with \cite{Wquant}, Subsection 3.2, the construction of $\U^\wedge$). Moreover, $X^{\wedge}$ becomes a topological $\U^\wedge\widehat{\otimes}\U^{\wedge,opp}$-module (the word ``topological'' here means that
the structure action map is continuous). Also we remark that the topology on $X$ is complete
and separated.

\begin{Lem}\label{Lem:11.1}
Let $X\in \HC(\U), M\in \widetilde{\Wh}^\t(e,\theta)$.
\begin{enumerate}
\item The natural map $\U^\wedge\otimes_\U X\rightarrow X^\wedge$ is an isomorphism.
\item The natural map $X\otimes_\U M\rightarrow X^\wedge\otimes_{\U^\wedge} M$ is an isomorphism.
\end{enumerate}
\end{Lem}
We remark that, since the topology on $M$ is discrete, $X^\wedge\otimes_{\U^\wedge} M$ is the same
as $X^\wedge\widehat{\otimes}_{\U^\wedge} M$.
\begin{proof}
Let us check (1) for $X=\U_L$, where $L$ is a  finite dimensional $G$-module. Here the assertion boils down to the claim that
the filtrations $L\otimes\widetilde{\m}_\chi^k$ and $\widetilde{\m}_\chi^k \U_L$ on
$\U_L=L\otimes\U$ are compatible. This easily follows from the fact that $\widetilde{\m}$ acts on
$L$ by nilpotent endomorphisms. Also for $X=\U_L$ (2) is clear.

In general, $X$ has a resolution consisting of modules of the form $\U_{L_i}$, where
$L_i$ is a finite dimensional $G$-module. Since all functors in consideration are
right-exact, both assertions follow from the 5-lemma.
\end{proof}

Recall the isomorphism $\Phi:\U^\wedge\xrightarrow{\sim} \W(\Walg)^\wedge$.
The $\W(\Walg)^\wedge$-bimodule $\Phi_*(X^\wedge)$ is complete in the $\Phi(\widetilde{\m}_\chi)=[(\widetilde{\m}\cap V)\otimes 1+1\otimes \Walg_{>0}]$-adic topology.
Therefore the map \begin{equation}\label{eq:11.1}\W^\wedge\widehat{\otimes} \Phi_*(X^\wedge)^{\ad V}\rightarrow \Phi_*(X^\wedge)\end{equation} is well-defined.

\begin{Prop}\label{Lem:11.2}
The map (\ref{eq:11.1}) is an isomorphism for any $X\in \HC(\U)$.
\end{Prop}
In  the proof we will need the following technical lemma.
\begin{Lem}\label{Lem:11.2.1}
Let $V$ be a symplectic vector space, $U\subset V$ be a lagrangian subspace, $\W$ be the Weyl algebra
of $V$, $\W^\wedge:=\varprojlim \W/\W U^k$. Further, let $Y$ be a topological $\W^\wedge\widehat{\otimes}\W^{\wedge,opp}$-module
such that the topology on $Y$ is complete and separated.
Then a natural map $\W^\wedge\widehat{\otimes}Y^{\ad V}\rightarrow Y$ is an isomorphism.
\end{Lem}
\begin{proof}
First of all, we note that all $\W^\wedge$-sub-bimodules
in  $\W^{\wedge}\widehat{\otimes}Y^{\ad V}$ have the form $\W^\wedge\widehat{\otimes}Y_0$ for $Y_0\subset Y^{\ad V}$. This can be
proved following the lines of the proof of Lemma 3.4.3 in \cite{Wquant}. It follows that the map (\ref{eq:11.1})
is injective.

It remains to prove that the map (\ref{eq:11.1}) is surjective. Let $q_1,\ldots, q_k$ denote a basis
in the lagrangian subspace $U$. Choose a complimentary lagrangian subspace
$U^*\subset V$ and let $p_1,\ldots,p_k$ denote the dual basis to $q_1,\ldots, q_k$.
So we have $[q_i,p_j]=\delta_{ij}$. Let us reduce the proof to the case $k=1$.

Consider the Weyl algebras $\W_0$ of the span $V_0$ of $p_1,q_1$ and $\W^0$ of the span $V^0$ of $p_2,\ldots,p_k,$ $q_2,\ldots,q_k$.
Let $\W_0^\wedge,\W^{0\wedge}$ be their completions with respect to the subspaces spanned
by $q_1$ and by $q_2,\ldots,q_k$. Then, of course, $\W^\wedge= \W_0^\wedge\widehat{\otimes}\W^{0\wedge}$.
Suppose now that the natural map $\W_0^\wedge\widehat{\otimes} Y^{\ad V_0}\rightarrow Y$
is bijective. The space $Y^{\ad V_0}$ is a closed $\W^0$-subbimodule of $Y$ and, in particular, a
$\W^{0\wedge}\widehat{\otimes}\W^{0\wedge opp}$-module. Also the topology on $Y^{\ad V_0}$
induced from $Y$ is again complete and separated. So it is enough
to show the claim of the lemma for $k=1$. Below we write $p,q$ instead of $p_1,q_1$.
We remark that $\W_0^\wedge q^i$ form a fundamental system of neighborhoods of $0$
in $\W_0^\wedge$.

Define the following two linear operators on $Y$:
\begin{equation}\label{eq:11.2}
\begin{split}
&\alpha(y):= \sum_{k=0}^\infty \frac{1}{k!}\ad(p)^k(y)q^k,\\
&\beta(y):= \sum_{k=0}^\infty \frac{(-1)^k}{k!} p^k \ad(q)^k(y).
\end{split}
\end{equation}
From the conditions on the topology on $Y$ it follows that the operators $\alpha,\beta$
are well-defined and continuous.  We remark that for $y_0\in Y^{\ad V}$ we have
$\alpha(p^i y_0 q^j)=\delta_{j,0}p^i y_0 q^j$ and $\beta(p^i y_0 q^j)=\delta_{i,0}p^i y_0 q^j$.
These equalities actually motivated the definition.


It is checked directly that $[\alpha,\beta]=0$ and  $[p,\alpha(y)]=[q,\beta(y)]=0$ for all $y\in Y$.
In particular, $\operatorname{im}(\alpha\circ\beta)\subset Y^{\ad V}$.
Further, similarly to the previous paragraph, the series \begin{equation}\label{eq:series1}\sum_{k,l=0}^\infty \frac{1}{k!l!}p^k
\alpha\circ\beta\left(\ad(q)^k \ad(p)^ly\right)q^l\end{equation} converges for any $y\in Y$. Moreover, expanding
  $\alpha$ and $\beta$ in (\ref{eq:series1}) (and getting a summation over 4 indexes) one gets that this sum coincides with $y$.
So (\ref{eq:series1}) is a presentation of $y$ as an element of
$\W^\wedge\widehat{\otimes}Y^{\ad V}\subset Y$.
\end{proof}
\begin{proof}[Proof of Proposition \ref{Lem:11.2}]
Set $U:=\widetilde{\m}\cap V$.
Consider $Y:=X^\wedge$ as an $\W^\wedge$-bimodule by means of the isomorphism $\U^\wedge\cong \W^\wedge\widehat{\otimes}
\Walg^\wedge$. Since $Y$ is a complete and separated topological $\U^{\wedge}$-bimodule, it is also
a complete and separated topological $\W^\wedge$-module. Our claim now follows from Lemma \ref{Lem:11.2.1}.
\end{proof}

For $X\in \HC(\U)$ we set $X_{\ddag}:=\Phi_*(X^\wedge)^{\ad V}$. This is a topological $T$-equivariant bimodule over $\Walg^\wedge$.
On the other hand, pick $Y\in \HC(\Walg)$ and consider its completion $Y^\wedge:=\varprojlim_{n\rightarrow\infty}
Y/Y(\Walg_{>0})^n$. This is also a topological $T$-equivariant $\Walg^\wedge$-bimodule.

\begin{Lem}\label{Lem:11.3}
The functors $\bullet_{\ddag}$ and $(\bullet_\dagger)^\wedge$ from $\HC(\U)$ to the category
of topological $T$-equivariant $\Walg^\wedge$-bimodules are isomorphic.
\end{Lem}
\begin{proof}
We start by constructing a natural transformation $(\bullet_\dagger)^\wedge\rightarrow \bullet_\ddag$.
Pick $X\in \HC(\U)$ and choose a finite dimensional  $\ad(\g)$-submodule  $X_0\subset X$ generating
$X$ as a left (or as a right) $\U$-module. Define the filtration on $X$ using
$X_0$ as in the construction of $\bullet_\dagger $. Consider the $\W^\wedge_\hbar\widehat{\otimes}_{\K[[\hbar]]}\Walg^\wedge_\hbar$-bimodule $X_\hbar^\wedge$. We have the twisted Kazhdan action of $\K^\times$
on $X_\hbar^\wedge$, compare with the end of Subsection \ref{SUBSECTION_Ocat}.
So we can consider the subspace $(X_\hbar^\wedge)_{\K^\times-fin}$ of $\K^\times$-finite vectors
in $X_\hbar^\wedge$.  Set $$X^\heartsuit:=(X_\hbar^\wedge)_{\K^\times-fin}/(\hbar-1)(X_\hbar^\wedge)_{\K^\times-fin}.$$
It is easy to see that $X^\heartsuit=\U^{\heartsuit}\otimes_{\U}X$.
So we get that the natural homomorphism
$X\rightarrow X^\wedge$ extends to $X^\heartsuit\rightarrow X^\wedge$.

Let us show that there is a natural homomorphism $X_\dagger\rightarrow (X^\heartsuit)^{\ad V}$.
Indeed, $T$ acts locally finitely on the space of $\K^\times$-finite (with respect to the usual Kazhdan action)
elements in $(X^\wedge_\hbar)^{\ad V}$ so this space of $\K^\times$-finite elements is included into  $[(X_\hbar^\wedge)_{\K^\times-fin}]^{\ad V}$. The image of the induced homomorphism $$X_{\dagger}=
[(X^\wedge_\hbar)^{\ad V}]_{\K^\times-fin}/(\hbar-1)[(X^\wedge_\hbar)^{\ad V}]_{\K^\times-fin}\rightarrow (X_\hbar^\wedge)_{\K^\times-fin}/(\hbar-1)(X_\hbar^\wedge)_{\K^\times-fin}=X^\heartsuit$$ lies in
$(X^\heartsuit)^{\ad V}$.

So we have constructed  a   homomorphism $X_\dagger\rightarrow X_\ddag$. From the construction
it follows that this homomorphism is functorial. Since $X_\ddag$ is complete with respect to the
$\Walg_{>0}$-adic topology, this homomorphism extends to $X_{\dagger}^\wedge\rightarrow X_{\ddag}$.
We are going to show that the corresponding natural transformation $\bullet_{\dagger}^\wedge\rightarrow \bullet_\ddag$ is an isomorphism.

Both functors are right exact (for $\bullet_\ddag$ this follows from Lemma \ref{Lem:11.2}).
Similarly to the proof of
Lemma \ref{Lem:11.1}, it is enough to  show that $(\U_{L\dagger})^{\wedge}\cong (\U_L)_{\ddag}$
or, alternatively, that the map \begin{equation}\label{eq:iso}\W(\U_{L\dagger})^\wedge \rightarrow \U_L^\wedge\end{equation}
induced by $(\U_{L\dagger})^\wedge\rightarrow \U_{L\ddag}$ is an isomorphism.
Recall (Proposition \ref{Prop:ULdagger}) that $\U_{L\dagger}=\Walg_L$.

The proof of Proposition \ref{Prop:ULdagger} implies that $\Phi_\hbar$ induces
an isomorphism $(\U_L)^\heartsuit\rightarrow \W(\Walg_L)^\heartsuit$.  Analogously to \cite{Wquant},
Subsection 3.2, we see that
$\W(\Walg_L)^\wedge$  is the completion of $\W(\Walg_L)^\heartsuit$  in the $[(\widetilde{\m}\cap V)\otimes 1+1\otimes \Walg_{>0}]$-adic topology. Similarly, $\U_L^\wedge$ is the completion of
$\U_L^\heartsuit$ in the $\widetilde{\m}_\chi$-adic topology. So
the isomorphism $\W(\Walg_L)^\heartsuit\rightarrow\U_L^\heartsuit$ extends
to a topological isomorphism   $\W(\Walg_L)^\wedge\rightarrow
\U_L^\wedge$.
This isomorphism is nothing else but (\ref{eq:iso}).
\end{proof}

\begin{proof}[Proof of Theorem \ref{Thm_tensor}]
By Lemma \ref{Lem:11.1}, $X\otimes_\U M=X^\wedge\widehat{\otimes}_{\U^\wedge}M$.
We have (see the discussion on the construction of $\KFun$)
\begin{align*}
\Phi_*(X^\wedge\widehat{\otimes}_{\U^\wedge}M)=\W^\wedge(X_{\ddag})\widehat{\otimes}_{\W^\wedge(\Walg^\wedge)}(\K[\widetilde{\m}\cap V]\otimes \KFun(M))=\K[\widetilde{\m}\cap V]\otimes (X_{\ddag}\widehat{\otimes}_{\Walg^\wedge}\KFun(M)).
\end{align*}
By Lemma \ref{Lem:11.3}, $X_{\ddag}=(X_{\dagger})^\wedge$. So it remains to show that $(X_\dagger)^\wedge=X_\dagger\otimes_{\Walg}\Walg^\wedge$. The proof is similar to that of Lemma \ref{Lem:11.1}.
\end{proof}

%

\subsection{Equivalence relation on $\Irr_{fin}(\Walg)$}\label{SUBSECTION_equivalence}
In this subsection we will introduce an equivalence relation on the set $\Irr_{fin}(\Walg)$
of (isomorphism classes) of finite dimensional irreducible $\Walg$-modules.
For this we will need to introduce tensor products of $\Walg$-modules with $G$-modules,
compare with \cite{Goodwin}. Then we will define a relation $\sim$ on $\Irr_{fin}(\Walg)$
using Proposition \ref{Prop:3.2}. Finally, we will prove that $\sim$ is indeed an equivalence
relation, Theorem \ref{Thm:3.3}.

Let $L$ be a finite dimensional $G$-module. Define the functor $L\rightthreetimes \bullet: \Walg$-$\Mod\rightarrow
\Walg$-$\Mod$ by $L\rightthreetimes N:=\Walg_L\otimes_\Walg N$.

\begin{Prop}\label{Prop:1.5.2}
\begin{enumerate}
\item We have a functorial $\q$-equivariant isomorphism $L\rightthreetimes N\cong L\otimes N$.
If $N$ is $Q$-equivariant (meaning that $Q$ acts on $N$   such that $\Walg\otimes N\rightarrow N$
is $Q$-equivariant and the $Q$-action integrates the
action of $\q\subset \Walg$), then the isomorphism above is also $Q$-equivariant.
\item The functor $L\rightthreetimes\bullet$ is exact.
\end{enumerate}
\end{Prop}
\begin{proof}
Assertion 1 follows from Lemma \ref{Lem:1.5.1}. The second assertion follows easily
from the first one.
\end{proof}

Below we will need the following lemma.

\begin{Lem}\label{Lem:1.5.3}
Let $N$ be a finite dimensional $\Walg$-module and $L$ a finite dimensional $G$-module.
Then there is a sub-bimodule $R\subset \Walg_L$ of finite codimension and such that $R\otimes_\Walg N=0$
so that $L\rightthreetimes N\cong (\Walg_L/R)\otimes_{\Walg}N$.
\end{Lem}
\begin{proof}
Set $R:=\Walg_L\Ann_\Walg(N)$. Since $\Ann_\Walg(N)$ has finite codimension in $\Walg$
and $\Walg_L$ is finitely generated as a right $\Walg$-module, we see that $\dim \Walg_L/R<\infty$.
On the other hand, $R\otimes_\Walg N=0$.
\end{proof}

The following claim follows from Theorem \ref{Thm_tensor}.

\begin{Cor}\label{Cor:1.5.6}
Let $L,L_1,L_2$ be finite dimensional $G$-modules. Then
\begin{enumerate}
\item
the functors $L\rightthreetimes\bullet$ and $L^*\rightthreetimes\bullet$
from $\OCat^\t(\theta)$ to itself are mutually adjoint.
\item the functors $L_1\rightthreetimes(L_2\rightthreetimes\bullet)$ and $(L_1\otimes L_2)\rightthreetimes \bullet$
are isomorphic.
\end{enumerate}
\end{Cor}
\begin{proof}
Recall that by Proposition \ref{Prop:ULdagger}, $\U_{L\dagger}=\Walg_L$. By Theorem \ref{Thm_tensor},
the functors $\KFun(L\otimes\bullet), L\rightthreetimes \KFun(\bullet): \widetilde{\Wh}^\t(e,\theta)\rightarrow \widetilde{\OCat}^\t(\theta)$ are isomorphic. Both claims of the lemma follow from the fact that
$\KFun$ is an equivalence between $\widetilde{\Wh}^\t(e,\theta)$ and $\widetilde{\OCat}^\t(\theta)$.
\end{proof}

 \begin{Prop}\label{Prop:3.2}
Let $N_1,N_2$ be irreducible $\Walg$-modules. Then the following conditions are equivalent:
 \begin{enumerate} \item there is an irreducible finite dimensional $G$-module $L$ such that
$N_2$ is a quotient of $L\rightthreetimes N_1$,
 \item there is an irreducible finite dimensional $G$-module $L$ such that $N_2$ is a subquotient
 of $L\rightthreetimes N_1$.
 \item there is an irreducible finite dimensional $G$-module $L$ such that the $\Walg$-bimodule $\Hom_{\K}(N_1,N_2)$
 is a quotient of $\Walg_L$,
\item there is an irreducible finite dimensional $G$-module $L$ and a sub-bimodule $R\subset \Walg_L$ of finite codimension such that $\Hom_\K(N_1,N_2)$ is a subquotient of $\Walg_L/R$.
 \end{enumerate}
 \end{Prop}
We say that $N_1,N_2$ are {\it equivalent} and write $N_1\sim N_2$ if the pair $(N_1,N_2)$ satisfies one of the four
equivalent conditions of Proposition \ref{Prop:3.2}. Below we will see that this is indeed an equivalence relation.
 \begin{proof}
(3)$\Rightarrow$(1).
It is enough to show that the functors $$(N_1,N_2)\mapsto \Hom_{\Walg-\Walg^{op}}(\Walg_L,\Hom_{\K}(N_1,N_2)), \Hom_{\Walg}(L\rightthreetimes N_1,N_2)$$ from $\Walg$-$\operatorname{mod}_{fd}\times \Walg$-$\operatorname{mod}_{fd}^{opp}$
 to the category of vector spaces  are isomorphic. Recall that $L\rightthreetimes N_1$ is, by definition,
$\Walg_L\otimes_{\Walg} N_1$. Now the isomorphism of functors is a standard fact.

(1)$\Rightarrow$(2) is tautological.

(2)$\Rightarrow$(4). Let $R$ be a sub-bimodule in $\Walg_L$ of finite codimension such that $R\otimes_{\Walg}N_1=0$
existing by Lemma \ref{Lem:1.5.3}. Set $P:=\Walg_L/R$. Choose a composition series $0=P^0\subset P^1\subset \ldots\subset P^m=P$
for $P$ and let $N_1^i,N_2^i$ be such that $P^i/P^{i-1}=\Hom_\K(N_1^i,N_2^i)$. Since the tensor product functor is
right exact, we see that that any composition factor of $L\rightthreetimes N_1=P\otimes_{\Walg}N_1$ has the form $N_2^i$ with
$N_1^i=N_1$, q.e.d.

(4)$\Rightarrow$(3).  Being a sub-bimodule in $\Walg_L$, $R$ is $Q^\circ$-stable. Replacing $R$ with $\bigcap_{q\in Q}q.R$ we may assume, in addition, that $R$ is  $Q$-stable. The bimodule
$\Hom(N_1,N_2)$ is a subquotient of $\Walg_L/R$ if and only if it is a direct summand of a composition factor of the
{\it $Q$-equivariant} bimodule $\Walg_L/R$. By Proposition \ref{Prop:11.2}, this composition factor is
a quotient of $\Walg_{L'}$ for some finite dimensional $G$-module $L'$.
 \end{proof}


\begin{Rem}\label{Rem:3.4}
 Since $\Walg_L$ is a $Q$-equivariant bimodule, we see that $\sim$ is an $A(e)$-invariant relation,
 where $A(e)$ denotes the component group of $e$, i.e., $A(e):=Q/Q^\circ$.
\end{Rem}

\begin{Prop}\label{Thm:3.3}
$\sim$ is an equivalence relation.
\end{Prop}
\begin{proof}
Reflectivity ($N\sim N$) is clear: take $L=\K$. The claim that $\sim$ is symmetric follows from
Corollary \ref{Cor:1.5.6}: $\Hom_\Walg(L\rightthreetimes N_1,N_2)=\Hom_\Walg(N_1,L^*\rightthreetimes N_2)$.

Let us check transitivity: let $N_1,N_2,N_3$ be irreducible finite dimensional $\Walg$-modules and $L_{12},L_{23}$ be irreducible
$G$-modules with $\Hom_{\Walg}(L_{23}\rightthreetimes N_3,N_2),\Hom_{\Walg}(L_{12}\rightthreetimes N_2,N_1)\neq 0$.
Choose nonzero elements $\varphi\in \Hom_{\Walg}(L_{23}\rightthreetimes N_3,N_2), \psi\in
\Hom_{\Walg}(L_{12}\rightthreetimes N_2,N_1)$. Since $N_1,N_2$ are irreducible, we see that
$\varphi,\psi$ are surjective. Since the functor $ L_{12}\leftthreetimes \bullet$ is exact, we see that
the homomorphism $\widetilde{\varphi}: (L_{12}\otimes L_{23})\rightthreetimes N_3=L_{12}\rightthreetimes (L_{23}\rightthreetimes N_{3})\rightarrow L_{12}\rightthreetimes N_2$ induced by $\varphi$ is surjective.
So $\psi\circ\widetilde{\varphi}$ is a nonzero element of $\Hom_{\Walg}((L_{12}\otimes L_{23})\rightthreetimes N_3),
N_1)$. Hence $N_1\sim N_3$.
\end{proof}

The reason why we need the equivalence relation $\sim$ is as follows.
Proposition \ref{Prop:3.2} and Proposition \ref{Prop:11.2} show that Theorem
\ref{Thm_main} is equivalent to the following claim:
\begin{itemize}
\item[(*)] Any two irreducible $\Walg$-modules $N_1,N_2$ with integral central characters that differ by an element of $P^+$
are equivalent.
\end{itemize}
In the next two subsections we will prove that (*) holds.

\subsection{The case of even nilpotents}\label{SUBSECTION_even}
In this subsection we will prove Theorem \ref{Thm_main} in the case when $\Orb$ is even,
i.e., when  all eigenvalues of $\ad h$ are even. But first we will need a more general result.

\begin{Prop}\label{Lem:12.3}
Let $\Orb$ be an arbitrary special nilpotent orbit  in $\g$.
Any two $A(e)$-orbits in $Y^{P^+}$ have equivalent points.
\end{Prop}
\begin{proof}
This claim is a reformulation of Corollary \ref{Cor:indecomp}.
\end{proof}

Next we are going to treat the case when $e$ is even.

\begin{Prop}\label{Prop:12.4}
Suppose $\Orb$ is even. Then there is an $A(e)$-fixed point in $Y^{P^+}$.
\end{Prop}
\begin{proof}
In the notation of the end of Subsection \ref{SUBSECTION_W_def}, let $P$ be the parabolic
subgroup of $G$ corresponding to the parabolic subalgebra $\bigoplus_{i\geqslant 0}\g(i)\subset \g$.
As Borho and Kraft checked in \cite{BB}, the action of $G$ on $G/P$ induces a surjection $\U\twoheadrightarrow D(G/P)$
(Theorem 3.8 in {\it loc. cit.}),
its kernel $\J$ has an integral central character (Corollary 3.7, loc.cit.).
Since $Q\subset P$, we see that  the map $T^*(G/P)\rightarrow
\overline{\Orb}$ is birational. It follows that $\gr \U/\J=\K[\overline{\Orb}]$.
 Therefore the multiplicity of $\U/\J$ on $\Orb$ is 1. So $\J_\dagger$ has codimension 1 in $\Walg$ and is $A(e)$-stable. A unique irreducible representation of $\Walg/\J_\dagger$ is fixed by $A(e)$.
\end{proof}

Then the claim (*) for an even $e$ follows directly from Propositions \ref{Lem:12.3}, \ref{Prop:12.4} and
Remark \ref{Rem:3.4}.

\subsection{A reduction}\label{SUBSECTION_reduction}
We use the notation introduced in Subsection \ref{SUBSECTION_Ocat}. Let $G_0$ be the Levi subgroup
of $G$ corresponding to $\g_0$. Let $\sim_0$ denote the equivalence on $\Irr_{fin}(\underline{\Walg})$
defined by the tensor product with $G_0$-modules.
\begin{Prop}\label{Prop:12.1}
Let $\underline{N}_1,\underline{N}_2$ be $\underline{\Walg}$-modules. If $\underline{N}_1\sim_0 \underline{N}_2$,
then $L^\theta(\underline{N}_1)\sim L^\theta(\underline{N}_2)$.
\end{Prop}
In the proof we will need an auxiliary lemma.

Let $N$ be an irreducible $\Walg$-module,
$\alpha\in \K$ be a maximal eigenvalue of $\theta$ on $N$ with eigenspace $\underline{N}$
("maximal" means that $\alpha+n$ is not an eigenvalue for a positive integer $n$) so that $N\cong L^\theta(\underline{N})$.  Let $L$ be an irreducible $G$-module, $\beta\in \K$ be the maximal eigenvalue of $\theta$, and $\underline{L}$ be the corresponding
eigenspace. Then $\underline{L}$ is an irreducible $G_0$-submodule of $L$ with the same highest
weight (with respect to an appropriate Borel subgroup) as $L$.  Finally, set $N^1:=L\rightthreetimes N$. Let $\gamma$ be the maximal eigenvalue of $\theta$ on $N^1$, $\underline{N}^1$ being the corresponding eigenspace.
Then $\underline{N}^1$ is a $\Walg^0$-module and so
we can consider it as $\underline{\Walg}$-module.

\begin{Lem}\label{Lem:12.2}
We have $\gamma=\alpha+\beta$ and $\underline{N}^1\cong\underline{L}\rightthreetimes \underline{N}$
(an isomorphism of $\underline{\Walg}$-modules).
\end{Lem}
\begin{proof}
The equality $\gamma=\alpha+\beta$ follows from the observation that $N^1$ and $L\otimes N$
are $Q$-equivariantly isomorphic (Proposition \ref{Prop:1.5.2}).

Let us prove the isomorphism. Recall an element $\delta\in \mathfrak{t}^*$ introduced in \cite{BGK}, see also Remark 5.5 in \cite{Ocat}
for the definition of $\delta$. If $\alpha$ is the maximal eigenvalue of $\theta$ in $N$, then $\alpha+\langle\delta,\theta\rangle$
is the maximal eigenvalue of $\theta$ in $\KFun^{-1}(N)$. Theorem \ref{Thm_Ocat} implies that $\underline{N}$ is the $\theta$-eigenspace
with eigenvalue $\alpha+\langle\delta,\theta\rangle$ in $\KFun^{-1}(N)^{\widetilde{\m}_\chi}$. The Skryabin equivalence theorem implies that
$\Sk_0(\underline{N})$ is the $\theta$-eigenspace with eigenvalue $\alpha+\langle\delta-\delta_0, \theta\rangle$
(where $\delta_0$ is the analog of $\delta$ for $\g_0$) in $\KFun^{-1}(N)^{\g_{>0}}$ and (since this eigenvalue is the maximal one) also in $\KFun^{-1}(N)$. Similarly, $\Sk_0(\underline{N}^1)$ is the $\theta$-eigenspace with eigenvalue $\gamma+\langle\delta-\delta_0,\theta\rangle$ in $\KFun^{-1}(L\rightthreetimes N)$. But $\underline{L}\otimes \Sk_0(\underline{N})$ is the $\theta$-eigenspace of eigenvalue $\gamma+\langle\delta-\delta_0,\theta\rangle$ in $L\otimes \KFun^{-1}(N)$.
According to Theorem \ref{Thm_tensor}, $\underline{L}\otimes \Sk_0(\underline{N})\cong \Sk_0(\underline{L}\rightthreetimes \underline{N})$, while $\KFun^{-1}(L\rightthreetimes N)=L\otimes \KFun^{-1}(N)$. It follows that $\Sk_0(\underline{L}\rightthreetimes \underline{N})=\Sk_0(\underline{N}_1)$ and so $\underline{L}\rightthreetimes \underline{N}=\underline{N}_1$.
\end{proof}

\begin{proof}[Proof of Proposition \ref{Prop:12.1}]
Pick a $G_0$-module $L'$ such that  $L'\rightthreetimes \underline{N}_1\twoheadrightarrow
\underline{N}_2$. Let $\lambda$ denote its highest weight. Assume for a moment that $L'=\underline{L}$ for some $G$-module $L$,  equivalently,  $\lambda$  is dominant for $G$. Then $L'\rightthreetimes \underline{N}_1$ is the highest weight subspace in $L\rightthreetimes L^\theta(\underline{N}_1)$ by  Lemma \ref{Lem:12.2}. Hence $L^\theta(\underline{N}_2)$ is a composition factor in
$L\rightthreetimes L^\theta(\underline{N}_1)$, and we are done.

In general, there is a
 character $\mu$ of $L$ such that both $\mu$ and $\lambda+\mu$ are dominant for $G$. Let $\K_\mu$
 denote the 1-dimensional module with highest weight $\mu$ for $G_0$. Replacing $\underline{N}_2$
 with $\K_\mu\rightthreetimes \underline{N}_2$ we get to the situation of the previous paragraph.
\end{proof}

Now we are in position to finish the proof of (*).

\begin{proof}
Assume that $\g_0$ is chosen such that $e$ is distinguished in $\g_0$, i.e., all
semisimple elements in $\z_{\g_0}(e)$ are in the center of $\g_0$. Such an element is
necessarily even in $\g_0$, see, e.g., \cite{CM}, 8.2.

Pick $N_1,N_2\in Y^{P^+}$.
Then $N_1=L^\theta(\underline{N}_1),N_2=L^\theta(\underline{N}_2)$, where $\underline{N}_1,\underline{N}_2$
are irreducible $\underline{\Walg}$-modules. They  have integral central characters (for $G_0$).
This follows, for example, from \cite{BGK}, Theorem 4.7, or from \cite{Miura}, Theorem 5.1.1.
Thanks to Proposition \ref{Prop:12.1}, $N_1\sim N_2$ provided $\underline{N}_1\sim \underline{N}_2$.
The latter follows   from Subsection \ref{SUBSECTION_even}.
\end{proof}

\section{Preliminaries on cells, Springer correspondence and Lusztig's groups}\label{SECTION_cell}
\subsection{Cells in Weyl groups and cell modules}\label{SUBSECTION_cells}
Let $\Hecke$ be the Hecke algebra of $W$ viewed as a $\ZZ[q^{\pm 1/2}]$-algebra.
Let $c_w, w\in W,$ be the Kazhdan-Lusztig basis, see e.g. \cite[5.1.1]{orange}.
Recall (see e.g. \cite[\S 5.1]{orange}) that
one can use it to partition $W$ into two-sided, left and right cells as follows. For $w\in W$ consider the based
two-sided ideal $I^d_w$ in $\Hecke$ generated by $c_w$ (``based'' means ``spanned by basis elements'').
By the two-sided cell $\dcell_w$  of $w$ we mean the set of all $u$ such that $c_u\in I^d_w$
but $c_u\not\in I^d_{w'}$ for $I^d_{w'}\subsetneq I^d_w$.
Similarly, we can consider the left (resp., right) based ideal $I^l_w$ (resp., $I^r_w$) generated by $w$
and define the left (resp., right) cell $\cell_w$ of $w$. It is clear that $W$ gets partitioned
into two-sided cells, and each two-sided cell gets partitioned into left cells. Moreover, one can show that
the map $w\mapsto w^{-1}$ preserves two-sided cells and maps left cells to right ones, and vice versa.
Also it is clear from the definition that the equivalence relation $\sim$ induced by the partition
of $W$ into two-sided cells is generated by the equivalence relations $\sim_L,\sim_R$ corresponding to
partitions into left and right cells, respectively.

To each two-sided cell $\dcell$ we assign the cell $\Hecke$-bimodule $[\dcell]_q$ that is the quotient
of $I^d_w, w\in \dcell,$ by the sum of all $I^d_{w'}\subsetneq I^d_{w}$. Similarly,
to each left cell $\cell$ we assign the left $\Hecke$-module $[\cell]_q$. All cell
bimodules and modules are flat over $\ZZ[q^{\pm 1/2}]$. Since $\ZZ W=\Hecke/(q^{1/2}-1)$,
we get the $W$-bimodule  $[\dcell]$ and the left $W$-module $[\cell]$. It is clear that
$\dim [\cell]=|\cell|$.  In the sequel we will consider $[\dcell]$ as a left module.
As such it decomposes as $[\dcell]=\bigoplus_{\cell\subset \dcell}[\cell]$.

We say that an irreducible $W$-module belongs to $\dcell$ if it appears
as an irreducible constituent of $[\dcell]$. This defines a partition of
$\Irr(W)$ into {\it families} $\Irr(W)^{\dcell}$ (in the sense of Lusztig, \cite{orange}).
Each family has a distinguished irreducible module, called a {\it special} module
(or a special representation).

In \cite{orange} Lusztig introduced a certain parametrization of
$\Irr(W)^{\dcell}$. This parametrization involves a certain finite group
$\bA(=\bA_{\dcell})$ constructed by Lusztig for the two-sided cell $\dcell$.  This group will be described
below (see e.g. Subsection \ref{SUBSECTION_Springer}). To each $U\in \Irr(W)^{\dcell}$
Lusztig assigned a pair $(x_U,V_U)$, where $x_U$ is an element in $\bA$ defined up to
conjugacy, and $V_U$ is an irreducible $Z_{\bA}(x_U)$-module. Following Lusztig,
we denote the set of such pairs by $\Mfr(\bA)$. For different
irreducibles the corresponding pairs are different, but not every
pair arises from some irreducible module.

\subsection{Orbits, Harish-Chandra bimodules, and the algebra $J$}\label{SUBSECTION_cell_HC}
It is a classical fact that left and two-sided cells admit an alternative
description in terms of  primitive ideals in $\Pr(\U_\rho)$.
Namely, it happens that the left cells are precisely the fibers
of the map $W\rightarrow \Pr(\U_\rho), w\mapsto \J(w\rho)$. Furthermore,
the two-sided cells are the fibers of the map that sends $w\in W$
to the open orbit in $\VA(\U/\J(w\rho))$. So we get  bijections
$\dcell\mapsto \Orb^{\dcell}, \Orb\mapsto \dcell^{\Orb}$
between the set of the special orbits and the set of the two-sided cells.

Pick a special orbit $\Orb$ and let $\dcell$ be the corresponding two-sided
cell. Recall that $\JCat_{\Orb}$ denotes the category $^{\rho}\!\HC^\rho_{\Orb}(\U)^{ss}$.
The based $\BQ$-algebra $[\JCat_{\Orb}]$ is known to be isomorphic to the block
$\BQ_{\otimes \ZZ}\JAlg_{\Orb}$ corresponding to $\Orb$ of the (rational form of the) Lusztig asymptotic Hecke algebra $\BQ\otimes_{\ZZ}\JAlg$.
Recall that $\JAlg_{\Orb}$ has a basis $t_x$ indexed by $\dcell$.  The set of simples in $\JCat_{\Orb}$ is again parameterized
by elements of $\dcell$ via $w\mapsto \M_w:=\BG(L(w\rho))$.
An isomorphism $[\JCat_{\Orb}]\rightarrow \JAlg_{\Orb}$ sends $\BG(L(w\rho))$ to $t_{w_*^{-1}}$,
where $w\mapsto w_*$ is a certain involution on $W$ studied by Joseph in \cite{Joseph_cyclicity}.
The results mentioned above in  this paragraph follow from \cite{Joseph_cyclicity} and the claim that the category
$\JCat_{\Orb}$ is multi-fusion, see e.g., Lemma \ref{Lem:J_MF}.

Recall that Lusztig defined an explicit (but complicated) homomorphism
$\Hecke \to \JAlg\otimes_{\ZZ}\ZZ[q^{\pm 1/2}]$, see e.g. \cite[\S 3.2]{Lusztig_subgroups}.
Under the specialization $q^{1/2}=1$, this homomorphism induces an isomorphism
$\BQ (W)\xrightarrow{\sim} \BQ\otimes_{\ZZ}\JAlg$, see e.g. \cite[\S 3.5]{Lusztig_subgroups}.
The last isomorphism induces a bijection between irreducible representations of $[\JCat_{\Orb}]$
and $\Irr(W)^{\Orb}(=\Irr(W)^{\dcell})$, see {\em loc. cit.}

Lusztig's homomorphism $\ZZ(W)\to \JAlg_\Orb$ can be described as follows.
Let us identify $\ZZ(W)$ with the Grothendieck ring of the category of {\em projective functors}
sending the generalized central character $\rho$ to itself, see \cite{BG} or \cite[5.1,5.2]{Mazorchuk}. Let ${\bf 1}\in \JCat_{\Orb}$ be
the unit object and let $\tilde {\bf 1}$ be an arbitrary lift of ${\bf 1}\in \JCat_\Orb \subset \HC_\Orb(\U)$
to an object of $\HC_{\bar \Orb}(\U)$. Let $F$ be a projective functor as above. Then the map
\begin{equation}\label{catLus}
\ZZ(W)\ni [F]\mapsto [F(\tilde {\bf 1})\pmod{\HC_{\partial\Orb}(\U)}]\in [\HC_{\Orb}(\U)]=[\JCat_\Orb]=\JAlg_\Orb
\end{equation}
is well defined and it follows from the results of
\cite{BG} combined with the Kazhdan-Lusztig conjecture that it
coincides with the Lusztig's homomorphism defined as in \cite[\S 3.5]{Lusztig_subgroups}.
We notice that the fact that the map \eqref{catLus} is a
homomorphism of rings is far from being trivial.

Let us introduce some notation.
The left and right annihilators of $\M_w$ are $\J(w\rho),\J(w^{-1}\rho)$. Consider the
subcategory $\JCat_{\cell}$ of all objects in $\JCat_{\Orb}$ such that the right annihilator
of the corresponding simple Harish-Chandra bimodule is $\J(w\rho), w\in \cell$. The tensor
product functor $\JCat_{\Orb}\boxtimes \JCat_{\Orb}\rightarrow \JCat_{\Orb}$ clearly
restricts to $\JCat_{\Orb}\boxtimes \JCat_{\cell}\rightarrow \JCat_{\cell}$. The rational
$K$-group $[\JCat_{\cell}]$ is naturally identified with $[\cell]$.

Similarly we can consider the subcategory $_{\cell}\!\JCat$ with left annihilator corresponding to $\cell$.
Next, for left cells $\cell_1,\cell_2$, set $_{\cell_1}\!\JCat_{\cell_2}:=
\,_{\cell_1}\!\JCat\cap \JCat_{\cell_2}$.

Now let us discuss the categories like $^\lambda\!\JCat$ in the case of an arbitrary (integral) central character $\lambda$.
Of course, if $\lambda$ is regular, then $^\lambda\!\JCat$ is identified with $\JCat$ via
a translation functor. For a singular character $\lambda$ we need some modifications.
Thanks  to (i)-(iii) in the end of Subsection \ref{SUBSECTION_shifts},
we see that the condition that $w$ is compatible with $\lambda\in P^+$
holds or does not hold simultaneously for all $w$ in a given left cell. So the notion of
compatibility of left cells and dominant weights makes sense.
For $\lambda\in P^+$ compatible with a left cell $\sigma$ let $^{\lambda}_{\sigma}\!\JCat$ be the full
subcategory in $^\lambda\!\JCat$ whose simples are in the image of $T_\rho^\lambda(_{\sigma}\!\JCat)$.

%




\subsection{Springer representation}\label{SUBSECTION_Springer}
Let us recall that to each nilpotent orbit $\Orb$ Springer attached a $W\times A(e)$-module
$\Spr(\Orb)$. As a vector space, the Springer representation has a basis
indexed by  irreducible
components of the Springer fiber $\mathcal{B}_e$ that is the preimage of
$e\in \g$ under the natural morphism $T^*\mathcal{B}\rightarrow \g\cong \g^*$.
In other words, $\mathcal{B}_e$ consists of all Borel subalgebras in $\g$ containing $e$.

The group $A(e)$ acts on the set of irreducible components of $\mathcal{B}_e$ by
permutations and so acts also on $\Spr(\Orb)$. It turns out that there is a natural
action of $W$ commuting with $A(e)$ on $\Spr(\Orb)$. An important property of
$\Spr(\Orb)$ is that it is a multiplicity free $W\times A(e)$-module, i.e., each irreducible
$W\times A(e)$-module appears at most once. Further, the trivial $A(e)$-module always appears in $\Spr(\Orb)$,
and the corresponding $W$-module is a unique special representation in $\Irr(W)^{\Orb}$
provided the orbit $\Orb$ is special. 

Now consider the isotypic component $\Spr(\Orb)^{\dcell}$, that is, the sum of all irreducible
$W$-modules that appear in $[\dcell]$. Then $\bA_{\dcell}$ is the quotient
of $A(e)$ by  the kernel of the action of $A(e)$ on $\Spr(\Orb)^{\dcell}$, see \cite[13.1.3]{orange}.

The following result establishes a relationship between Lusztig's parametrization of $\Irr(W)^{\dcell}$,
see the end of Subsection \ref{SUBSECTION_cells}, and the structure of $\Spr(\Orb)^{\dcell}$. 
It follows from \cite[Corollary 0.5]{Lusztig_families}.

\begin{Prop}\label{Prop:Springer1}
We have the following isomorphism of $W\times \bA$-modules: $$\Spr(\Orb)^{\dcell}=\bigoplus_{U\in \Irr(W)^{\dcell}, x_U=1} U\otimes V_U.$$
\end{Prop}

\subsection{Lusztig's subgroups $H_{\cell}$} \label{hsigma}
In \cite[Proposition 3.8]{Lusztig_subgroups},  to each left cell $\cell$ in $\dcell$, Lusztig assigned a subgroup
$H_{\cell}\subset \bA$ determined up to conjugacy. The subgroups depend on
the cell modules rather than on the left cells themselves.

Let us explain an important property
of those subgroups, \cite[Proposition 3.16]{Lusztig_subgroups}.
Consider the left cells $\cell_1,\cell_2$ in $\dcell$.

\begin{Lem}\label{Lem:cells_vs_H}
The number $\#\,_{\cell_1}\! \JCat_{\cell_2}$ of simples in $\,_{\cell_1}\!\JCat_{\cell_2}$ (equal
to $|\cell_1^{-1}\cap \cell_2|$) coincides with $\#\Coh^{H_{\cell_2}}(\bA/H_{\cell_1})=\#\Coh^{\bA}(\bA/H_{\cell_1}\times \bA/H_{\cell_2})$.
\end{Lem}

We will also need a relationship between the (left) cell
modules, $\Spr(\Orb)$, and the Lusztig subgroups.

\begin{Prop}\label{Prop:Springer2}
We have an $\bA$-equivariant isomorphism $\BQ (\bA/H_{\cell})\cong \Hom_W([\JCat^{\cell}], \Spr(\Orb))$
(the latter is an $\bA$- and not just an $A(\Orb)$-module, thanks to Proposition \ref{Prop:Springer1}).
\end{Prop}

A problem is that it is not quite easy to extract the necessary information on Lusztig's subgroups
from his work. So, in the next three subsections, we will just produce certain subgroups in $\bA$
and prove Lemma \ref{Lem:cells_vs_H} and Proposition \ref{Prop:Springer2} for them. One can
also show that our subgroups are the same as Lusztig's but since the lemma and the proposition
above is all we need, we will not prove that the definitions are equivalent.

\subsection{Explicit descriptions: types B and C}\label{SUBSECTION_explicit_bc}
We recall (see \cite[\S 4.5]{orange}) that the irreducible representations of the Weyl group $W(B_n)\simeq W(C_n)$ are labeled by
{\em symbols} of rank $n$
and defect 1. By definition, such a symbol $\Lambda =\binom{M'}M$ is just two subsets $M, M'$ of $\ZZ_{\ge 0}$
with $|M'|=|M|+1$ and $\sum_{x\in M}x+\sum_{x\in M'}x=n+|M|^2$ up to an equivalence relation
described in {\em loc. cit.}; each equivalence class contains a unique {\em reduced}
symbol with $0\not \in M\cap M'$.

In types $B,C$ the combinatorial description of
families (i.e., the subsets of the form $\Irr(W)^{\dcell}$)
is as follows: two representations labeled by reduced symbols $\binom{M'_1}{M_1}$ and $\binom{M'_2}{M_2}$
are in the same family if and only if $M_1\cup M_1'=M_2\cup M_2'$ and
$M_1\cap M_1'=M_2\cap M_2'$, see \cite[\S 4.5]{orange}. Thus a family is completely determined by two
subsets $Z_1, Z_2\subset \ZZ_{\ge 0}$ such that $Z_1\cap Z_2=\varnothing$, $0\not \in Z_2$ and
$|Z_1|$ is odd: the family corresponding to $Z_1,Z_2$ consists of representations with reduced symbols  $\binom{M'}M$
such that $M\cap M'=Z_2$ and $M\cup M'=Z_1\cup Z_2$; let $\Fam(Z_1,Z_2)$ denote such a family.
Let us write the elements of $Z_1$ in an increasing order: $z_0< z_1< \ldots < z_{2m}$.
Then any representation from $\Fam(Z_1,Z_2)$ corresponds to reduced symbol
$\binom{Z_2\cup (Z_1-M)}{Z_2\cup M}$ where $M$ is a subset of $Z_1$ with $|M|=m$; thus we
have a bijection between $\Fam(Z_1,Z_2)$ and the set of subsets of $Z_1$ of cardinality $m$.
The special representation from $\Fam(Z_1,Z_2)$ corresponds to the subset
$M_0=\{ z_1, z_3, \ldots z_{2m-1}\} \subset Z_1$.

Fix a family $\Fam(Z_1,Z_2)$ and let $\dcell,\Orb$
 be the  corresponding two-sided cell and the special nilpotent orbit.
Recall that in the classical types nilpotent orbits are parameterized by partitions.
\cite{Carter}, 13.3, provides  a combinatorial recipe to compute $\Orb$ (i.e., the corresponding
partition) from the symbol of the special representation inside $\Fam(Z_1,Z_2)$.
%
%

We now describe representations of the form $[\cell]$ with $\cell\subset \dcell$.
The $W$-modules of the form $[\cell]$ (but not the left cells themselves!)
are in bijection with {\em Temperley-Lieb patterns} of the following form:
$$
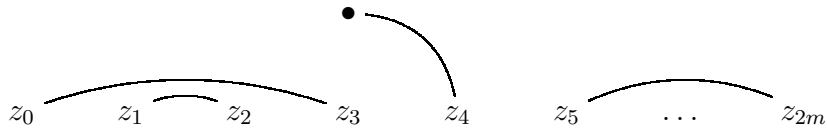
\xymatrix{
&&&\bullet \ar@/^1pc/@{-}[rd]&&&&\\
z_0\ar@/^1pc/@{-}[rrr]&z_1\ar@/^.5pc/@{-}[r]&z_2&z_3&z_4&z_5\ar@/^1pc/@{-}[rr]&\ldots &z_{2m}
}
$$

Here is a formal definition: Temperley-Lieb (shortly, TL) pattern above is an embedded into
$\BR \times [0,1]$ unoriented cobordism of the set
$Z_1\subset \ZZ \subset \BR=\BR \times 0\subset \BR \times [0,1]$ to a 1-point set embedded
into $\BR =\BR \times 1\subset  \BR \times [0,1]$. To such a pattern $Y$ one associates a
representation $[Y]$ of $W(B_n)$ which is the direct sum of the irreducible representations labeled by symbols
$\binom{Z_2\cup (Z_1-M)}{Z_2\cup M}$ (all with multiplicities 1) where $M\subset Z_1$ contains
precisely one of each $z_i$ and $z_j$ connected by an arc. This procedure produces a bijection
between TL patterns and left cell modules, see \cite[11.1]{fam} and \cite{cellules}.

\begin{Ex} Assume that $Z_2=\{ 1,3\}$ and $Z_1=\{0,2,5,6,7\}$. Then the pattern
$$
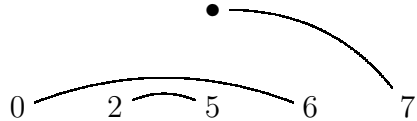
\xymatrix{
&&\bullet \ar@/^1pc/@{-}[rrd]&&\\
0\ar@/^1pc/@{-}[rrr]&2\ar@/^.5pc/@{-}[r]&5&6&7
}
$$
corresponds to the following cell representation of $W(B_{12})$:
$$\binom{0,1,2,3,7}{1,3,5,6}+\binom{1,2,3,6,7}{0,1,3,5}+\binom{0,1,3,5,7}{1,2,3,6}+\binom{1,3,5,6,7}{0,1,2,3}.$$
\end{Ex}

Now let us describe the Lusztig group $\bA:=\bA_{\dcell}$.
Let $V_{Z_1}$ be the set of subsets of $Z_1$  of even cardinality. Then $V_{Z_1}$ has a natural
structure of a symplectic vector space over the field $\bF_2$: the sum is the symmetric difference and
the symplectic form is $$(M,M')=|M\cap M'|\pmod{2}.$$ Let $e_i=\{z_{i-1},z_i\} \in V_{Z_1}.$ Then
$e_1,\ldots e_{2m}$ is a basis of $V_{Z_1}$ with $(e_i,e_j)=1$ if and only if $|i-j|=1$.

Let $\bA^B, \bA^C\subset V_{Z_1}$ be the Lagrangian subspaces with bases
$\{e_1, e_3, \ldots, e_{2m-1}\}$ (for $\bA^B$) and $\{e_2, e_4, \ldots, e_{2m}\}$
(for $\bA^C$).
It is clear that $V_{Z_1}=
\bA^B\oplus \bA^C$ and the symplectic form gives the identification $(\bA^B)^*=\bA^C$.

It turns out that the Lusztig group $\bA$ gets identified with $\bA^B$ in type $B$
and with $\bA^C$ in type $C$, see \cite[\S 4.5]{orange}.
So we have an identification of $V_{Z_1}$ with $\bA\oplus \bA^*$ and hence with Lusztig's set
$\Mfr(\bA)$ recalled in Subsection \ref{SUBSECTION_cells}.

Now let us describe the Lusztig parameterization $U\mapsto (x_U,V_U)$ with $x_U\in \bA, V_U\in \bA^*$.
The set $\Fam(Z_1,Z_2)$ is embedded in $V_{Z_1}$ as follows: the representation labeled by the symbol
$\binom{Z_2\cup (Z_1-M)}{Z_2\cup M}$ corresponds to the symmetric difference of $M$ and $M_0=\{z_1,z_3,\ldots,z_{2m-1}\}$.
This gives the Lusztig parameterization (we take the ``same type component'' of the image
of $U$ for $x_U$, and  the ``different type component'' for $V_U$).




It will be convenient for us to describe the cell modules in a way slightly different from the above.

Let $T$ be a Temperley-Lieb pattern as above; then the subspace $\Lagr_T$ of $V_{Z_1}$ with basis given
by two element subsets $\{ z_i,z_j\}$, where $z_i$ and $z_j$ are connected by an arc, is Lagrangian.

\begin{Lem}\label{Lem:L_T_BC}
We have $\Lagr_T\subset \Fam(Z_1,Z_2)$ and $[T]=\bigoplus_{(x_U,V_U)\in \Lagr_T}U$.
\end{Lem}
\begin{proof}
An element $\ell$ of $\Lagr_T$ is the union of pairs $\{z_{i_k},z_{j_k}\}$, each pair being connected  by an arc.
We remark that an arc always connect elements with different parities. The set $M_0$ consists
of all elements with fixed parity. For $M$ take the symmetric difference of $M_0$ and the union
of all pairs in $\ell$. Then $\ell$ corresponds to the representation with symbol $\binom{Z_2\cup (Z_1-M)}{Z_2\cup M}$.

The second statement is just the reformulation of the description of the map $U\mapsto (x_U,V_U)$.
given above.
\end{proof}


We now describe the subgroup $H_\cell \subset \bA$ corresponding to a left cell $\cell$.
For this we first  introduce  certain subgroups $H^B_T\subset \bA^B, H^C_T\subset \bA^C$
for a TL pattern $T$.

Type $B$: connect $z_0$ and $z_1$, $z_2$ and $z_3$ and so on. A basis of
$H_T^B\subset \bA^B$ is labeled by the connected components of the resulting picture
homeomorphic to a circle; the basis element corresponding to such a connected component is
$\sum e_{2i+1}$, where the summation is over the indices $i$ such that the arc connecting $z_{2i}, z_{2i+1}$ appears in that connected component.

Type $C$:  connect $z_1$ and $z_2$, $z_3$ and $z_4$ and so on. A basis of
$H_T^C\subset \bA^C$ is labeled by the connected components of the resulting picture
homeomorphic to a circle;  the basis element corresponding to such a connected component is
$\sum e_{2i}$ with the summation   over the indices $i$ such that the arc connecting $z_{2i-1}, z_{2i}$ appears in that connected component.

\begin{Ex} Consider the pattern
$$
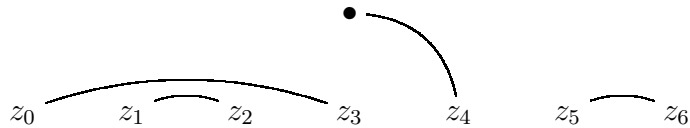
\xymatrix{
&&&\bullet \ar@/^1pc/@{-}[rd]&&&\\
z_0\ar@/^1pc/@{-}[rrr]&z_1\ar@/^.5pc/@{-}[r]&z_2&z_3&z_4&z_5\ar@/^.5pc/@{-}[r]&z_6
}
$$

The procedure above in type $B$ gives:
$$
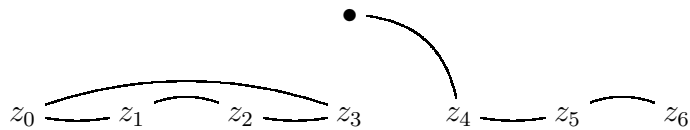
\xymatrix{
&&&\bullet \ar@/^1pc/@{-}[rd]&&&\\
z_0\ar@/^1pc/@{-}[rrr]\ar@/_.25pc/@{-}[r]&z_1\ar@/^.5pc/@{-}[r]&z_2\ar@/_.25pc/@{-}[r]&z_3&
z_4\ar@/_.25pc/@{-}[r]&z_5\ar@/^.5pc/@{-}[r]&z_6
}
$$
Thus, $H_T^B=\langle e_1+e_3\rangle \subset \langle e_1,e_3,e_5\rangle$.

The procedure above in type $C$ gives:
$$
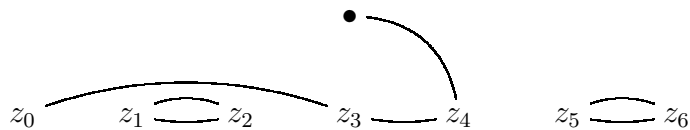
\xymatrix{
&&&\bullet \ar@/^1pc/@{-}[rd]&&&\\
z_0\ar@/^1pc/@{-}[rrr]&z_1\ar@/^.5pc/@{-}[r]\ar@/_.25pc/@{-}[r]&z_2&z_3\ar@/_.25pc/@{-}[r]&
z_4&z_5\ar@/^.5pc/@{-}[r]\ar@/_.25pc/@{-}[r]&z_6
}
$$
Thus, $H_T^C=\langle e_2,e_6\rangle \subset \langle e_2,e_4,e_6\rangle$.
\end{Ex}


\begin{Lem}\label{lem:lagr_subspace} For a Temperley-Lieb pattern $T$ the subgroup
$H_T^B\oplus H_T^C\subset \bA^B\oplus \bA^C=V_{Z_1}$ equals $\Lagr_T$.
\end{Lem}
\begin{proof}
It follows from the definitions of $H_T^B, H_T^C$ that both are contained in $\mathcal{L}_T$.
It remains to show that $\dim H_T^B+\dim H_T^C=m$. This is done by induction on $m$. Namely,
in $T$ there will be an arc connecting $z_i,z_{i+1}$ for some $i$. The vector $e_i$ lies
either in $H_T^B$ or in $H_T^C$. Then delete the points $z_i,z_{i+1}$. Denote the resulting
TL pattern by $T'$. Replace the arcs
connecting $z_{i-1}$ with $z_i$ and $z_{i+1}$ with $z_{i+2}$ (used when we construct
$H_B^\bullet,H_C^\bullet$) with the arc connecting $z_{i-1},z_{i+2}$. We will get some subspaces
$H_B^{T'},H_C^{T'}$ with $\dim H_B^{T'}\oplus H_C^{T'}=m-1$. But $H_B^T\oplus H_C^T=
H_B^{T'}\oplus H_C^{T'}\oplus\langle e_i\rangle$ and so the dimension of $H_B^T\oplus H_C^T$
equals $m$.
\end{proof}
\begin{Prop}\label{Prop:LSubgBC}
Let $\cell\subset \dcell$ be a left cell, and let $T$ be a unique TL pattern
with $[T]=[\cell]$. Then $H_{\cell}:=H^B_T$ in type $B$ and $H_{\cell}:=H^C_T$
satisfy Lemma \ref{Lem:cells_vs_H} and Proposition \ref{Prop:Springer2}.
\end{Prop}
\begin{proof}
Take cells $\cell,\cell'$ and let $T,T'$ be the corresponding TL patterns.
We will consider type B, type C is completely analogous.

Lemma \ref{Lem:L_T_BC} implies that $\dim \Hom_W([T],[T'])=|\Lagr_{T}\cap \Lagr_{T'}|$.
On the other hand
$$\#\Coh^{\bA}(\bA/H_{\cell}\times \bA/H_{\cell'})=\frac{|\bA||H_{\cell}\cap H_{\cell'}|}{|H_{\cell}||H_{\cell'}|}.$$
The equality $|\Lagr_{T_1}\cap \Lagr_{T_2}|=\frac{|\bA||H_{\cell}\cap H_{\cell'}|}{|H_{\cell}||H_{\cell'}|}$
is a straightforward consequence of Lemma \ref{lem:lagr_subspace}.

Now let us check that Proposition \ref{Prop:Springer2} holds. Clearly, we have
$\Hom_{W}([T], \Spr(\Orb))=\Hom_{W}([T], \Spr(\Orb)^{\dcell})$.
From Proposition \ref{Prop:Springer1}, we deduce that
$\Hom_{W}([T], \Spr(\Orb)^{\dcell})$ is the direct sum of the irreducible
$\bA$-modules from $\Lagr_T\cap (\bA^B)^*$, each with multiplicity $1$.
But thanks to Lemma \ref{lem:lagr_subspace}, the latter sum is nothing else
but $\BQ (\bA^B/H^B_T)$ (as an $\bA^B$-module).
\end{proof}

In our proof of Theorem \ref{Thm:very_main} in Subsection \ref{SUBSECTION:proof_classical}
we will need the existence of left cells $\cell$ with very special Lusztig
subgroups. This is established in the following proposition.

\begin{Prop}\label{Prop:cell_existbc}
There exist left cells $\cell_0,\cell_{\varnothing},\cell_1,\ldots,\cell_m\subset \dcell$
(where $m$ is the dimension of $\bA$ over $\bF_2$) such that
$H_{\cell_0}=\bA, H_{\cell_{\varnothing}}=\{0\}, \codim_{\bA}H_{\cell_j}=1$ for all
$j$ and $\bigcap_{j=1}^m H_{\cell_j}=\{0\}$.
\end{Prop}
\begin{proof}
We are just going to present  TL patterns such that the corresponding $H^B$ and $H^C$-subgroups
have the indicated properties.

Consider the Temperley-Lieb pattern $T_B$:
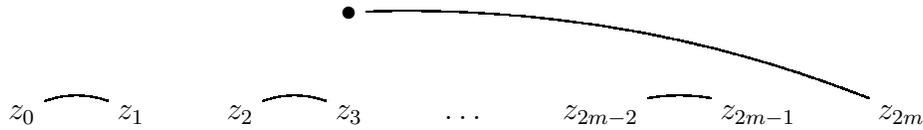
$$
\xymatrix{
&&&\bullet \ar@/^1pc/@{-}[rrrrd]&&&&\\
z_0\ar@/^.5pc/@{-}[r]&z_1&z_2\ar@/^.5pc/@{-}[r]&z_3&\ldots&z_{2m-2}\ar@/^.5pc/@{-}[r]&z_{2m-1}&z_{2m}
}
$$
It is clear that $H_{T_B}^B=\bA^B$ and $H_{T_B}^C=\{0\}$.

Similarly, consider the Temperley-Lieb pattern $T_C$:
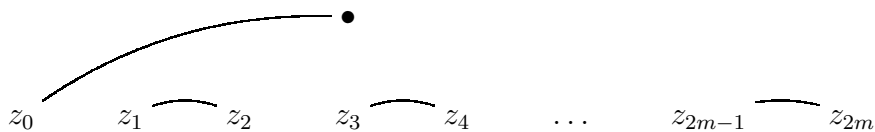
$$
\xymatrix{
&&&\bullet \ar@/_1pc/@{-}[llld]&&&&\\
z_0&z_1\ar@/^.5pc/@{-}[r]&z_2&z_3\ar@/^.5pc/@{-}[r]&z_4&\ldots&z_{2m-1}\ar@/^.5pc/@{-}[r]&z_{2m}
}
$$
It is clear that $H_{T_C}^B=\{0\}$ and $H_{T_C}^C=\bA^C$.

Now consider the patterns $T_B^j$ (and similar patterns for type $C$) given by:
$$
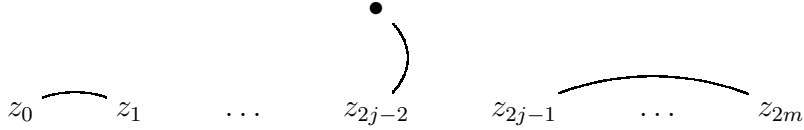
\xymatrix{
&&&\bullet \ar@/^1pc/@{-}[d]&&&\\
z_0\ar@/^.5pc/@{-}[r]&z_1&\ldots&z_{2j-2}&z_{2j-1}\ar@/^1pc/@{-}[rr]&\ldots&z_{2m}
}
$$
It is clear that $H_{T_B^j}^B$ is spanned by $e_{2i-1}, i\ne j$, and so these subgroups
have the required properties.
\end{proof}

\begin{Rem}\label{Rem:bA_quotient_realiz}
In fact, it is important for our purposes to present $\bA$ not as an abstract group
but as a quotient of $A(e), e\in \Orb$. Let us explain a recipe to compute the epimorphism
$A(e)\twoheadrightarrow \bA$. The group $A(e)$ is again the direct sum of several
copies of $\ZZ/2\ZZ$ and can be read explicitly
from the partition corresponding to $\Orb$, see, for instance, \cite{CM}, Section 6.1.
Recall that, by our definition, $\bA$ is the smallest quotient
of $A(e)$ acting on $\Spr(\Orb)^{\dcell}$. Proposition \ref{Prop:Springer1} asserts
that $\Spr(\Orb)^{\dcell}$ is the direct sum of  $U\otimes V_U$ for all $U\in \Fam(Z_1,Z_2)\cap \bA^*$.
In fact, $\bA^*\subset \Fam(Z_1,Z_2)$, because $\bA^*=\Lagr_T$ for $T=T_C$ in type $B$
and $T=T_B$ in type $C$. Pick a basis $V_1,\ldots,V_m$ in $\bA^*$. Let $U_1,\ldots,U_m$
be the corresponding irreducible $W$-modules, whose symbols we can compute explicitly,
thanks to the above description of the embedding $\Fam(Z_1,Z_2)\hookrightarrow V_{Z_1}$.
Then from the symbols we can recover the irreducible $A(e)$-modules (=$A(e)$-characters)
$\chi_1,\ldots,\chi_m$ corresponding to $U_1,\ldots,U_m$, see \cite{Carter}, 13.3.
The dual map to the epimorphism $A(e)\twoheadrightarrow \bA$ must send $V_i$ to $\chi_i$
and this determines $A(e)\twoheadrightarrow \bA$ uniquely.
\end{Rem}

\subsection{Explicit descriptions: type D}\label{SUBSECTION_explicit_d}
Recall, see \cite[\S 4.6]{orange}, that a {\em symbol} $\Lambda =\binom{M'}M=
\binom{M}{M'}$ of rank $n$ and defect 0 is an unordered pair $(M,M')$ of subsets of $\ZZ_{\ge 0}$
with $|M|=|M'|$ and $\sum_{x\in M}x+\sum_{x\in M'}x=n+|M|^2-|M|$ up to an equivalence relation
described in {\em loc. cit.}; each equivalence class contains a unique {\em reduced}
symbol with $0\not \in M\cap M'$.  Such a symbol is called {\em degenerate} if $M=M'$
(this condition does not depend on the representative from the equivalence class) and {\em non-degenerate} otherwise. To each symbol $\Lambda$ of rank $n$ and defect zero one associates a representation $[\Lambda]$ of the Weyl group $W(D_n)$ as in \cite[\S 4.6]{orange}; the representation $[\Lambda]$ is irreducible if $\Lambda$ is non-degenerate and splits into a sum of two distinct irreducible representations $[\Lambda]_I$ and $[\Lambda]_{II}$ if
$\Lambda$ is degenerate. This gives a parametrization of irreducible representations of $W(D_n)$:
they are of the form $[\Lambda]$, $[\Lambda]_I$, $[\Lambda]_{II}$ and all these representations are
distinct.

Representations $[\Lambda]_I$ and $[\Lambda]_{II}$ are special and each of them form a family by itself.
The group $\bA$ is trivial.

From now on we will consider only non-degenerate symbols $\Lambda$.
Two representations labeled by reduced symbols $\binom{M'_1}{M_1}$ and $\binom{M'_2}{M_2}$
are in the same family if and only if $M_1\cup M_1'=M_2\cup M_2'$ and
$M_1\cap M_1'=M_2\cap M_2'$, see \cite[\S 4.6]{orange}. Thus a family is completely determined by two subsets $Z_1, Z_2\subset \ZZ_{\ge 0}$ such that $Z_1\cap Z_2=\varnothing$, $0\not \in Z_2$ and $|Z_1|$ is even: the corresponding family  consists
of representations with reduced symbols $\binom{M'}M$ such that $M\cap M'=Z_2$ and $M\cup M'=Z_1\cup Z_2$; we will denote such a family by $\Fam(Z_1,Z_2)$. Let us write the elements of $Z_1$ in an increasing order: $z_1< \ldots < z_{2m}$.
Then any representation from $\Fam(Z_1,Z_2)$ corresponds to the reduced symbol
$\binom{Z_2\cup (Z_1-M)}{Z_2\cup M}$, where $M$ is a subset of $Z_1$ with $|M|=m$; thus we
have a bijection between $\Fam(Z_1,Z_2)$ and the set of subsets of $Z_1$ of cardinality $m$ modulo
an identification of $M$ and $Z_1-M$.
The special representation from $\Fam(Z_1,Z_2)$ corresponds to the subset
$M_0=\{ z_1, z_3, \ldots z_{2m-1}\} \subset Z_1$ or to the subset
$\{ z_2, z_4, \ldots z_{2m}\} \subset Z_1$.

Fix a family $\Fam(Z_1,Z_2)$ and the corresponding two-sided cell $\dcell$ (and the corresponding
special orbit $\Orb$). Let us describe the possible left cell modules.
They    are again in bijection with Temperley-Lieb patterns, now of the following form:
$$
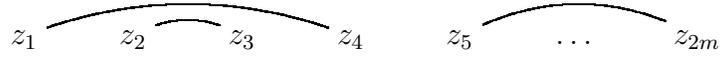
\xymatrix{
&&&&&&\\
z_1\ar@/^1pc/@{-}[rrr]&z_2\ar@/^.5pc/@{-}[r]&z_3&z_4&z_5\ar@/^1pc/@{-}[rr]&\ldots &z_{2m}
}
$$

To such a pattern $T$ one associates a
representation $[T]$ of $W(D_n)$ which is a direct sum of representations labeled by symbols
$\binom{Z_2\cup (Z_1-M)}{Z_2\cup M}$ (all with multiplicities 1) where $M\subset Z_1$ contains
precisely one of each $z_i$ and $z_j$ connected by an arc.

\begin{Ex} Assume that $Z_2=\{ 1,3\}$ and $Z_1=\{0,2,5,6,7,9\}$. Then the pattern
$$
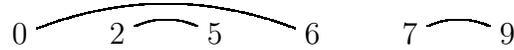
\xymatrix{
&&&&&\\
0\ar@/^1pc/@{-}[rrr]&2\ar@/^.5pc/@{-}[r]&5&6&7\ar@/^.5pc/@{-}[r]&9
}
$$
corresponds to the following representation of $W(D_{17})$:
$$\binom{0,1,2,3,7}{1,3,5,6,9}+\binom{1,2,3,6,7}{0,1,3,5,9}+\binom{0,1,3,5,7}{1,2,3,6,9}+\binom{1,3,5,6,7}{0,1,2,3,9}.$$
\end{Ex}

Let us proceed to describing the group $\bA$.
Let $V_{Z_1}$ be the set of subsets of $Z_1$  of even cardinality;
we consider it as an $\bF_2$-vector space with sum given by the symmetric difference.
The space $V_{Z_1}$ is endowed with an
alternating bilinear form $(M,M')=|M\cap M'|\pmod{2}$. The elements $e_i=\{ z_i, z_{i+1}\}$ with
$1\le i\le 2m-1$ form a basis of $V_{Z_1}$ with $(e_i,e_j)=1$ if and only if $|i-j|=1$.
Let $V_{Z_1}'$ be the quotient of $V_{Z_1}$ by the kernel of the bilinear form (which is the line spanned by $Z_1$);
clearly $V_{Z_1}'$ is a symplectic vector space spanned by the images $\bar e_i$ of $e_i$; the
only relation between the $\bar e_i$'s is $\bar e_1+\bar e_3+\ldots \bar e_{2m-1}=0$.
Let $\bA^D, \bA^{*D}\subset V_{Z_1}'$ be the Lagrangian subspaces spanned by
$\{\bar e_2, \bar e_4, \ldots, \bar e_{2m-2}\}$ and $\{\bar e_1, \bar e_3, \ldots, \bar e_{2m-1}\}$.
It is clear that $V_{Z_1}'=\bA^D\oplus \bA^{*D}$ and the symplectic form gives the identification $(\bA^D)^*=\bA^{*D}$.

The set $\Fam(Z_1,Z_2)$ is embedded in $V_{Z_1}'$ as follows:
the representation labeled by the symbol $\binom{Z_2\cup (Z_1-M)}{Z_2\cup M}$ corresponds to the symmetric difference of
$M$ and $M_0$.
The Lusztig's quotient $\bA$ associated with family $\Fam(Z_1,Z_2)$ is the group $\bA^D$,
see \cite[\S 4.6]{orange}.
So $\Fam(Z_1,Z_2)\subset V_{Z_1}'=\bA^D\oplus (\bA^D)^*=\Mfr(\bA^D)$, which is precisely
Lusztig's embedding from \cite{orange} recalled in Subsection \ref{SUBSECTION_cells}.

Let $T$ be a Temperley-Lieb pattern as above; then the subspace $\Lagr_T$ of $V_{Z_1}'$ spanned
by two element subsets $\{ z_i,z_j\}$ where $z_i$ and $z_j$ are connected by an arc is Lagrangian.
It is clear from the above that the representation $[T]$ is the direct sum
of the representations corresponding to the elements of $\Lagr_T$.

We now describe the subgroup $H_\cell \subset \bA$ corresponding to a left cell $\cell$.
The subgroup  $H_\cell$ again depends only on $[\cell]$; thus we will use the  notation $H_T$ for
$H_\cell$, where $[\cell]=[T]$ for a Temperley-Lieb pattern $T$.
Let us connect $z_2$ and $z_3$, $z_4$ and $z_5$ and so on. A basis of
$H_T\subset \bA^D$ is labeled by the connected components of the resulting picture
homeomorphic to a circle; the basis element corresponding to such a connected component is
$\sum \bar e_{2i}$, where the summation is over the indices $i$ such that the arc connecting $z_{2i}, z_{2i+1}$ appears in that connected component.

\begin{Ex} Consider the pattern
$$
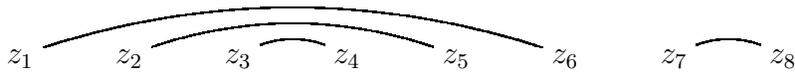
\xymatrix{
&&&&&&&\\
z_1\ar@/^1.5pc/@{-}[rrrrr]&z_2\ar@/^1pc/@{-}[rrr]&z_3\ar@/^.5pc/@{-}[r]&z_4&z_5&z_6&z_7\ar@/^.5pc/@{-}[r]&z_8
}
$$

The procedure above gives:
$$
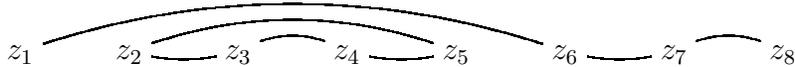
\xymatrix{
&&&&&&&\\
z_1\ar@/^1.5pc/@{-}[rrrrr]&z_2\ar@/^1pc/@{-}[rrr]\ar@/_.25pc/@{-}[r]&z_3\ar@/^.5pc/@{-}[r]&z_4\ar@/_.25pc/@{-}[r]&z_5&
z_6\ar@/_.25pc/@{-}[r]&z_7\ar@/^.5pc/@{-}[r]&z_8
}
$$
Thus, $H_T=\langle \bar e_2+\bar e_4\rangle \subset \langle \bar e_2,\bar e_4,\bar e_6\rangle$.
\end{Ex}

\begin{Rem}
One observes that the orthogonal complement $H_T^\perp \subset (\bA^D)^*\simeq \bA^{*D}$ can
be computed in terms of $T$ as follows: connect $z_1$ and $z_2$, $z_3$ and $z_4$ etc.
Then  $H_T^\perp$ is spanned by the elements $\sum \bar e_{2i-1}$ labeled by the connected
components of the resulting picture with the summation over the indices $i$ such that the arc
connecting $z_{2i-1}, z_{2i}$ appears in that connected component. In addition, one
has $\Lagr_T=H_T\oplus H_T^\perp \subset \bA^D\oplus \bA^{*D}\simeq V_{Z_1}'$.
\end{Rem}

\begin{Rem} \label{D=BC}
The combinatorics described in this section is completely parallel to
the combinatorics in type $B$ via the following transformation of Temperley-Lieb patterns:
a type $D$ pattern $T$ with $Z_1=\{ z_1, \ldots ,z_{2m}\}$
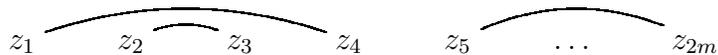
$$
\xymatrix{
&&&&&&\\
z_1\ar@/^1pc/@{-}[rrr]&z_2\ar@/^.5pc/@{-}[r]&z_3&z_4&z_5\ar@/^1pc/@{-}[rr]&\ldots &z_{2m}
}
$$
corresponds to the type $B$ pattern $T'$ with $Z_1=\{ z_2,\ldots z_{2m}\}$.
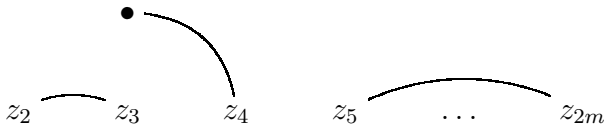
$$
\xymatrix{
&&\bullet \ar@/^1pc/@{-}[rd]&&&\\
&z_2\ar@/^.5pc/@{-}[r]&z_3&z_4&z_5\ar@/^1pc/@{-}[rr]&\ldots &z_{2m}
}
$$
More precisely, consider the $\bF_2$-linear isomorphism $\iota:\bA^B\xrightarrow{\sim} \bA^D$ given
by $e_{2i-1}\rightarrow \bar{e}_{2i}, i=1,\ldots,m$ and the induced isomorphism $\iota:\bA^B\oplus (\bA^B)^*
\xrightarrow{\sim} \bA^D\oplus (\bA^D)^*$. This isomorphism maps the lagrangian subspace $\Lagr_{T'}\subset
\bA^B\oplus (\bA^B)^*$ to $\Lagr_T$ and so $H^B_{T'}$ to $H_{T}$.
\end{Rem}

Thanks to the previous remark, Proposition \ref{Prop:LSubgBC} implies that Proposition \ref{Prop:Springer2}
and Lemma \ref{Lem:cells_vs_H} hold for the subgroups $H_T\subset \bA^D$.
Also we have a complete analog of Proposition \ref{Prop:cell_existbc}.

\begin{Prop}\label{Prop:cell_existd}
There exist left cells $\cell_0,\cell_{\varnothing},\cell_1,\ldots,\cell_m\subset \dcell$
(where $m$ is the dimension of $\bA$ over $\bF_2$) such that
$H_{\cell_0}=\bA, H_{\cell_{\varnothing}}=\{0\}, \codim_{\bA}H_{\cell_j}=1$ for all
$j$ and $\bigcap_{j=1}^m H_{\cell_j}=\{0\}$.
\end{Prop}

\subsection{Explicit descriptions: exceptional cases with $\bA=\ZZ/2\ZZ$}\label{SUBSECTION_expl_excep_2}
%
Recall that one can describe $\bA$ in terms of $\Spr(\Orb)^{\dcell}$.
Inspecting  the list of left cell modules, \cite[13.2]{Carter}, and tables in \cite[13.3]{Carter}  giving the Springer representations, we get the following lists of special orbits with
$\bA=\ZZ/2\ZZ$:
\begin{itemize}
 \item[$G_2$:] no.
\item[$F_4$:] $\tilde{A}_1,F_4(a_1)$.
\item[$E_6$:] $A_2, E_6(a_3)$.
\item[$E_7$:] $A_2, A_2+A_1, D_4(a_1)+A_1, A_4, A_4+A_1, D_5(a_1), E_6(a_3), E_6(a_1), E_7(a_3)$.
\item[$E_8$:] $A_2, A_2+A_1, 2A_2, A_4, D_4(a_1)+A_2, A_4+A_1, D_5(a_1), A_4+2A_1, E_6(a_3), $
$D_6(a_1), E_6(a_1), $ $D_7(a_2), E_6(a_1)+A_1, E_7(a_3), E_8(a_5), E_8(a_4),E_8(a_3)$.
\end{itemize}

All modules $\Spr(\Orb)^{\dcell}$ have the form $U_1\otimes V_1\oplus U_2\otimes V_{\epsilon}$,
where $V_1,V_\epsilon$ are the trivial and the sign modules for $\bA$.

There are two possible behaviors of the left cell modules. First, for all but 3 orbits
there are two different left cell modules and they have the form $U^1:=U_1\oplus U_2, U^2:=U_1\oplus \tilde{U}_2$.
Here $\tilde{U}_2$ is an irreducible $W$-module different from $U_1,U_2$. Define $H_{\cell}=\{1\}$
if $[\cell]=U^1, H_{\cell}=\bA$ if $[\cell]=U^2$. It is straightforward to see that these subgroups
satisfy Proposition \ref{Prop:Springer2} and Lemma \ref{Lem:cells_vs_H}.

Now, for the orbits $A_4+A_1$ in $E_7$, $A_4+A_1, E_6(a_1)+A_1$ in $E_8$ there is only one cell module
$U_1\oplus U_2$. These orbits (and the corresponding two-sided cells) are called {\it exceptional}.
All subgroups $H_{\cell}$ are equal to $\{1\}$.


\subsection{Explicit computations: exceptional cases with $\bA=S_3,S_4,S_5$}\label{SUBSECTION_explicit_excep2}
\subsubsection{$\bA=S_3$}
We have the following numbers of special orbits with $\bA=S_3$:
one for $G_2$, one for $E_6$, two for $E_7$, six for $E_8$.

First, we consider the $G_2$ case. Here the Springer representation has the form
$U_0\otimes V_3\oplus U_1\otimes V_{21}$. Here $U_0,U_1$ are certain irreducible
$W$-modules, $U_0$ is special, and for a partition $\lambda$ of $3$ we write $V_\lambda$
for the correspoding irreducible $S_3$-module. There are two possibilities for the
left cell modules as follows: $U^1:=U_0+U_1+\tilde{U}_1, U^2:=U_0+\tilde{U}_1+\tilde{U}_2$, where $\tilde{U}_1,\tilde{U}_2$
are some other irreducible $W$-modules different from $U_0,U_1$. Set $H_\sigma:=S_2$ if
$[\sigma]=U^1$, and $H_\sigma:=S_3$ if $[\sigma]=U^2$. Proposition \ref{Prop:Springer2}
and Lemma \ref{Lem:cells_vs_H} hold for this choice of subgroups.

In the remaining cases the situation is uniform.
There the Springer representation has the form
$U_0\otimes V_3+U_1\otimes V_{21}+U_2\otimes V_{111}$. There are three possibilities
for the left cell modules $U^1:=U_0+U_1^{\oplus 2}+U_2, U^2:=U_0+U_1+\tilde{U}_1, U^3:=U_0+\tilde{U}_1+\tilde{U}_2$.
We have $H_\sigma=\{1\},S_2,S_3$ when $[\sigma]=U^1,U^2,U^3$, respectively. Again, Proposition \ref{Prop:Springer2}
and Lemma \ref{Lem:cells_vs_H} hold.

\subsubsection{$\bA=S_4$}
There is only one special orbit with this $\bA$ and it is in type $F_4$. The Springer representation is
of the form $U_0\otimes V_4+U_1\otimes V_{31}+U_2\otimes V_{22}+ U_3\otimes V_{211}$, in the notation
similar to the $S_3$-part.

There are five different left cell modules:
\begin{itemize}
 \item[(i)] $U^1=U_0+\tilde{U}_1+\tilde{U}_2+\tilde{U}_3+\tilde{U}_4$,
\item[(ii)] $U^2=U_0+\tilde{U}_1+U_1+\tilde{U}_3+\tilde{U}_5$.
\item[(iii)] $U^3=U_0+\tilde{U}_1+2\tilde{U}_2+U_2+\tilde{U}_4+\tilde{U}_6$.
\item[(iv)] $U^4=U_0+\tilde{U}_1+2U_1+U_2+\tilde{U}_5+U_3$.
\item[(v)] $U^5=U_0+2\tilde{U}_1+U_1+\tilde{U}_2+U_2+\tilde{U}_7$.
\end{itemize}

Here, as above, $\tilde{U}_1,\ldots,\tilde{U}_7$ are pairwise non-isomorphic
irreducible $W$-modules different from $U_0,\ldots,U_3$.

For $[\cell]=U^i$, we set $H_{\cell}=H^i$,
where  $H^1:=S_4,H^2:=S_3,H^3:=\operatorname{Dyh}_8,H^4:=S_2,H^5:=S_2\times S_2$.
Here $\operatorname{Dyh}_8$ is the dihedral group of order 8. It is not difficult
but tedious to check that Proposition \ref{Prop:Springer2} and Lemma \ref{Lem:cells_vs_H}
hold for this choice of subgroups.

\subsubsection{$\bA=S_5$}
Here $\Spr(\Orb)$ has the form $U_0\otimes V_5\oplus U_1\otimes V_{41}+U_2\otimes V_{32}\oplus
U_3\otimes V_{311}\oplus U_4\otimes V_{221}\oplus U_5\otimes V_{2111}$.
There are 7 left cell modules $U^1,\ldots, U^7$, see \cite{Carter}, the very end of 13.2, we need the modules
including the special representation $\phi_{4480,16}$. They correspond  to the following
Lusztig subgroups $H^1:=S_5,H^2:=S_4, H^3:=S_3\times S_2, H^4:=\operatorname{Dyh}_8,
H^5:=S_3, H^6:=S_2\times S_2, H^7:=S_2$. Again, it is  not very difficult but very tedious to check
that both Proposition \ref{Prop:Springer2} and Lemma \ref{Lem:cells_vs_H} hold in this situation.

\subsection{Classification of left cells}\label{SUBSECTION_classification}

\begin{Lem} \label{TeLi}
 For any finite Weyl group $W$ the cell modules are linearly independent as virtual
representations of $W$.
\end{Lem}

\begin{proof} Since cell modules corresponding to different families are disjoint it is enough to show
that the statement is true for the cell modules from a fixed family. We can assume that $W$ is irreducible
(since families and cell modules for $W'\times W''$ are just tensor products of those for $W'$ and $W''$).
For $W$ of type $A$ the statement is trivial; for $W$ of exceptional types one verifies it by a direct check.

For $W$ of types $B$ and $C$ we compute the matrix of scalar products of (characters of)
cell modules. It follows from description in Section \ref{SUBSECTION_explicit_bc}
that the scalar product of cell modules corresponding to TL patterns $T'$ and
$T''$ can be computed as follows: turn the TL pattern $T''$ upside down and concatenate it with TL
pattern $T'$. The scalar product of cell modules corresponding to $T'$ and $T''$ is $2^d$ where $d$
is the number of compact components (i.e. circles) in the resulting picture.  One observes that
the resulting matrix is, by definition, the matrix of the composition pairing $\Hom([1],[2n+1])\times
\Hom([2n+1],[1])\to \Hom([1],[1])=\BQ$ in the so-called {\em Temperley-Lieb category}
${\mathcal TL}(\tau)$ with $\tau=1$, see e.g. \cite{GW}. Since the category ${\mathcal TL}(1)$
is semisimple (see {\em loc. cit.}), the pairing above is non-degenerate and the lemma
is proved in this case.

Finally, for $W$ of type $D$ one proceeds similarly, or uses Remark \ref{D=BC}.
\end{proof}

Recall that $[\dcell]=\bigoplus_{\cell\subset \dcell}[\cell]$, see Section \ref{SUBSECTION_cells}.
Clearly, $[\dcell]=\bigoplus_{E\in \Irr(W)^\dcell}\dim(E)E$. Thus Lemma \ref{TeLi} guarantees that
we can compute the number of left cells $\sigma$ with a given cell module. In particular, we can
compute explicitly the Lusztig's set $Y'=\bigsqcup_{\cell}\bA/H_\cell$, see Section \ref{hsigma}.

\begin{Ex} (a) Assume that $\g=\mathfrak{so}(15)$ and the (special) nilpotent orbit $\Orb$ corresponds to partition $(7,3^2,1^2)$. Using \cite[\S 13.3]{Carter} one computes that the corresponding special
representation of $W$ has a symbol $\binom{0,2,5}{1,3}$. Here is a table of dimensions of the
representations from the corresponding family:
\begin{equation*}
\begin{array}{c|c|c|c|c|c|c|c|c|c}
\binom{0,2,5}{1,3}&\binom{0,1,2}{3,5}&\binom{0,1,3}{2,5}&\binom{0,1,5}{2,3}&
\binom{0,2,3}{1,5}&\binom{0,3,5}{1,2}&\binom{1,2,3}{0,5}&\binom{1,2,5}{0,3}&
\binom{1,3,5}{0,2}&\binom{2,3,5}{0,1}\\
\hline
210&14&63&70&84&105&35&126&112&21\\
\end{array}
\end{equation*}

The representations $\binom{0,1,2}{3,5}$ and $\binom{2,3,5}{0,1}$ appear only in the cell modules
corresponding to the following TL patterns:

$$
T_1=
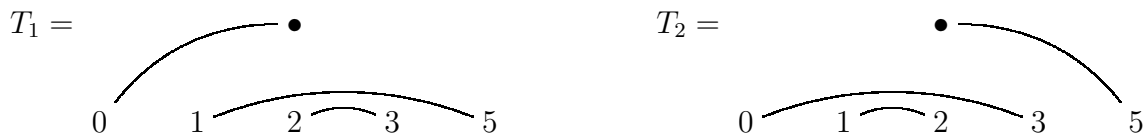
\xymatrix{
&&\bullet \ar@/_1pc/@{-}[lld]&&\\
0&1\ar@/^1pc/@{-}[rrr]&2\ar@/^.5pc/@{-}[r]&3&5
} \phantom{blankspace}
T_2=\xymatrix{
&&\bullet \ar@/^1pc/@{-}[rrd]&&\\
0\ar@/^1pc/@{-}[rrr]&1\ar@/^.5pc/@{-}[r]&2&3&5
}
$$

Thus we will have $\dim \binom{0,1,2}{3,5}=14$ left cells with left cell module described by $T_1$ and
$\dim \binom{2,3,5}{0,1}=21$ left cells of type $T_2$. Similarly, using representation
$\binom{1,2,3}{0,5}$ we find that there are 35 left cells of type $T_3$ where
$$
T_3=
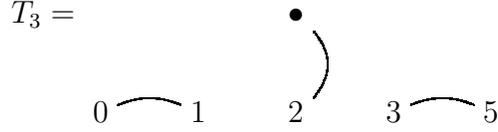
\xymatrix{
&&\bullet \ar@/^1pc/@{-}[d]&&\\
0\ar@/^.5pc/@{-}[r]&1&2&3\ar@/^.5pc/@{-}[r]&5
}
$$

Two remaining TL patterns are
$$
T_4=
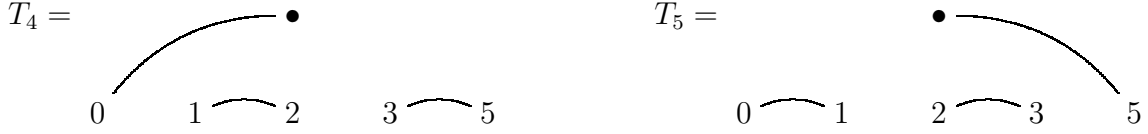
\xymatrix{
&&\bullet \ar@/_1pc/@{-}[lld]&&\\
0&1\ar@/^.5pc/@{-}[r]&2&3\ar@/^.5pc/@{-}[r]&5
} \phantom{blankspace}
T_5=\xymatrix{
&&\bullet \ar@/^1pc/@{-}[rrd]&&\\
0\ar@/^.5pc/@{-}[r]&1&2\ar@/^.5pc/@{-}[r]&3&5
}
$$
The representation $\binom{0,1,5}{2,3}$ appears only in left cell modules corresponding to $T_4$ and
$T_2$; thus the number of left cells of type $T_4$ is $\dim \binom{0,1,5}{2,3}-21=49$. Similarly,
using representation $\binom{0,3,5}{1,2}$ we find that the number of left cells of type $T_5$ is 91.

Using the descriptions of the group $\bA$ and the  subgroups $H_\sigma$ from
Section \ref{SUBSECTION_explicit_bc}, we find that
$\bA=\langle e_1,e_3\rangle$ and
$$Y'=(\bA/\langle e_3\rangle)^{14}\sqcup (\bA/\langle e_1+e_3\rangle)^{21}\sqcup
(\bA/\langle e_1\rangle)^{35}\sqcup (\bA)^{49}\sqcup (\bA/\bA)^{91}.
$$
In particular, the cardinality of $Y'$ is $14\cdot 2+21\cdot 2 +35\cdot 2+49\cdot 4+91=427$.

(b) Let $\g=\mathfrak{sp}(14)$ and $\Orb$ corresponds to partition $(6,4,2^2)$. The corresponding special
representation of $W$ again has a symbol $\binom{0,2,5}{1,3}$ and using almost the same
computation as above we find that $\bA=\langle e_2,e_4\rangle$ and
$$Y'=(\bA/\langle e_2+e_4\rangle)^{14}\sqcup (\bA/\langle e_2\rangle)^{21}\sqcup
(\bA/\langle e_4\rangle)^{35}\sqcup (\bA/\bA)^{49}\sqcup (\bA)^{91}.
$$

(c) Let $\Orb$ be a special orbit such that the corresponding cell $\dcell$ is non-exeptional
with $\bA=\ZZ/2\ZZ$. The corresponding family contains 3 representations of $W$: the special
representation $U_0$, the representation $U_1$ which appears in $\Spr(\Orb)$, and one more
representation $\tilde{U}_1$. It is well known (and it follows from \cite{cellules})
that $\dim(U_0)=\dim(U_1)+\dim(\tilde{U}_1)$. In all such cases we have
$$Y'=(\bA)^{\dim(U_1)}\sqcup (\bA/\bA)^{\dim(\tilde{U}_1)}.$$

(d) We list here all cases where $\bA =S_r$, $3\le r\le 5$. We follow \cite[\S 13]{Carter} in the notation
for nilpotent orbits.

$\g$ is of type $G_2$ and $\Orb$ is of type $G_2(a_1)$: in this case $\bA=S_3$ and
$Y'=S_3/S_2 \sqcup S_3/S_3$.

$\g$ is of type $E_6$ and $\Orb$ is of type $D_4(a_1)$: in this case $\bA=S_3$ and
$$Y'=(S_3)^{20} \sqcup (S_3/S_3)^{10}\sqcup (S_3/S_2)^{50}.$$

$\g$ is of type $E_7$ and $\Orb$ is of type $D_4(a_1)$: in this case $\bA=S_3$ and
$$Y'=(S_3)^{35} \sqcup (S_3/S_3)^{70}\sqcup (S_3/S_2)^{210}.$$

$\g$ is of type $E_8$ and $\Orb$ is of type $D_4(a_1)$ or $E_8(b_5)$: in both cases $\bA=S_3$ and
$$Y'=(S_3)^{56} \sqcup (S_3/S_3)^{448}\sqcup (S_3/S_2)^{896}.$$

$\g$ is of type $E_8$ and $\Orb$ is of type $D_4(a_1)+A_1$ or $E_8(a_6)$: in both cases $\bA=S_3$ and
$$Y'=(S_3)^{350} \sqcup (S_3/S_3)^{175}\sqcup (S_3/S_2)^{875}.$$

$\g$ is of type $F_4$ and $\Orb$ is of type $F_4(a_3)$: in this case $\bA=S_4$ and
$$Y'=(S_4/S_4)^3\sqcup (S_4/S_3)^3 \sqcup (S_4/S_2\times S_2)^4\sqcup S_4/S_2 \sqcup
S_4/\mbox{Dyh}_8.$$

$\g$ is of type $E_8$ and $\Orb$ is of type $E_8(a_7)$: in this case $\bA=S_5$ and
$$Y'=(S_5/S_5)^{420}\sqcup (S_5/S_4)^{756}\sqcup (S_5/\mbox{Dyh}_8)^{168}\sqcup (S_5/S_2)^{70}
\sqcup(S_5/S_3\times S_2)^{1596}$$
$$\sqcup (S_5/S_2\times S_2)^{1092} \sqcup (S_5/S_3)^{378}.$$
\end{Ex}

Above we have determined (in some form, at least) the number of left cells with given cell
module. One, however, can ask how to compute the cell module starting from a cell itself.
In fact, in classical types there is a combinatorial classification of left cells
due to Barbash and Vogan, \cite{BV_class}. To each $w\in W$ they combinatorially assigned a standard $2n$-tableau.
Then they proved that $w,w'$ are in the same left cell if their tableaux coincide and provided
a recipe to determine $\VA(\U/\J(w\rho))$ from the tableau of $w$. Garfinkle in a series of papers
produced an equivalent (but simpler) combinatorial classification of left cells in terms of combinatorial objects
standard  domino $n$-tableaux of special form that are again produced
from Weyl group elements, see \cite{Ga} for the definition.  In \cite{McGovern_class}
McGovern found a combinatorial recipe to produce the cell module from the Garfinkle tableau
corresponding to a left cell. So the conclusion is that one can combinatorially compute the group
$H_\sigma$ from starting from an element $w\in \sigma$.

\section{Proof of the main theorem}\label{SECTION:proof_main}
\subsection{Results of Dodd: $\Irr_{fin}(\Walg_\rho)$ vs $\Spr(\Orb)$}\label{SUBSECTION_Dodd}
In this subsection we will quote results of Dodd, \cite{Dodd}, relating the
$K$-group $[\Walg_{\lambda}-\operatorname{mod}_{fd}]$ of the category of finite dimensional
 modules over the central reduction $\Walg_\lambda$ for some central character $\lambda$ and the Springer
$W\times A(\Orb)$-module $\Spr(\Orb)$.

In \cite[\S 3]{Dodd}, Dodd defined a natural $A(\Orb)$-equivariant map $[\Walg_{\lambda}-\operatorname{mod}_{fd}]\rightarrow \Spr(\Orb)$; moreover he proved that this map is an embedding, see \cite[Theorem 1]{Dodd}.
Furtermore, for $\lambda =\rho$ he used Goodwin's translation functors from \cite{Goodwin} to define
$W-$action on $[\Walg_{\rho}-\operatorname{mod}_{fd}]$ and showed that the map above is
$W-$equivariant with respect to the standard $W-$action on $\Spr(\Orb)$, see \cite[\S 8]{Dodd}.
Recall that $[\Walg_{\rho}-\operatorname{mod}_{fd}]$ carries another $W-$action coming from the
epimorphism $\BQ (W)\twoheadrightarrow  [\JCat_{\Orb}]$, see \S \ref{SUBSECTION_cell_HC}.

\begin{Prop}
\label{two actions}
The two $W-$actions on $[\Walg_{\rho}-\operatorname{mod}_{fd}]$ described above coincide.
\end{Prop}


\begin{proof}
From Theorem \ref{Thm_tensor} (applied in the case, when $T$ is trivial and so $\theta=0$)
it follows that this action coming from the Goodwin translation functors
is the same as the one coming from \begin{equation}\label{eq:cat_action}^{\rho}\!\HC(\U)^{\rho}\times \Walg^{\rho}-\operatorname{mod}_{fd}\rightarrow \Walg^{\rho}-\operatorname{mod}_{fd}, (X,N)\mapsto X_\dagger\otimes_{\Walg}N\end{equation} (where the superscript $\rho$ means the generalized central  character $\rho$). We claim that
 the action $\JCat_{\Orb}\times \Walg_{\rho}-\operatorname{mod}_{fd}\rightarrow \Walg_{\rho}-\operatorname{mod}_{fd}$
induced by (\ref{eq:cat_action})
is compatible with the epimorphism $\BQ (W)\twoheadrightarrow  [\JCat_{\Orb}]$
from Subsection \ref{SUBSECTION_cell_HC}.
Indeed, consider the category $\hat{\,^{\rho}\!\HC^\rho}(\U)$
whose objects have the form $\varprojlim_{n,m}X_{n,m}$, where $X_{n,m}$ is a HC bimodule annihilated by $\Centr(\U)_\rho^n$
on the left and $\Centr(\U)_\rho^m$ on the right. In particular, $\,^{\rho}\!\HC^{\rho}(\U)\subset \hat{\,^{\rho}\!\HC^\rho}(\U)$.
An advantage of $\hat{\,^{\rho}\!\HC^\rho}(\U)$ over $\,^{\rho}\!\HC^{\rho}(\U)$ is that the former has enough projective objects.
The category $\hat{\,^{\rho}\!\HC^\rho}(\U)$ is monoidal: its objects are bimodules over the completion
$\varprojlim_n \U/\U\Centr(\U)_\rho^n$, and we take tensor products of bimodules over this algebra. The projective objects of this category are identified with projective functors sending the generalized
central character $\rho$ to itself
so that the tensor product becomes the composition, see \cite{BG}.
As we have already mentioned in Subsection \ref{SUBSECTION_cell_HC}, the $\BQ$-form of the Grothendieck ring of $\hat{\,^{\rho}\!\HC^\rho}(\U)-\operatorname{proj}$ is $\BQ W$.

For $X\in \hat{\,^{\rho}\!\HC^\rho}(\U), Y\in \HC_{\Orb}(\U), N\in \Walg$-$\operatorname{mod}$
we have $X_\dagger\otimes_{\Walg} (Y_\dagger\otimes_\Walg  N)=(X\otimes_\U  Y)_\dagger\otimes_{\Walg}N$ by Theorem \ref{Thm_dagger} (4).
Now let $N$ be a simple in $\Walg_\rho$-$\operatorname{mod}$, $Y={\bf 1}\in\JCat_\Orb$ be
the unit object, and $X\in \hat{\,^{\rho}\!\HC^\rho}(\U)$. Then, of course, $Y\otimes_{\Walg}N=N$ and so we have $(X\otimes_\U {\bf 1})_\dagger \otimes_{\Walg}N=X_\dagger\otimes_{\Walg}N$. On the level of the Grothendieck groups, the map $\BQ W\rightarrow \JAlg_{\Orb}$ that sends the class of $X\in \hat{\,^{\rho}\!\HC^\rho}(\U)-\operatorname{proj}$ to the class of $X\otimes {\bf 1}$ is just Lusztig's homomorphism, see \eqref{catLus}.

So Dodd's action of $\BQ (W)$ on $[\Walg_{\rho}-\operatorname{mod}_{fd}]$ is the same as ours.
\end{proof}

\begin{Rem}
We sketch here an alternative proof of the Dodd's result on the injectivity of the map
$[\Walg_{\rho}-\operatorname{mod}_{fd}]\rightarrow \Spr(\Orb)$ based on
some results of Lusztig from \cite{Lusztig_families}. Namely, Dodd considers a map $[\Walg_{\rho}-\operatorname{mod}_{fd}]
\rightarrow H^*(\mathcal{B}_e)$
that is the composition of the reduction mod $p$ for $p\gg 0$
and the map from the Grothendieck group of the restricted representations of the W-algebra in characteristic
$p$ to $H^*(\mathcal{B}_e)\cong \BQ\otimes_{\ZZ}K_0(\operatorname{Coh}(\mathcal{B}_e))$. The first map is injective
from the construction. The second is an isomorphism, it comes from the derived localization, that is an equivalence
of the corresponding derived categories. So the map $[\Walg_{\rho}-\operatorname{mod}_{fd}]
\rightarrow H^*(\mathcal{B}_e)$ is injective. Dodd's main result is that the image of that map
actually lies in the top cohomology, \cite[\S 7]{Dodd}.
Independently, the map above is also shown to be $W$-equivariant, see \cite[\S 8.2]{Dodd}.
Proposition \ref{two actions} implies that irreducible constituents of its image are in $\Irr(W)^{\Orb}=
\Irr(W)^{\dcell}$.
But according to \cite[Proposition 0.2]{Lusztig_families},
if a representation from $\Irr(W)^{\dcell}$ appears in $H^*(\mathcal{B}_e)$, it only appears in the top degree, i.e., in $\operatorname{Spr}(e)$.
So the W-equivariance actually implies the injectivity.
\end{Rem}

\subsection{Summary} \label{summary}
In this subsection we are going to summarize the results obtained so far in the form
suitable for the proof of the main theorem. We start by interpreting results
from Subsection \ref{SUBSECTION_primitive}, Theorem \ref{Thm_main}, Proposition \ref{Prop:int_hom}.
Until Subsection  \ref{SUBSECTION_nonint} we assume that $G$ is of adjoint type.

Fix a finite set $\Lambda$ of dominant weights (for $\g$) containing a regular weight $\varrho$.
Let $\Walg^{ss}_\Lambda$ be the quotient of $\Walg$ by the intersection of
all maximal ideals of finite codimension in $\Walg$ with central characters from $\Lambda$.
This is a finite dimensional semisimple associative algebra equipped with a Hamiltonian action of $Q$.
So we can consider the category  $\Bimod^Q(\Walg^{ss}_\Lambda)$ of finite dimensional $Q$-equivariant
bimodules. The $Q$-equivariance condition implies, in particular, that the $Q^\circ$-action on an
object of $\Bimod^Q(\Walg^{ss}_\Lambda)$ is recovered from the adjoint $\Walg^{ss}_\Lambda$-action. Pick
 a finite subgroup  $\fA\subset Q$ surjecting onto $A(\Orb)=Q/Q^\circ$. The existence of such a
subgroup is a standard fact, see, e.g., \cite[Proposition 7]{Vinberg}. Then we have a natural inclusion
$\Bimod^Q(\Walg^{ss}_\Lambda)\hookrightarrow \Bimod^{\fA}(\Walg^{ss}_\Lambda)$. The category
$\Bimod^{\fA}(\Walg^{ss}_\Lambda)$ is known to be isomorphic to $\Coh^{\fA,\psi}(Y^\Lambda\times Y^{\Lambda})$,
for an appropriate collection of 2-cocycles $\psi$, see \cite[\S 5.1]{BeO}.

Lemma \ref{Lem:J_MF} and Proposition \ref{Prop:J_indecomp} imply the following
\begin{itemize}
\item[(A1)] There is an embedding $\,^{\Lambda}\!\JCat_{\Orb}^\Lambda\hookrightarrow \Coh^{\fA,\psi}(Y^\Lambda\times Y^{\Lambda})$
of multi-fusion categories. Moreover, $^{\Lambda}\!\JCat_{\Orb}^\Lambda$ is indecomposable.
\end{itemize}

Let $\YCat^\Lambda:=\Coh(Y^\Lambda)$. We can view $\YCat^\Lambda$ as a left $\,^{\Lambda}\!\JCat_{\Orb}^\Lambda$-module
 via the embedding $\,^{\Lambda}\!\JCat^\Lambda_{\Orb}\hookrightarrow
 \Coh^{\fA,\psi}(Y^\Lambda\times Y^\Lambda)$. Theorem \ref{Thm_main} means
\begin{itemize}
\item[(A2)] The left $\,^{\Lambda}\!\JCat_{\Orb}^\Lambda$-module $\YCat^\Lambda$ is indecomposable.
\end{itemize}

Next, Proposition \ref{Prop:W_translation} gives
\begin{itemize}
\item[(A3)] We have $\bH_\cell\subset \bH_\cell^\lambda$ for any compatible $\cell$ and $\lambda$.
\end{itemize}

Now let us interpret results from Section \ref{SECTION_categories}. Lemma \ref{Lem:0.2}
together with assertion (iii) of Lemma \ref{Lem:0.3pre} imply the following statement.
\begin{itemize}
\item[(B1)] There is a quotient $\bfA$ of $\fA$ and a class $\omega\in H^3(\bfA, \K^\times)$
independent of $\Lambda$ such that the action of $\fA$ on $Y^\Lambda$ factors through $\bfA$
and $\,^{\Lambda}\!\JCat_{\Orb}^\Lambda=(\Vec_{\bfA}^\omega)^*_{\YCat^\Lambda}$. For the
decomposition $Y^\Lambda=\bigsqcup_{\lambda\in \Lambda, \sigma} Y^{\lambda}_{\sigma}$ into
$\bfA$-orbits we have $\Coh(Y^\lambda_\sigma)=\YCat^{\lambda}_{\sigma}:=e^\lambda_\sigma\otimes \YCat^\Lambda$.
The latter are indecomposable right $\Vec_{\bfA}^\omega$-modules.
\end{itemize}

Further, applying Lemmas \ref{Lem:0.3pre}, \ref{Lem:0.3}, \ref{Lem:cells_vs_H} we get
\begin{itemize}
\item[(B2)] For any dominant weights $\lambda,\mu$ and any left cells $\sigma,\cellb$
compatible with $\lambda,\mu$ (meaning that $w$ and $\lambda$ are compatible for each $w\in \sigma$
and similarly for $\cellb,\mu$), respectively, the following numbers are all equal:
\begin{enumerate}
\item $\#\,_\sigma^\lambda\!\JCat_\cellb^\mu$.
\item $\Hom_{[\,^\Lambda\!\JCat^\Lambda]}([\JCat_\sigma^\lambda], [\JCat_\cellb^\mu])$.
\item $\Hom_W([\sigma],[\cellb])$.
\item $\#\Coh^{\bA}(\bA/H_\sigma\times \bA/H_\cellb)$.
\item $\#\Coh^{\bfA,\psi}(Y^\lambda_\sigma\times Y^\mu_\cellb)$.
\end{enumerate}
\end{itemize}
In more detail, the coincidence of (1) and (2) follows from Lemma \ref{Lem:0.3}
(applied to the direct summand $e_{\mu}^{\cellb}$ of unit and
$\MCat:=\,^{\Lambda}\!\J^{\Lambda}\otimes e_{\lambda}^{\sigma}$). The coincidence
of (2) and (3) follows from Subsection \ref{SUBSECTION_cell_HC}.
The coincidence of (3) and (4) is Lemma \ref{Lem:cells_vs_H}.
Finally, the coincidence of (1) and (5) is Lemma \ref{Lem:0.3pre}(iii) (applied to the direct summands $e^\lambda_{\sigma},e^{\mu}_{\cellb}$
of $\be$ and $\MCat:=\YCat^{\Lambda}$; here we interpret the category in (5) as $\Fun_{\Vec_{\bfA}^\omega}(e^{\mu}_{\cellb}\MCat, e^\lambda_{\sigma}\MCat)$).

Now we are going to apply Lemma \ref{Lem:0.3pre},(i), and Lemma \ref{Lem:0.3} to $e_{\sigma}(:=e_\sigma^\rho)$ and $\MCat=\YCat$.
\begin{itemize}
\item[(B3)] For any dominant weight $\lambda$ and a left
cell $\sigma$ compatible the following $\fA$-modules coincide:
\begin{enumerate}
\item $\Hom_{W}([\sigma],[\YCat])$, where $W$ acts on $[\YCat]$
as explained in Subsection \ref{SUBSECTION_Dodd}.
\item $[\Fun_{\JCat_{\Orb}}(\JCat^\sigma,\YCat)]$, where the $\bfA$-action comes from the right $\Vec_{\bfA}^\omega$-action on $\YCat$.
\item $\BQ (Y_\sigma)=[e_\sigma\otimes\YCat]$.
\end{enumerate}
\end{itemize}

Finally, let us recall an embedding $\BQ (Y)\hookrightarrow \Spr(\Orb)^{\dcell}$
of $W\times \fA$-modules quoted in Subsection \ref{SUBSECTION_Dodd}. It induces
an embedding $\Hom_W([\sigma],\BQ (Y))\hookrightarrow \Hom_W([\sigma],\Spr(\Orb))$
of $\fA$-modules.
The source module is $\BQ(Y_\sigma)$ by (B3), while the target module is
 $\BQ (\bA/H_\sigma)$ by Proposition \ref{Prop:Springer2}. Combining this with   Propositions \ref{Prop:Springer1},  we get
\begin{itemize}
\item[(B4)] There is an embedding $\BQ (Y_\sigma)\hookrightarrow \bA/H_\sigma$ of $\fA$-modules.
The image coincides with the sum of all $V$-isotypic components of $\bA/H_\sigma$, where $V$ is an $\bA$-module such that $V$
(or, equivalently,  the $W$-module corresponding to $(1,V)$
under the Lusztig parametrization) appears in $\BQ (Y)$.
\end{itemize}

\subsection{Preparation for the proof}
We use the notation introduced in the previous subsection.
Recall that $\bH^\lambda_\cell$ denotes the stabilizer of a point from $Y_\cell^\lambda$.
Our main goal is to prove that $\bH^\lambda_\cell$ coincides with $H_\cell$
(up to conjugacy).

\begin{Lem}\label{Lem:prep1}
(i) If $H_\cell=\bA$, then $\bH_\cell^\lambda=\bfA$.

(ii) If the two-sided cell $\dcell$ is non-exceptional (this excludes
precisely 3 cells in types $E_7,E_8$), then there is $\cell$
with $H_{\cell}=\bA$.

(iii) If $\bH_\cell^\lambda=\bfA$ for some compatible $\lambda,\cell$, then $\omega=0$.
\end{Lem}
\begin{proof}
(i) follows from (B4), and (ii) follows from the explicit
descriptions of the subgroups $H_\cell$ given in Section \ref{SECTION_cell}.
To prove (iii) let us recall that $\YCat^\lambda_\cell=\Coh(Y^\lambda_\cell)$
is an indecomposable right module over $\Vec_{\bfA}^\omega$, see (B1). From the description
of indecomposable $\Vec_{\bfA}^\omega$-modules in Subsection \ref{SUBSECTION_module_categories}
we see that $\omega|_{\bH^\lambda_\sigma}$ is trivial (in fact, it does not
matter whether we consider left or right modules). Since $\bH^\lambda_\sigma=\bfA$,
we are done.
\end{proof}

In particular, we see that for non-exceptional cells $\omega$ is 0. So
$\YCat^\Lambda$ is isomorphic to the right $\Vec_{\bfA}$-module corresponding
to some  collection $\bpsi$ of 2-cocycles.
In particular $\,^\Lambda\!\JCat^\Lambda_{\Orb}\cong \Coh^{\bfA,\bar{\psi}}(Y^\Lambda\times Y^{\Lambda})$,

Here is a technical claim that we are going to prove:
\begin{Thm}\label{theorem:main} Assume that the two-sided cell $\dcell$ is not exceptional.

(i) The quotient $\bfA$ of $\fA$ coincides with Lusztig's quotient $\bA$. Moreover, $\bH^\lambda_\cell=H_\cell$,
whenever $\lambda$ and $\cell$ are compatible.

(ii) There is $\bpsi_0 \in H^2(\bA,\K^\times )$ such that  $\bpsi_\cell$ is cohomologous to the restriction of $\bpsi_0$
to $H_\cell$;

(iii) The image of the embedding $\BQ (Y^\cell)\hookrightarrow \Spr(\Orb)$ is $\Spr(\Orb)^{\dcell}$.
\end{Thm}

\begin{Rem}\label{Rem:faithfulness}
We will see below that $\bfA$ acts faithfully on $Y$, i.e., only the unit element acts trivially.
\end{Rem}

\begin{Rem}\label{Rem:same_inv}
Although we do not know yet that $\bfA$ coincides with $\bA$, we can say from the beginning
that $|\bA|=|\bfA|$. Indeed, the categories $\,^{\Lambda}\!\JCat_{\Orb}^\Lambda$ and
$(\Vec_{\bfA}^\omega)^*_{\YCat^\Lambda}$ are indecomposable multi-fusion categories and for
such categories the {\em Frobenius-Perron dimensions} (see \cite[\S 8.2]{ENO}) of all
component categories (i.e., of fusion categories of the form $e\mathfrak{C}e\subset \mathfrak{C}$, where
$e$ is an indecomposable summand of unit) coincide, see
\cite[Corollary 8.14]{ENO}. For the category $(\Vec_{\bfA}^\omega)^*_{\YCat^\Lambda}$ this
common value equals $|\bfA|$ by {\em loc. cit.} For the category
$\,^{\Lambda}\!\JCat_{\Orb}^\Lambda$ this common value can be read of the character
tables of Grothendieck rings of component categories computed in \cite{Lusztig_subgroups}
and equals $|\bA|$ (cf. \cite[p. 225]{BFO1}). We deduce the desired equality since
$\,^{\Lambda}\!\JCat_{\Orb}^\Lambda=(\Vec_{\bfA}^\omega)^*_{\YCat^\Lambda}$ by (B1).

\end{Rem}

\begin{Rem}\label{Rem:Lusztig_conj}
We remark that Theorem \ref{theorem:main} gives an alternative proof of the Lusztig's conjecture
proved in \cite[Theorem 4]{BFO1}.
Also we want to remark that not only $^{\Lambda}\!\JCat^{\Lambda}_{\Orb}$
is isomorphic to $\Coh^{\bA}(Y^{\Lambda}\times Y^{\Lambda})$ but actually this equivalence is
realized by the embedding $^{\Lambda}\!\JCat^{\Lambda}_{\Orb}\hookrightarrow \Coh^{\fA}(Y^{\Lambda}\times Y^{\Lambda})$.
Let us explain what we mean by this. Consider a simple object $\M\in ^{\Lambda}\!\JCat^{\Lambda}_{\Orb}$
and a point $(x,y)\in Y^{\Lambda}\times Y^{\Lambda}$, where the fiber of $\M$ is nonzero.
It is not true that the fiber is a genuine representation of $\bA_{(x,y)}$, it is still
a projective $\bA_{(x,y)}$-module but the Schur multiplier is a coboundary, so we can
view the fiber as a $\bA_{(x,y)}$-module, say $V$. By construction, the embedding
$^{\Lambda}\!\JCat^{\Lambda}_{\Orb}\hookrightarrow \Coh^{\bA}(Y^{\Lambda}\times Y^{\Lambda})$
sends $\M$ to the simple equivariant sheaf corresponding to $(x,y,V)$.
 In particular, the multiplicity of $\M$
on $\Orb$ can be computed as follows. According to Theorem \ref{Thm_dagger}, this multiplicity
equals $\dim \M_{\dagger}$. But as a vector space, $\M_{\dagger}:=V^{\bigoplus |\bA/\bA_{(x,y)}|}\otimes \Hom_{\K}(N_x,N_y)$,
where $N_x,N_y$ are irreducible $\Walg$-modules corresponding to the points $(x,y)$.
So
\begin{equation}\label{eq:multiplicity}
\operatorname{mult}_{\Orb}(\M)=\frac{|\bA|}{|\bA_{(x,y)}|}\dim V\dim N_x\dim N_y.
\end{equation}
This formula will be of great importance in \cite{Goldie} and is one of the main reasons
why our classification business is important for the computation of the Goldie ranks.
\end{Rem}

Now let us establish a few more technical tools to be used in the proof of Theorem \ref{theorem:main}.

\begin{Lem}\label{Lem:prep2}
Let $\cell_0$ be a left cell with $H_{\cell_0}=\bA$. Then for any compatible $\lambda,\cell$ we have
$\#\Rep(H_\cell)=\#\Rep^{\bpsi_{\cell_0}-\bpsi_\cell^\lambda}(\bH^\lambda_\cell)$.
\end{Lem}
\begin{proof}
The left and right hand sides are (4) and (5) in (B2), respectively.
\end{proof}

\begin{Cor}\label{Cor:rep_coinc}
Suppose $\cell$ is not exceptional. Then $\#\Rep\bfA=\#\Rep\bA$.
\end{Cor}
\begin{proof}
Take a left cell ${\sigma_0}$ with $H_{\sigma_0}=\bA$. By Lemma \ref{Lem:prep1}, $\bH_{\sigma_0}=\bfA$.
Applying Lemma \ref{Lem:prep2} to $\sigma=\sigma_0$, we get the required equality.
\end{proof}

\begin{Lem}\label{Lem:prep3}
Suppose that $\bA$-modules $V_1,\ldots,V_m$ occur in $\BQ (Y)$. Let $\BQ(\bA/H_\cell)'$
denote the sum of all irreducible components of $\BQ (\bA/H_\cell)$ isomorphic to $V_1,\ldots,V_n$.
Then we have inclusion $\BQ (\bA/H_\cell)'\hookrightarrow \BQ (Y_\cell)\hookrightarrow \BQ (\bA/H_\cell)$.
\end{Lem}
\begin{proof}
This follows from (B4).
\end{proof}

\begin{Cor}\label{Cor:prep3}
If $\bA$ is abelian, we have an $\fA$-equivariant surjection $\bA/H_\cell\twoheadrightarrow \bfA/\bH_\cell$.
\end{Cor}

Now let us explain the general strategy of the proof. First, we note that, in virtue of Lemma \ref{Lem:prep3},
the claims (i) and (iii)  of the theorem for $\lambda=\rho$ are very closely related (in fact, equivalent
when $\bA$ is abelian). So, basically, we need to establish the existence of sufficiently many
irreducible constituents of $\BQ (Y)$. On the other hand, the only tool for us to get constituents of
$\BQ (Y)$ is to prove the equalities $\bH_\cell=H_\cell$. For some $H_\cell$ (roughly, for large ones)
this is doable by using Lemma \ref{Lem:prep2} and the second inclusion of Lemma \ref{Lem:prep3}.

When (i) is fully established, proving (ii) is not difficult (sometimes we
need to prove these two claims simultaneously, though). After establishing (i),(ii) for $\lambda=\rho$ we will
treat the case of general $\lambda$.

Now we give a proof of Theorem \ref{theorem:main} in the case when the cell $\dcell$ is exceptional.
It is proved in \cite[Theorem 1.1 and Remark 1.2 (iv)]{O} that in this case we have a tensor equivalence
$\,^\Lambda\!\JCat^\Lambda_{\Orb}\cong \Vec_{\ZZ/2\ZZ}^\omega \boxtimes \Coh(Y'\times Y')$ for
the nontrivial element $\omega \in H^3(\ZZ/2\ZZ ,\K^\times)\cong \ZZ/2\ZZ$ and some finite set $Y'$. It follows from
\cite[Proposition 2.3 and Example 2.1]{Orbi}
that the category $\Vec_{\ZZ/2\ZZ}^\omega \boxtimes \Coh(Y'\times Y')$
has a unique (up to equivalence of module categories)
indecomposable module category, namely $\Vec_{\ZZ/2\ZZ}^\omega \boxtimes \Coh(Y')$.
The assertions of Theorem \ref{theorem:main} follow
easily from this and statement (A2) in Section \ref{summary}. Notice that in this case
$\fA=\bfA=\ZZ/2\ZZ$ and $H_\cell$ is trivial for any left  cell $\cell \subset \dcell$.

\subsection{Proof for classical types}\label{SUBSECTION:proof_classical}
Recall, Subsections \ref{SUBSECTION_explicit_bc},\ref{SUBSECTION_explicit_d} that $\bA\cong \bF_2^m$.

In the proof we will need to use cells $\cell_0,\cell_1,\ldots,\cell_m,\cell_\varnothing$ with
the following properties: $H_{\cell_0}=\bA, H_{\cell_k}$ has index 2 in $\bA$, and $\bigcap_{k=1}^m H_{\cell_k}=\{1\}$,
and, finally, $H_{\cell_\varnothing}=\{1\}$. The existence of such cells for all classical types follows
from Propositions \ref{Prop:cell_existbc} (types B and C) and \ref{Prop:cell_existd} (type D).  In fact, $\bH_{\cell_\varnothing}=H_{\cell_\varnothing}$ implies $\bH^\lambda_\cell=H_\cell$
for all $\lambda,\cell$. But to prove
that $\bH_{\cell_\varnothing}=H_{\cell_\varnothing}$  we will need to check the coincidence of the subgroups
for $\cell_1,\ldots,\cell_m$.

The proof  will be divided into the following steps:
\begin{itemize}
\item[Step 1:] Prove $\bfA/\bH_{\cell_i}=\bA/H_{\cell_i}$ (the equality of
quotients of $\fA$) for $i=1,\ldots,m$.
\item[Step 2:] Establish the inclusion of ``sufficiently many'' simple $\bA$-modules into $\BQ (Y)$.
\item[Step 3:] Prove $\bfA=\bA$.
\item[Step 4:] Prove $\bH_{\cell_\varnothing}=H_{\cell_\varnothing}$ based on Step 2.
\item[Step 5:] Deduce $\bH_\cell^\lambda=H_\cell$ in general.
\item[Step 6:] Prove (ii) of the theorem.
\item[Step 7:] Prove (iii) of the theorem.
\end{itemize}
%

{\it Step 1.} Set $\cell:=\cell_0$  and $\cellb:=\cell_i$ for some $i$.
Corollary \ref{Cor:prep3} implies that either $\bH_\cellb=\bfA$ or $\bfA/\bH_\cellb=\bA/H_\cellb$.
By Lemma \ref{Lem:prep2}, $\#\Rep^{\bpsi^\cell-\bpsi^\cellb}(\bH_\cellb)=\#\Rep(H_\cellb)$.
The r.h.s. is $|\bA|/2$ or, equivalently, by Remark \ref{Rem:same_inv}, $|\bfA|/2$.

We remark that if $\bpsi$ is a non-trivial (=not coboundary) 2-cocycle on $\bH_\cellb$,  then there is no 1-dimensional representation in
$\Rep^{\bpsi} \bH_\cellb$.
The category $\Rep^{\bpsi} \bH_\cellb$ is the same as the category of modules over
the twisted group algebra $\K^{\bpsi} \bH_\cellb$ with $\dim \K^{\bpsi} \bH_\cellb=|\bH_\cellb|$.
It follows that $\sum_{V}(\dim V)^2=|\bH_\cellb|$, where the summation is taken over all
irreducible $\K^{\bpsi}\bH_{\cellb}$-modules. But the number of
simple representations in $\Rep^{\bpsi}\bH_\cellb$ is $|\bA|/2=|\bfA|/2$ by the previous paragraph.

From here we see that $\bH_\cellb\neq \bfA$ (otherwise the number of simples is $\#\Rep\bfA=\#\Rep\bA=|\bA|$
if $\bpsi$ is a coboundary -- the first equality follows from Corollary \ref{Cor:rep_coinc}--
and does not exceed $|\bfA|/4=|\bA|/4$ else). 
So the quotients $\bfA/\bH_\cellb$ and $\bA/H_\cellb$ coincide.

{\it Step 2.} Now apply (B4) to $\cell=\cell_0,\cellb=\cell_i$.
We get that the non-trivial $\bA/H_{\cell_i}$-module $V_i$
appears in the $\bA$-module $\BQ (Y)$.

{\it Step 3.} From the previous step we see that $V_1,\ldots,V_m$
appear in $\BQ (Y)$. But the common kernel
of those representations in $\fA$ coincides with the kernel of the projection $\fA\twoheadrightarrow
\bA$. On the other hand, the $\fA$-action on the right hand side factors through $\bfA$.
From here we get  $\bfA\twoheadrightarrow \bA$ and hence $\bfA=\bA$.

{\it Step 4.} Now apply Lemma \ref{Lem:prep3} to $\cell=\cell_{\varnothing}$
and the $\bA$-modules $V_1,\ldots,V_m$. The $\bfA=\bA$-module $\BQ(\bA/H_\cell)'$
is just $\bigoplus_{i=1}^m V_i$. It follows that the $\bA$-action on $\BQ(Y_\cell)$
is faithful and therefore $\bH_\cell=\{1\}$.


{\it Step 5.} Now let $\cell=\cell_\varnothing,\lambda=\rho$, and take arbitrary compatible $\mu,\cellb$.
Expression (4) in (B2) is $|\bA/H_\cellb|$, while (5) is $|\bfA/\bH^\mu_\cellb|$.
So (B2) implies $|H_\cellb|=|\bH^\mu_\cellb|$. For $\mu=\rho$ we deduce $H_\cellb=\bH_\cellb$
from Corollary \ref{Cor:prep3}. For general $\mu$ (A3) reads $H_\cellb=\bH_\cellb\subset \bH_\cellb^\mu$
and hence $H_\cellb=\bH_\cellb^\mu$. So (i) of Theorem \ref{theorem:main} is fully proved.

{\it Step 6.} Now set $\cell:=\cell_0,\lambda=\rho, \bpsi_0:=\bpsi_\cell$  and pick arbitrary
compatible  $\mu,\cellb$.
Using the equality of (4) and (5) in (B2) we see that $\#\Rep^{\bpsi_0-\bpsi_\mu^\cellb}H_\cellb=\#\Rep H_\cellb$.
Similarly to Step 1, this implies that $\bpsi_0-\bpsi_\mu^\cellb$
is cohomologous to 0.

{\it Step 7.} To get (iii) of Theorem \ref{theorem:main}, apply (B4) to $\cell=\cell_\varnothing$
and $\lambda=\rho$.

\subsection{Proof for exceptional types}
Now we are going to prove Theorem \ref{theorem:main} for non-exceptional cells in
exceptional types. We need to consider the cases $\bA=\ZZ/2\ZZ,S_3,S_4,S_5$
separately.

\subsubsection{$\bA=\ZZ/2\ZZ$}
The proof repeats that in the classical case (where we omit Steps 4 and 6 -- no 2nd cohomology
for the subgroups of $\bA$).

\subsubsection{$\bA=S_3$}
Recall that $H_\cell$ is one of the subgroups $\{1\},S_2,S_3$.
We have $\bfA=\bH_{\cell}^\lambda$ provided
$H_\cell=S_3$, thanks to Lemma \ref{Lem:prep1} and (A3). We also remark that (ii) follows
readily from the fact that the subgroups of $S_3$ has no 2nd cohomology.
Lemma \ref{Lem:prep2} implies that $\#\Rep \bH^\lambda_\cell=\#\Rep H_\cell$. Also
Lemma \ref{Lem:prep3} together with (A3) show that $\BQ (S_3/\bH_\lambda^\cell)\hookrightarrow \BQ (S_3/H_\cell)$
(as $\fA$-modules). From here it is easy to deduce that $|\bH^\lambda_\cell|=|H_\cell|$
and hence $\BQ (S_3/\bH^\lambda_\cell)=\BQ (S_3/H_\cell)$. The $\fA$-action on the l.h.s.
factors through $\bfA$. Since $\bA=S_3$ acts faithfully on $\BQ (S_3/S_2)$
we deduce from $|\bfA|=|\bA|$ that $\bfA=\bA$. Finally, we see that $H_\cell=\bH_\cell^\lambda$.
Assertion (iii) of Theorem \ref{theorem:main} follows from (B4) applied to $\cell$
with $H_\cell=\{1\}$ (outside $G_2$) and with $H_\cell=S_2$ (in $G_2$).

\subsubsection{$\bA=S_4$}
Recall that we have the following Lusztig subgroups $H^1=S_4,H^2=S_3,H^3=\operatorname{Dyh}_8,H^4=S_2,H^5=S_2\times S_2$.
Recall that $\bH_\cell^\lambda=H_\cell$ whenever $H_\cell=H^1$. We use the notation for
$\bA$-modules and for $W$-modules introduced in Subsection \ref{SUBSECTION_explicit_excep2}.

Thanks to Remark \ref{Rem:same_inv}, the group $\bfA$ has
24 elements.

The proof of Theorem \ref{theorem:main} is carried out in the following steps:
\begin{itemize}
\item[Step 1:] Prove that $V_{31}, V_{22}$ appear in $\BQ (Y)$.
\item[Step 2:] Deduce that $\bfA=\bA$.
\item[Step 3:] Deduce that $\bH_\cell=H_\cell$ if $H_\cell=S_3,\operatorname{Dyh}_8,S_2\times S_2$.
\item[Step 4:] Prove that $\bH_\cell=H_\cell$ whenever $H_\cell=S_2$.
\item[Step 5:] Deduce assertion (iii) of Theorem \ref{theorem:main}.
\item[Step 6:] Prove that $\bH_\cell^\lambda=H_\cell$ for all $\lambda$.
\item[Step 7:] Deduce assertion (ii).
\end{itemize}

{\it Step 1.} Let us show that both  $V_{22}$ and $V_{31}$ appear in $\BQ (Y)$.
Assume the converse.
Let $\cell,\cellb$ be  left cells with $H_\cell=S_4, H_\cellb=S_2\times S_2$.
We have $\BQ (S_4/S_2\times S_2)=V_4\oplus V_{31}\oplus V_{22}$
Applying (B4) to the left cell $\cellb$, we see that $\bH_\cellb=\bfA$  (neither $V_{31}$ nor
$V_{22}$ appear in $\BQ (Y)$)
or  $\BQ (\bfA/\bH_\cellb)= \BQ (S_4/S_3)$ ($V_{31}$ appears in $\BQ (Y)$ but $V_{22}$
does not) or $\BQ (\bfA/\bH_\cellb)=\BQ (S_4/\operatorname{Dyh}_8)$. Lemma \ref{Lem:prep2}
implies $\#\Rep^{\bpsi^\cell-\bpsi^\cellb}(\bH_\cellb)=\#\Rep(H_\cellb)=4$. But  the sum of squared dimensions of the simples in
$\Rep^{\bpsi^\cell-\bpsi^\cellb}(\bH_\cellb)$ must be equal to $|\bH_\cellb|$. Since $\bfA$ consists
of $24$ elements, we see that $|\bH_\cellb|=6$ or $8$ or $24$. But neither of these numbers
can be represented as the sum of 4 positive squares.
So  $V_{31},V_{22}\subset \BQ (Y)$.

{\it Step 2.} The equality $\bfA=\bA$ follows now from $|\bfA|=24$ and the observation that $S_4$ acts
faithfully on $V_{31}$.

{\it Step 3.} Let $\cell$ be as on Step 1, and $\cellb$ be one of the cells with $H_\cellb=H^2,H^3$ or $H^5$.
Thanks to Step 1 and Lemma \ref{Lem:prep3}, we have $\BQ (S_4/\bH_\cellb)=\BQ (S_4/H_\cellb)$. It is easy to see that this equality
actually implies $\bH_\cellb=H_\cellb$.

{\it Step 4.} Now take $\cellb$ with $H_\cellb=S_2$. We have $\BQ (S_4/S_2)=V_4\oplus V_{31}^{\oplus 2}\oplus V_{22}\oplus V_{2111}$.
Thanks to Lemma \ref{Lem:prep3} applied to $V_{31}$ and $V_{22}$, we see that $V_4\oplus V_{31}^{\oplus 2}\oplus V_{22}
\subset \BQ (S_4/\bH_\cellb)\subset V_4\oplus V_{31}^{\oplus 2}\oplus V_{22}\oplus V_{2111}$. The dimension of
$V_4\oplus V_{31}^{\oplus 2}\oplus V_{22}$ is $9$ and does not divide $24$. So
$\BQ (S_4/\bH_\cellb)= V_4\oplus V_{31}^{\oplus 2}\oplus V_{22}\oplus V_{2111}$. From here
one can deduce that $\bH_\cellb=S_2$.

{\it Step 5.} Assertion (iii) of the theorem follows from (B4) applied to $\cellb$ with $\bH_\cellb=S_2$,
compare with the proof for the classical types.

{\it Step 6.} Take $\cell$ with $H_\cell=S_2$ and arbitrary compatible $\mu,\cellb$.
Apply (B2) to that choice (with $\lambda=\rho$). From the coincidence of (4) and (5) we deduce
that the number of $S_2$-equivariant sheaves on $S_4/H_\cellb$ and on $S_4/\bH_\cellb^\mu$
coincide. Recall, see (A3), that $\bH_\cellb\subset \bH_\cellb^\mu$.

For an $S_2$-set $X$ the number $s(X)$ of simple $S_2$-equivariant
sheaves on $X$ is $2n_X+\frac{1}{2}m_X$, where $n_X$ (resp., $m_X$) is the number of $S_2$-fixed
(resp., non $S_2$-fixed) points. Since the transpositions generate $S_4$, we see that $m_X>0$
unless the $S_2$-action is trivial.

If $H_\cellb=S_3$, then $s(S_4/\bH_\cellb^\mu)=s(S_4/S_3)=5$ and $S_3\subset \bH_\cellb^\mu$.
This is only possible if $\bH_\cellb^\mu=S_3$.

If $H_\cellb=\operatorname{Dyh}_8$, then   $s(S_4/\bH_\cellb^\mu)=s(S_4/\operatorname{Dyh}_8)=3$ and $\operatorname{Dyh}_8\subset \bH_\cellb^\mu$.
This is only possible if $\bH_\cellb^\mu=\operatorname{Dyh}_8$.

If $H_\cellb=S_2\times S_2$, then $s(S_4/\bH_\cellb^\mu)=s(S_4/S_2\times S_2)=6$ and $S_2\times S_2\subset \bH_\cellb^\mu$.
This is only possible if $\bH_\cellb^\mu=S_2\times S_2$ because the only subgroups containing $S_2\times S_2$
are $\operatorname{Dyh}_8$ and $S_4$.

Finally, consider the case $H_\cellb=S_2$. Here $s(S_4/S_2)=9$. So for $X=S_4/\bH_\cellb^\mu$ we should
have $s(X)=2n_X+\frac{1}{2}m_X=9$. Also $n_X+m_X=|X|$  divides $12$. This is only possible
if $n_X=2, m_X=10$ and so $\bH_\cellb^\mu=S_2$.

{\it Step 7.} Applying Lemma \ref{Lem:prep2} to $\mu$ and $\cellb$ we see that
$\#\Rep^{\bpsi_\cell-\bpsi_\cellb^\mu}(H_\cellb)=\#\Rep(H_\cellb)$. Doing the sum of
squares analysis as above, we see that the category on the left hand side is forced
to have a $1$-dimensional representation. This implies that the 2-cocycle is actually
a coboundary, compare with Step 1 of the proof in the classical types.

\subsubsection{$\bA=S_5$} Again, our proof is in several steps, more or less following
the pattern of the preceding cases.

{\it Step 1.} Let us prove that, first, $\bfA=S_5$ (as quotients of $\fA$), second, $\bH_\cellb=H_\cellb$ for $H_\cellb=S_3\times S_2$
and, third, $V_{41},V_{32}$ appear in $\BQ (Y)$.

We start by showing that $\bfA\cong S_5$ as abstract groups. Indeed, $_{\cell}\!\J_{\cell}\cong \Rep(\bfA)\cong \Rep(S_5)$
by Lemma \ref{Lem:0.3pre},(ii). Now \cite{EtGe} implies that $S_5$ and $\bfA$ have to be isomorphic.

Let us notice that $\BQ (S_5/S_3\times S_2)=V_5\oplus V_{41}\oplus V_{32}$. By Lemma \ref{Lem:prep3},
we have an embedding $\BQ (\bfA/ \bH_\cellb)\hookrightarrow \BQ (S_5/S_3\times S_2)$
of $\fA$-modules. This gives us four possibilities for the module $\BQ(\bfA/\bH_\cellb)$
and so $H_\cellb$ is forced to have one of the following orders: $120, 24, 20,12$.
On the other hand,  Lemma \ref{Lem:prep2} gives $\#\Rep^{\bpsi_\cell-\bpsi_\cellb}(\bH_\cellb)=\#\Rep(S_3\times S_2)=6$.
According to the last table in \cite{BFO1}, this implies that $H_\cellb=S_3\times S_2$
(and then the 2-cocycle is a coboundary automatically) because the order of $H_\cellb$ is at least 12. 

To complete the proof of the claims in the beginning of the step it remains to show that
$\bfA=\bA$ as quotients of $\fA$. This follows from the observation that $S_5$
acts faithfully on both $V_{41}$ and $V_{32}$ that are constituents of $\BQ(S_5/S_3\times S_2)$.



{\it Step 2.} We have $\BQ (S_5/S_4)=V_5\oplus V_{41}$.  Applying Lemma \ref{Lem:prep3}
to $\cellb$ with $H_\cellb=S_4$ to the $S_5$-module $V_{41}$, we see that $\bH_\cellb=S_4$.
Also checking the appropriate table in \cite{BFO1}, we see that $\bpsi^\cell-\bpsi^\cellb$
is a coboundary.

{\it Step 3.} Let us show that $\bH_\cellb=H_\cellb$ if $H_\cellb=S_3$
and that $V_{311}$ is contained in $\BQ (Y)$. We have $\BQ (S_5/S_3)=V_5\oplus V_{41}^{\oplus 2}\oplus
V_{32}\oplus V_{311}$. Applying Lemma \ref{Lem:prep3} to our cell $\cellb$ and the irreducible
$S_5$-modules $V_{41},V_{32}$, we see that $V_5\oplus V_{41}^{\oplus 2}\oplus
V_{32}\subset \BQ (\bfA/\bH_\cellb)\subset  V_5\oplus V_{41}^{\oplus 2}\oplus V_{32}\oplus V_{311}$.
Since  $\dim \BQ (S_5/H_\cellb)$ divides $120$, we see that $V_{311}\subset \BQ (Y)$ and
$\BQ (S_5/\bH_\cellb)=\BQ (S_5/S_3)$. Since $V_{41}$ appears in $\BQ (S_5/\bH_\cellb)$ with multiplicity $2$, it follows
that the $\bH_\cellb$-fixed point space  in the reflection representation $V_{41}$ is 2-dimensional.
Together with $|\bH_\cellb|=6$ this implies $\bH_\cellb=S_3$.

{\it Step 4.} Now we are going to deal with $H_\cellb=\operatorname{Dyh}_8$. We have $\BQ (S_5/\operatorname{Dyh}_8)=
V_5\oplus V_{41}\oplus V_{32}\oplus V_{221}$. Similarly to the previous step, we see that $V_5\oplus V_{41}\oplus V_{32}
\subset \BQ (S_5/\bH_\cellb)\subset V_5\oplus V_{41}\oplus V_{32}\oplus V_{221}$.
Assume that $\BQ (S_5/H_\cellb)=V_5\oplus V_{41}\oplus V_{32}$.
Here $\bH_\cellb=S_3\times S_2$. But, again, $5=\# \Rep(\operatorname{Dyh}_8)$ does not coincide with the number of
simples in $\Rep^\psi(S_3\times S_2)$ that equals $3$ or $6$, see \cite{BFO1}. Applying Lemma \ref{Lem:prep2},
we see that one cannot have $\bH_\cellb=S_3\times S_2$. So  $\BQ (S_5/\bH_\cellb)= V_5\oplus V_{41}\oplus V_{32}\oplus V_{221}$. Since $\operatorname{Dyh}_8$ is the only
subgroup of order $8$ in $S_5$, we are done. Also we see that $V_{221}$ appears in $\BQ (Y)$.

{\it Step 5.}  Consider a left cell $\cellb$ with $H_\cellb=S_2\times S_2$.
We have $\BQ (S_5/S_2\times S_2)=V_5\oplus V_{41}^{\oplus 2}\oplus V_{311}\oplus V_{32}^{\oplus 2}\oplus V_{221}$.
As we have seen on the previous steps, all irreducible summands lie in $\BQ (Y)$ and so, thanks to Lemma \ref{Lem:prep3},
$\BQ (S_5/\bH_\cellb)=\BQ (S_5/S_2\times S_2)$. Since the space of $\bH_\cellb$-fixed vectors in $V_{41}$ is 2-dimensional
and $|\bH_\cellb|=4$, we see that $\bH_\cellb=S_2\times S_2$.

{\it Step 6.} Now we are going to consider the remaining subgroup $H_\cellb=S_2$. The multiplicity of
$V_{2111}$ in $\BQ (S_5/S_2)$ is $1$, also, by Lemma \ref{Lem:prep3}, we know that $\BQ (\bA/\bH_\cellb)=
\BQ (S_5/S_2)$ or $\BQ(\bA/\bH_\cellb)=(\BQ (S_5/S_2))/V_{2111}$.
However, the dimension of the last space, $56$, does not divide $120$ and so we have  $\BQ (\bA/\bH_\cellb)=\BQ (S_5/S_2)$.
As above, this implies $\bH_\cellb=S_2$ and completes the proof of $\bH_\cellb=H_\cellb$ for all $\cellb$.
Also we see that all irreducible constituents of $\Spr(\Orb)^{\dcell}$ appear in $\BQ (Y)$, whence (iii).

{\it Step 7.} Now we are going to verify $\bH_\cellb^\mu=\bH_\cellb$ for all compatible $\mu,\cellb$.
Recall, (A3), that $H_\cellb=\bH_\cellb\subset \bH_\cellb^\mu$. Similarly to the $S_4$-case we have (in the notation
of that proof) \begin{equation}\label{eq:E8_equal} 2m_{S_5/\bH_\cellb^\mu}+\frac{1}{2}n_{S_5/\bH_\cellb^\mu}=2m_{S_5/H_\cellb}+\frac{1}{2}n_{S_5/H_\cellb}.\end{equation}
The last equality shows that $|\bH_\cellb^\mu|/|H_\cellb|=4$ is only possible if $n_{S_5/\bH_\cellb^\mu}=0=m_{S_5/H_\cellb}$.
However the last equality means that $H_\cellb$ contains no transposition, which never happens.

Assume now that $|\bH_\cellb^\mu|/|H_\cellb|=3$, equivalently, the fibers of the natural projection
$S_5/H_\cellb\twoheadrightarrow S_5/\bH_\cellb^\mu$ consist of 3 elements. But in this case the fiber
of each $S_2$-fixed point again contains an $S_2$-fixed point. So $m_{S_5/H_\cellb}\geqslant m_{S_5/\bH_\cellb^\mu}$
and (\ref{eq:E8_equal}) cannot hold.

So it only remains to consider the case $|\bH_\cellb^\mu|=2|H_\cellb|$. This is only possible for $H_\cellb=S_2,S_2\times S_2, S_3$.
If all transpositions contained in $\bH_\cellb^\mu$ are also contained in $H_\cellb$, we have $m_{S_5/H_\cellb}\geqslant m_{S_5/\bH_\cellb^\mu}$,
which is impossible. This excludes the case $H_\cellb=S_2\times S_2$. So we have either $H_\cellb=S_2, \bH_\cellb^\mu=S_2\times S_2$
or $H_\cellb=S_3, \bH_\cellb^\mu=S_3\times S_2$. We have $s(S_5/S_2)=2\cdot 6+\frac{1}{2}(60-6)=39, s(S_5/S_2\times S_2)=2\cdot 6+\frac{1}{2}24=24$. Next, $s(S_5/S_3)=2\cdot 6+\frac{1}{2}(20-6)=19, s(S_5/S_3\times S_2)=2\cdot 2+\frac{1}{2}(10-2)=8$.
So we see that $|\bH_\cellb^\mu|\neq 2|H_\cellb|$ and the equality $\bH_\cellb^\mu=H_\cellb$ is proved.

{\it Step 8.} Finally, let us prove assertion (ii). This basically follows from the fact that $\#\Rep(H_\cellb)\neq \#\Rep^\varphi(H_\cellb)$
for a non-trivial class $\varphi$. The latter inequality can be checked using the last  table in \cite{BFO1}.

\subsection{Conjectures in the non-integral case}\label{SUBSECTION_nonint}
We concentrate on the case when the central character is regular. The general
case should be obtained from this one using the translations. We have two conjectures
generalizing Theorem \ref{Thm:very_main}.

Let $\lambda$ be a representative of the central character in $\h^*$
that is dominant in the sense that $\langle\lambda,\alpha^\vee\rangle\not\in \ZZ_{<0}$
for every positive root $\alpha$ (there may be several such $\lambda$ but we fix one).
Let $W_\lambda$ be the integral Weyl group of $\lambda$, i.e., the subgroup
of $W$ generated by all reflections $s_\alpha$, where $\langle \lambda,\alpha^\vee\rangle\in \ZZ$.
We assume that the group  $G$ is simply connected.

%


Let $W_a$ be the affine Weyl group, that is, the semi-direct product of $W$ with the root lattice. It is well
known that $W_a$ is a Coxeter group, so the theory of cells applies to it. A deep theorem of Lusztig
(\cite[Theorem 4.8]{cells4})
states that two-sided cells in $W_a$ are in bijection with nilpotent orbits. Let $\dcell^a_\Orb$
denote the two sided cell in $W_a$ corresponding to the orbit $\Orb$. We recall that each left cell
in $W_a$ contains a unique {\em distinguished involution}. Let $\D_\Orb$ be the set of distinguished
involutions contained in $\dcell^a_\Orb$; as explained above the set $\D_\Orb$
is in bijection with the set of left cells contained in $\dcell^a_\Orb$. Conjecture \cite[10.5]{cells4}
associates
to each left cell contained in $\dcell^a_\Orb$ a subgroup of finite index in $Q$ defined up
to conjugacy. We note that a weak form of this conjecture proved in \cite[Theorem 4]{BeO} is
sufficient to define this subgroup. Equivalently, for each $d\in \D_\Orb$ we have a subgroup
$H_d\subset A(e)$ defined up to conjugacy.

Recall that the group $W_\lambda$ is a {\em parahoric} subgroup, that is, the projection under $W_a\to W$
of a conjugate of a well defined standard parabolic subgroup $W_I$ of $W_a$. It follows from
\cite[Lemma 7.4]{cells4} and \cite[Theorem 3.10]{variety}
that the set $\Prim_{\Orb}(\U_\lambda)$ is in bijection with $W_I\cap \D_\Orb$ (equivalently, with
the set of left cells in $W_I\cap \dcell^a_\Orb$).

\begin{Conj}\label{Conj:non_int v2}
Assume that $\J \in \Prim_{\Orb}(\U_\lambda)$ correspond to $d\in W_I\cap \D_\Orb$.
The stabilizer of the $A(e)$-orbit in $\Irr_{fin}(\Walg)$ lying over $\J$
is the subgroup $H_d\subset A(e)$.
\end{Conj}

It is easy to see that Conjecture \ref{Conj:non_int v2} holds in the extreme cases of regular and
trivial nilpotent orbits; however it is not clear whether Conjecture \ref{Conj:non_int v2} is
compatible with Theorem \ref{Thm:very_main}.
%

Let us now explain our second  conjecture that basically  generalizes Theorem \ref{theorem:main} (iii).

 Let $E$ be the irreducible representation of $W$ which corresponds to the trivial local
system on the orbit $\Orb$ under the Springer correspondence.
We consider the special representations $E_1$ of $W_\lambda$ with the following properties,
cf. \cite[1.3]{spun},

(a) $E_1$ appears with nonzero multiplicity in $E|_{W_\lambda}$;

(b) the number $a_{E_1}$ (see \cite[Section 4.1]{orange}) equals $(\dim\g-\rank \g-\dim \Orb)/2$.

It follows from \cite[Theorem 3.10]{variety} that there is a natural bijection $\J \mapsto \cell_\J$ between
the set $\Prim_{\Orb}(\U_\lambda)$ and the set
of left cells in $W_\lambda$ such that the corresponding cell representation contains a special
representation satisfying (a) and (b) above.

Let $\Spr(\Orb)^\lambda$ be the maximal submodule of $\Spr(\Orb)|_{W_\lambda}$ with
irreducible constituents from all families of $W_\lambda$ which contain special representations
satisfying (a) and (b) above. Notice that if $\lambda$ is integral then $\Spr(\Orb)^\lambda=
\Spr(\Orb)^\dcell$.

\begin{Conj}\label{Conj:non_int v3}
Let $H_\J$ be the stabilizer of the $A(e)$-orbit in $\Irr_{fin}(\Walg)$ lying over
$\J \in \Prim_{\Orb}(\U_\lambda)$. Then there is an isomorphism of $A(e)-$modules
$\BQ (A(e)/H_\J)=\Hom_{W_\lambda}([\cell_\J],\Spr(\Orb))$ (this determines the subgroup $H_\J$
uniquely up to conjugacy). The Grothendieck group of the category of finite dimensional
$\Walg$-modules with central character represented by $\lambda$ is
isomorphic to $\Spr(\Orb)^\lambda$ as an $A(e)-$module. In particular,
the number of isomorphism classes of irreducible finite dimensional $\Walg$-modules
with regular central character represented by $\lambda$ equals $\dim \Spr(\Orb)^\lambda$.
\end{Conj}

Obviously, Conjecture \ref{Conj:non_int v3} is compatible with Theorem \ref{Thm:very_main};
however it is not clear whether Conjecture \ref{Conj:non_int v3} is
compatible with Conjecture \ref{Conj:non_int v2}.

Now let $\lambda$ be arbitrary dominant but not necessarily regular weight. Let $W_\lambda^0 \subset W_\lambda$ be the stabilizer
of $\lambda$ with respect to the dot action. We expect that the Grothendieck group of the category of
finite dimensional $\Walg$-modules with central character represented by $\lambda$ is
isomorphic to $(\Spr(\Orb)^\lambda)^{W_\lambda^0}$ as an $A(e)$-module.

We  finish with a less precise conjecture. Pick a  special representation $E_1$ of $W_\lambda$
as above. To this representation we can assign a two-sided cell $\dcell_1$ in $W_\lambda$. On the other hand,
recall an equivalence relation $\sim$ on the irreducible $\Walg_\lambda$-modules and also the induced
equivalence relation on $\Prim_{\Orb}(\U_\lambda)$. The equivalence classes in the latter are parameterized
by the two-sided cells of the form $\dcell_1$. Let $\Prim_{\dcell_1}(\U_\lambda)$ be the class
corresponding to $\dcell_1$ and $\Irr_{\dcell_1}(\Walg_\lambda)$ denote the preimage of
$\Prim_{\dcell_1}(\U_\lambda)$. As before, any two fibers of the map
$\Irr_{\dcell_1}(\Walg_\lambda)\rightarrow \Prim_{\dcell_1}(\U_\lambda)$ contain equivalent
points but not all points in   $\Irr_{\dcell_1}(\Walg_\lambda)$ need to be equivalent.
It is easy to see that any such equivalence relation is as follows: there is a uniquely
determined $A(e)$-orbit, say $A(e)/A_{\dcell_1}$, and an $A(e)$-equivariant surjection $
\Irr_{\dcell_1}(\Walg_\lambda)\rightarrow A(e)/A_{\dcell_1}$, such that two points are equivalent
if and only if their images coincide. It is an interesting question how to see this subgroup
$A_{\dcell_1}$ in the settings of any of the two conjectures above.

Now consider the Lusztig group $\bA_{\dcell_1}$ constructed for the two-sided cell $\dcell_1\subset W_\lambda$.
There should be a natural epimorphism $A_{\dcell_1}\twoheadrightarrow \bA_{\dcell_1}$ with the following property:
the stabilizer of the $A(e)$-orbit in $\Irr_{\dcell_1}(\Walg_\lambda)$  corresponding to a left cell
$\sigma_1\subset \dcell_1$ is the preimage of $H_{\sigma_1}$ under the epimorphism above.

\end{document}